\documentclass{amsart}
\usepackage{amsmath,amsthm}
\usepackage{amsfonts,amssymb}
\usepackage{accents}
\usepackage{enumerate}
\usepackage{accents,color}
\usepackage{graphicx}
\usepackage{wrapfig}
\newcommand{\lvt}{\left|\kern-1.35pt\left|\kern-1.3pt\left|}
\newcommand{\rvt}{\right|\kern-1.3pt\right|\kern-1.35pt\right|}

\makeatletter
\def\@tocline#1#2#3#4#5#6#7{\relax
  \ifnum #1>\c@tocdepth % then omit
  \else
    \par \addpenalty\@secpenalty\addvspace{#2}%
        \begingroup \hyphenpenalty\@M
    \@ifempty{#4}{%
      \@tempdima\csname r@tocindent\number#1\endcsname\relax
    }{%
      \@tempdima#4\relax
    }%
    \parindent\z@ \leftskip#3\relax \advance\leftskip\@tempdima\relax
    \rightskip\@pnumwidth plus4em \parfillskip-\@pnumwidth
    #5\leavevmode\hskip-\@tempdima #6\nobreak\relax
    \ifnum #1=0
    \hfil\hbox to\@pnumwidth{}
    \else
    \hfil\hbox to\@pnumwidth{\@tocpagenum{#7}}\fi
    \par
    \nobreak
    \endgroup
  \fi}
\makeatother
\newtheorem{thm}{Theorem}[section]
\newtheorem{cor}[thm]{Corollary}
\newtheorem{lem}[thm]{Lemma}
\newtheorem{prop}[thm]{Proposition}

\newtheorem{defn}[thm]{Definition}
 
\theoremstyle{remark}
\newtheorem{rem}{Remark}[section]

 \def\la{{\langle}}
 \def\ra{{\rangle}}
 \def\ve{{\varepsilon}}
\def\({\left(}
\def \){ \right)}
\def\[{\left[}
\def \]{ \right]}

 \def\d{\mathrm{d}}
 
 \def\sph{{\mathbb{S}^{d-1}}}

 \def\sE{{\mathsf E}}
 
 \def\sH{{\mathsf H}}
 \def\sK{{\mathsf K}}
 \def\sL{{\mathsf L}}
 
 \def\sP{{\mathsf P}}
 
 \def\sS{{\mathsf S}}
 
 \def\bs{{\mathsf b}}
 \def\sc{{\mathsf c}}
 \def\sd{{\mathsf d}}
 \def\sm{{\mathsf m}}
 \def\sw{{\mathsf w}}

 \def\fD{{\mathfrak D}}
 
 \def\a{{\alpha}}
 \def\b{{\beta}}
 \def\g{{\gamma}}
 \def\k{{\kappa}}
 \def\t{{\theta}}
 \def\l{{\lambda}}
 \def\o{{\omega}}
 \def\s{\sigma}
 \def\la{{\langle}}
 \def\ra{{\rangle}}
 \def\ve{{\varepsilon}}
 
 \def\bb{{\mathbf b}}
 \def\cb{{\mathbf c}}

 \def\kb{{\mathbf k}}

 \def\Eb{{\mathbf E}}

 \def\Hb{{\mathbf H}} 
 \def\Jb{{\mathbf J}}
 \def\Kb{{\mathbf K}}
 \def\Lb{{\mathbf L}}
 \def\Pb{{\mathbf P}}
 
 \def\Sb{{\mathbf S}}

 \def\CH{{\mathcal H}}

 \def\CP{{\mathcal P}}
 
 \def\CT{{\mathcal T}}
 
 \def\CV{{\mathcal V}}
 
 \def\CW{{\mathcal W}}
 \def\BB{{\mathbb B}}

 \def\NN{{\mathbb N}}

 \def\RR{{\mathbb R}}
 \def\SS{{\mathbb S}}
 
 \def\VV{{\mathbb V}}
 \def\XX{{\mathbb X}}
 
      \def\proj{\operatorname{proj}}

\def\lla{\langle{\kern-2.5pt}\langle}      
\def\rra{\rangle{\kern-2.5pt}\rangle}

\newcommand{\wh}{\widehat}

\def\f{\frac}

\graphicspath{{./}}

\begin{document}

\title[Approximation and localized polynomial frame]
{Approximation and localized polynomial frame on conic domains}%s and hyperboloids}

\author{Yuan~Xu}
\address{Department of Mathematics, University of Oregon, Eugene, 
OR 97403--1222, USA}
\email{yuan@uoregon.edu} 
\thanks{The author is partially supported by Simons Foundation Grant \#849676.}

\date{\today}  

\begin{abstract} 
Highly localized kernels constructed by orthogonal polynomials have been fundamental in the 
recent development of approximation and computational analysis on the unit sphere, unit ball, and several
other regular domains. In this work, we first study homogeneous spaces that are assumed to contain highly
localized kernels and establish a framework for approximation and localized tight frames in such spaces, 
which extends recent works on bounded regular domains. We then show that the framework is applicable 
to homogeneous spaces defined on  bounded conic domains, which consists of conic surfaces
and the solid domains bounded by such surfaces and hyperplanes. The highly localized kernels on conic 
domains require precise estimates that rely on recently discovered addition formulas for orthogonal 
polynomials with respect to special weight functions on each domain and an intrinsic distance that takes
into account the boundary of the domain, the latter is not comparable to the Euclidean distance at around 
the apex of the cone. 

The main results provide construction of semi-discrete localized tight frame in weighted $L^2$ norm and 
characterization of best approximation by polynomials on conic domains. The latter is achieved by using
a $K$-functional, defined via the differential operator that has orthogonal polynomials as eigenfunctions, 
as well as a modulus of smoothness defined via a multiplier operator that is equivalent to the $K$-functional. 
Several intermediate results are of interest in their own right, including the Marcinkiewicz-Zygmund inequalities,
positive cubature rules, Christoeffel functions, and Bernstein type inequalities. Moreover, although 
the highly localizable kernels hold only for special families of weight functions on each domain, many 
intermediate results are shown to hold for doubling weights defined via the intrinsic distance on the domain. 
\end{abstract}
 
\maketitle
\tableofcontents
 
\section{Introduction}
\setcounter{equation}{0}

In recent years highly localized kernels constructed via orthogonal polynomials have become 
important tools in approximation theory, both computational and theoretical harmonic analysis, 
and functional analysis. We start with an introduction on the background of this development
in the first subsection, describe our main results in the second subsection and state the organization of
the work in the third subsection. 

\subsection{Background} \label{sec:background}

Let $\Omega$ be a set in $\RR^d$, either an algebraic surface or a domain with non-empty interior,
and let $\varpi$ be a nonnegative weight function defined on $\Omega$, normalized with unit integral 
with respect to the Lebesgue measure $\d x$ on $\Omega$, so that the bilinear form 
\begin{equation*}
    \la f, g\ra_{\varpi} = \int_{\Omega} f(x) g(x) \varpi(x) \d x
\end{equation*}
is a well defined inner product on the space of polynomials restricted to $\Omega$. Let $\CV_n^d(\varpi)$ 
be the space of orthogonal polynomials of degree $n$ with respect to this inner product. Assume that 
$\varpi$ is regular so that the orthogonal decomposition 
$$
   L^2(\Omega, \varpi) = \bigoplus_{n=0}^\infty \CV_n^d(\varpi)
$$
holds. Let $P_n(\varpi; \cdot,\cdot)$ be the reproducing kernel of the space $\CV_n^d(\varpi)$. The projection 
operator $\proj_n: L^2(\Omega, \varpi)\mapsto \CV_n^d(\varpi)$ can be written as 
\begin{equation*}
  (\proj_n f)(x) =\int_\Omega P_n(\varpi; x,y)f(y)\varpi(y) \d y, \qquad f\in L^2(\Omega, \varpi).
\end{equation*}
If $\wh a$ is a cut-off function, defined as a compactly supported function in $C^\infty(\RR_+)$ and $\wh a(t) = 1$
for $t$ near zero, then our highly localized kernels are of the form 
\begin{equation*} 
L_n(\varpi;x,y) := \sum_{j=0}^\infty \wh a \Big(\frac{j}{n}\Big) P_j(\varpi; x,y).
\end{equation*}
For orthogonal systems on several regular domains, the kernel $L_n(\varpi;x,y)$ for a suitable $\wh a$ satisfies 
a {\it localization principle}, meaning that the kernel decays at rates faster than any inverse polynomial 
rate away from the main diagonal $y=x$ in $\Omega \times \Omega$ with respect  to an intrinsic distance 
in $\Omega$. 

To get a sense of the highly localized estimate, we specify to two cases. The first one is the unit sphere
$\sph$ of $\RR^d$ with $\varpi(x) \d x = \d \s$ is the surface measure, for which the kernel $L_n(\cdot,\cdot) =
L_n(\varpi; \cdot,\cdot)$ 
is highly localized in the sense that \cite{NPW1}, for every $\k > 0$,  
\begin{equation} \label{eq:decaySph}
   |L_n(x,y)| \le c_\k \frac{n^{d-1}}{(1+n \,\sd_\SS(x,y))^\k}, \qquad x,y \in \sph,
\end{equation}
where $\sd_\SS$ denotes the geodesic distance 
\begin{equation} \label{eq:distSph}
  \sd_\SS(x,y) = \arccos \la x,y\ra, \qquad x, y \in \sph. 
\end{equation}
For the unit ball, the classical weight function is $\varpi_\mu(x) = (1-\|x\|^2)^{\mu-\f12}$, $\mu > -\f12$. For
$\mu \ge 0$, the kernel $L_n(\varpi_\mu; \cdot,\cdot)$ is highly localized in the sense that \cite{PX2}
\begin{equation} \label{eq:decayBall}
  \left |L_n(\varpi_\mu;x,y)\right| \le c_\k \frac{n^d}{\sqrt{\varpi_\mu(n;x)}\sqrt{\varpi_\mu(n;y)}
  (1+n\, \sd_\BB(x,y))^\k}, \qquad x,y \in \BB^d,
\end{equation}
where $\varpi_\mu(n;x) = (1-\|x\|^2 + n^{-2})^\mu$ and $\sd_\BB$ is the distance on $\BB^d$ defined by 
\begin{equation} \label{eq:distBall}
   \sd_\BB(x,y) = \arccos \left(\la x, y\ra + \sqrt{1-\|x\|^2}\sqrt{1-\|y\|^2}\right).
\end{equation}
By choosing $\k$ large, the kernel decays away form $y =x$ by \eqref{eq:decaySph} and \eqref{eq:decayBall}
faster than any polynomial rate. Furthermore, under a technical restriction on the cut-off function, both estimates 
can be improved to sub-exponential rate \cite{IPX}. The proof of these estimates relies on closed-form formula 
of the reproducing kernel $P_n(\varpi; \cdot,\cdot)$ on the unit sphere and the unit ball; see the next subsection. 

Let us denote the integral operator with $L_n(\varpi)$ as its kernel by $L_n * f$; that is, 
\begin{equation} \label{def:Ln-operator}
   L_n * f (x): = \int_\Omega f(y)L_n(\varpi; x,y) \varpi(y) \d y. 
\end{equation}
If $\wh a$ has the support $[0,2]$ and satisfies $\wh a(t) =1$ for $0\le t \le 1$, then the operator $L_n*f$ 
is a polynomial of degree $2n$ and it reproduces polynomials of degree $n$. Furthermore, in many cases, 
it provides the near best polynomial approximation to $f \in L^p(\Omega, \varpi)$ in the sense that  
$$
   \|L_n * f -f \|_{L^p(\Omega, \varpi)} \le c \inf_{\deg g \le n} \|f - g \|_{L^p(\Omega, \varpi)}, \quad 1 \le p \le \infty.
$$
The rapid decay of the kernel makes it a powerful tool for studying polynomial approximation on $\Omega$. 
Moreover, the operator $L_n* f$ plays an important role in the characterization of best approximation, which is a 
central problem in approximation theory. For the unit sphere, for example, it is used in the classical result
that characterizes the error of best polynomial approximation by a modulus of smoothness defied via spherical 
means or by its equivalent $K$-functional defined via the Laplace-Beltrami operator on the sphere. The theory 
has a long history that spans from \cite{BBP,P} to \cite{Rus} with contributions from many researchers in 
between. It has also been extended, more recently, to weight approximation on the sphere, the unit ball, and 
the simplex in \cite{X05}. 

The highly localized kernels can also be used to construct localized frames via a semi-continuous 
Calder\'on type decomposition based on the decomposition,
$$
  f  = \sum_{j=0}^\infty L_{2^j} \ast L_{2^j} \ast f, \qquad  f\in L^2(\Omega, \varpi), 
$$
which holds if the cut-off function $\wh a$ satisfies $\sum_{k=0}^\infty |\wh a(2^{-k} t)|^2 = 1$. To define the 
frame elements, the integrals in the right-hand side of the above expansion need to be appropriately 
discretized by a cubature rule, established with the help of the highly localized kernels, which states that 
$$
   \int_\Omega f(x)\d x  = \sum_{\xi \in \Xi_j} \l_{\xi} f(\xi), \qquad f \in \Pi_{2^j}(\Omega),
$$
where $\l_{\xi} > 0$, $\Xi_j$ is a finite subset in $\Omega$, and $\Pi_n(\Omega)$ denotes the space of 
polynomials, restricted on $\Omega$, of degree at most $n$. The discretization leads to the tight frame 
$\{\psi_\xi\}_{\xi \in \Xi}$ of $L^2(\Omega, \varpi)$, where $\Xi$ is a discrete set of well separated, with 
respect to the intrinsic distance, points in $\Omega$, and the frame elements $\psi_\xi$, also called 
{\it needlets}, are of the form 
$$
\psi_\xi(x) = \sqrt{\l_{\xi}} L_{2^j} (x,\xi), \qquad \xi \in \Xi,
$$ 
which is highly localized with its center at $\xi$ since $\l_\xi$ can be quantized and $L_{2^j} (x,\xi)$ is
highly localized. The tight frame means that, for all $f \in L^2(\Omega, \varpi)$, 
$$
  f = \sum_{\xi \in \Xi} \la f,\psi_\xi\ra_{\varpi} \psi_\xi \quad \hbox{and} \quad 
   \int_\Omega |f(x)|^2 \varpi(x) \d x = \sum_{\xi \in \Xi} |\la f, \psi_\xi\ra|^2.
$$
Because of their rapid decay, needlets can be used for decomposition of spaces of functions and 
distributions in various settings, including $L^p$, Sobolev, Besov, and Triebel- Lizorkin spaces. 

The above narrative has been carried out in the setting of the regular domains. The localized kernels 
are most heavily studied and used when $\Omega$ is the unit sphere $\sph$ of $\RR^d$, for which 
our outline of the localized frames is initiated in \cite{NPW1, NPW2}. They are further explored in, for 
example, \cite{DaiX,IPX}, on the sphere, and also studied on the interval $[-1,1]$ with the Jacobi weight 
\cite{BD, KPX1, PX1}, the unit ball with the Gegenbauer weight \cite{KPX2,PX2}, and the simplex with 
the multivariable Jacobi weight \cite{IPX} as well as for $\RR_+^d$ with the product Laguerre weight 
and $\RR^d$ with the Hermite weight \cite{Dzub, Epp2, KPPX,PX3} and for product domains \cite{IPX2}. 
The subexponential decay of the kernel was established in \cite{IPX}, following earlier work in \cite{DzH}. 
The needlets on the sphere and the ball have found applications in computational harmonics analysis, 
mathematical physics and statistics; see, for example, \cite{BKMP1,BKMP2,KP1,KP2,LSWW,WLSW}
and references in them. 

\subsection{New contributions} \label{sec:new-results}
Our starting point is the recent study of orthogonal structures on conic domains of $\RR^{d+1}$ in 
\cite{X20a,X20b}, which shows that many properties of spherical harmonics and classical orthogonal 
polynomials on the unit ball have analogs on conic domains. Most importantly, there are families 
of orthogonal polynomials on conic domains that possess properties akin to the addition formula and 
the Laplace-Beltrami operator for the spherical harmonics, the former provides a closed-form formula 
for the reproducing kernels and the latter is a second order differential operator that has orthogonal 
polynomials as eigenfunctions. This paves the way for carrying out the narrative outlined in the previous 
subsection on conic domains, which has been a topic largely untouched hitherto. 
In the present paper, we will deal with two types of conic domains and they are standardized as: 
% There are four types of conic domains and they are standardized as:
\begin{enumerate}[\quad (1)]
 \item conic surface $\VV_0^{d+1}$ defined by 
$$
   \VV_0^{d+1}: = \left\{(x,t) \in \RR^{d+1}: \|x\| = t, \, 0 \le t \le 1, \, x\in \RR^d\right\};
$$
\item solid cone $\VV^{d+1}$ bounded by $\VV_0^{d+1}$ and the hyperplane $t =1$.
\iffalse
\item two-sheets hyperbolic surface ${}_\varrho \XX_0^{d+1}$ defined by 
$$
{}_\varrho \XX_0^{d+1}: = \left\{(x,t) \in \RR^{d+1}: \|x\|^2 = t^2-\varrho^2, \, 
         \varrho \le |t| \le 1+ \varrho, \, x\in \RR^d\right\}, 
$$
where $\varrho \ge 0$, which becomes the double conic surface $\XX_0^{d+1}$ when $\varrho =0$;
\item solid hyperboloid ${}_\varrho \XX^{d+1}$ bounded by ${}_\varrho \XX_0^{d+1}$ and the hyperplanes 
$t = \pm 1$, which becomes solid double cone $\XX^{d+1}$ when $\varrho =0$. 
\fi
\end{enumerate}
Two other conic domains, double hyperbolic surface and solid hyperboloid, will be dealt with elsewhere 
because of the page limitation. 
On each domain, a closed-form formula of the reproducing kernels is uncovered for a family of weight
functions, which allows us to study highly localized kernels and carries out the program described in 
Subsection \ref{sec:background}, and a differential operator is defined that has orthogonal polynomials as 
eigenfunctions, which leads to a $K$-functional that is used to characterize the best approximation by
polynomials on the domain. 

Instead of dealing with each case individually, we will develop a unified theory for all homogeneous 
spaces that contain highly localized kernels that we describe in the first two subsections below. 
The unified theory is applicable to conic domains, which is described in the third and the fourth subsections. 

\subsubsection{Localizable homogeneous space}  
A homogeneous space is a measure space $(\Omega, \d \mu, \sd)$, where $\d \mu$ is a nonnegative 
doubling measure with respect to the metric $\sd(\cdot,\cdot)$ on $\Omega$. When $\d \mu = \sw(x) \d x$,
where $\sw$ is a doubling weight function on $\Omega$, we denote the homogenous space by
$(\Omega, \sw, \sd)$. For example, the space $(\sph, \d \s, \sd_\SS)$ is a homogeneous space, so is 
$(\BB^d, \varpi_\mu, \sd_\BB)$ for $\mu > -\f12$. 

We call a homogeneous space $(\Omega, \varpi, \sd)$ {\it localizable} if the orthogonal polynomials with 
respect to $\varpi$ on $\Omega$ admit highly localized kernels, which requires a fast decay estimate 
on $|L_n(\varpi; x, y)|$ as well as on $|L_n(\varpi; x_1,y) -L_n(\varpi; x_2,y)|$; see Subsection 
\ref{sec:localizableHS} for precise definition. Our first main result is to show that the narrative on the 
localized tight frame in the previous subsection can be established in a localizable homogenous space. This 
provides a unified theory that is applicable to conical domains to be studied in later sections. Furthermore, 
it allows us to establish several intermediate results that are of interest in their own right for doubling 
weight functions, including Marcinkiewicz-Zygmund inequalities and positive cubature rules. 

The development generalizes and follows closely recent studies on regular domains, especially those
on the unit sphere and the unit ball. We follow the approach on the unit sphere in \cite{Dai1,DaiX}, which 
is formulated in the homogeneous space on the sphere. While some of the results for the unit sphere can 
be extended easily, others need to be dealt with more carefully. Amongst various reasons, we mention that, 
on the unit sphere, it is the Lebesgue measure $\d \s$ that admits the highly localized kernels. It is, 
however, not the case in general. For example, the highly localized kernels on the conic surface $\VV_0^{d+1}$ 
are established for the weight function $\sw_{-1,\g}(t)= t^{-1}(1-t)^\g$, which does not include the Lebesgue 
measure. While spherical caps of the same radius on $\sph$ have the same constant surface area, this 
is not the case on $\VV_0^{d+1}$, where we need to deal with conical caps defined via the intrinsic distance 
$\sd_{\VV_0}$, defined in \eqref{def:distV0} below, and their volumes are measured by integrals over the caps 
against the weight function $\sw_{-1,\g}$. 

\subsubsection{Best approximation by polynomials}
The characterization of best approximation is at the core of approximation theory (cf. \cite{DL}). The classical
approach on the unit sphere relies essentially on multiplier operators of the Fourier harmonic series and the 
convolution structure arising from the addition formula. We provide a general framework on localized 
homogeneous space by assuming that the reproducing kernel $P_n(\varpi; \cdot,\cdot)$ satisfies an addition 
formula of a specific form satisfied by all known cases, including those on conic domains, and that there is 
a second order differential operator that has orthogonal polynomials as eigenfunctions. The former allows 
us to introduce a convolution structure that can be used to define a modulus of smoothness via a multiplier 
operator, and the latter is used to define a $K$-functional. The two quantities are then shown to be 
equivalent and either one can be used to characterize the error of the best polynomial approximation on the 
domain. The framework generalizes the classical result on the unit sphere and is applicable to other regular 
domains satisfying our assumptions, which are met by all four types of conic domains.

\subsubsection{Conic surface $\VV_0^{d+1}$ and cone $\VV^{d+1}$} 
On the conic surface $\VV_0^{d+1}$, the orthogonal structure is studied in \cite{X20a} for the weight 
function $\sw_{\b,\g}(t) = t^\b (1-t)^\g$ for $\b > - d$ and $\g > -1$. For our study we need an intrinsic 
metric on the conic surface. In contrast to the unit sphere with its geodesic distance $\sd_\SS(\cdot,\cdot)$
that depend only on relative positions of the points, the conic surface $\VV_0^{d+1}$ has a boundary 
at $t =1$ and an apex point at $t=0$ that its metric needs to take into account. The distance that we 
will use is defined by 
\begin{equation}\label{def:distV0}
  \sd_{\VV_0}( (x,t), (y,s)) = \arccos \left[\sqrt{\frac{\la x ,y \ra + t s}{2}}  + \sqrt{1-t} \sqrt{1-s} \right].
\end{equation}
Using this distance function and the closed-form formula of the reproducing kernel proved in \cite{X20a}, 
we shall show that $\sw_{\b,\g}$ is a doubling weight and that the space $(\Omega, \sw_{-1,\g}, \sd_{\VV_0})$ 
is a localizable homogeneous space. These results are established through delicate estimates, partly 
because the distance $\sd_{\VV_0}(\cdot,\cdot)$ is not comparable to the Euclidean distance at around 
the apex of $\VV_0^{d+1}$. With the highly localized kernel, we can then carry out our program and 
establish positive cubature rules, and use them to establish the highly localized frame. Furthermore, 
using the second order differential operator that has orthogonal polynomials with respect to $\sw_{-1,\g}$ 
as eigenfunctions, which is an analog of the Laplace-Beltrami operator in the unit sphere, we can 
define a $K$-functional and use it to establish a characterization of the best approximation by polynomials 
in the weighted $L^p$ norm on $\VV_0^{d+1}$. 

Our study on the cone $\VV^{d+1}$ follows along the similar line. The distance function is defined by 
identifying $\VV^{d+1}$ with a subset of $\VV_0^{d+2}$. The orthogonal structure is studied in 
\cite{X20a} for the weight function $W_{\b,\g,\mu}(x,t) = t^\b(1-t)^\g(t^2-\|x\|^2)^{\mu-\f12}$, which has
a more involved close form formula for its reproducing kernels. While part of the estimate can be 
carried out for the kernel associated with $W_{0,\g,\mu}$, the localizable homogeneous space is 
established for $W_{0,\g,0}$, which nevertheless allows us to complete our program on the solid 
cone $\VV^{d+1}$.

\medskip
Several remarks are in order over the above description of our main results. First of all, analysis on conical 
domains have not been seriously studied in the literature as far as we are aware, and it does pose new 
phenomena and new challenges. In particular, while the framework for localizable 
homogeneous space is applicable to conic domains, the verification of assertions under which the framework 
holds possess considerable difficulty in each case.
%, especially in the cases of $\VV_0^{d+1}$ and $\VV^{d+1}$. 
Second, there are many intermediate results that are of interests. For example, to quantize the coefficients
in positive cubature rule, we need bounds on the Christoeffel functions 
\begin{align*} %\label{eq:Christoffe}
   \l_n(\sw;x) = \inf_{\substack{g(x) =1 \\ g \in \Pi_n(\Omega)}} \int_{\Omega} |g(x)|^2 \sw(x)  \d x,
\end{align*} 
which are of interest in their own right. See, for example, \cite{CD, DX, Kroo, KrooLub, Pr1, Pr2, X95} and their 
references for recent works on Christoffel functions in several variables. While the lower bound of the Christoeffel 
function on a homogenous space can be derived using highly localized kernels when $\sw$ admits such kernels, 
we will establish the upper bound for all doubling weight under the assumption that certain fast decaying polynomials
exist on $\Omega$. The latter requires the construction of such polynomials in each conic domain. Third, there 
have been several recent works that deal with multivariate approximation by polynomials or Marcinkiewicz-Zygmund 
inequalities and positive cubature rules on either polyhedra domain or $C^2$ domain \cite{DP2,DP3,T1,T2}. 
They do not cover conical domains and our approach is different and relies on a specific orthogonal structure
on the domain, akin to those on the unit sphere and the unit  ball.

Finally, we should mention the important book {\it Analysis on Symmetric Cones } \cite{FK} and the literature around it. 
While the topic works in the theory of Euclidean Jordan algebras and lies in a more abstract setting, our study requires
very specific structures of Fourier orthogonal series and is far less abstract. We have not been able to discern a 
connection between the two topics duo to our lack of background in Jordan algebras. It would be of great interest if 
a connection could be identified. 

\subsection{Organization and convention}
The paper is organized as follows. The localizable homogeneous space will be defined and studied 
in the next section, which consists of results discussed in Subsection 1.2.1. The best approximation by 
polynomials in the homogeneous space is discussed in the third section, which consists of the results 
described in Subsection 1.2.2. 
The results on the conic surface and solid cone will be discussed in Sections 4 and 5, respectively. 
%The four cases of conical domains will be discussed in the Sections 4--7, in the order 
%listed in Subsection \ref{sec:new-results}.  
Each section will contain several subsections and its organization will be described in the preamble 
of the section. 

Throughout this paper, we will denote by $L^p(\Omega, \sw)$ the weighted $L^p$ space with respect to
the weight function $\sw$ defined on the domain $\Omega$ for $1 \le p \le \infty$. Its norm will be denote 
by $\|\cdot\|_{p,\sw}$ for $1 \le p \le \infty$ with the understanding that the space is $C(\Omega)$ with 
the uniform norm when $p= \infty$. 

For conical domains, in an attempt to distinguish conic surfaces and solid conic bodies, we shall denote
the operator on the surface in sans serif font, such as $\sP_n$ and $\sL_n$, and the operator on the solid
domains in bold font, such as $\Pb_n$ and $\Lb_n$. 

Finally, we shall use $c$, $c'$, $c_1$, $c_2$ etc. to denote positive constants that depend on fixed parameters and 
their values may change from line to line. Furthermore, we shall write $A \sim B$ if $ c' A \le B \le c A$. 
 
%%%%%%%%%%%%
%%     section 2     %%    
%%%%%%%%%%%%
\section{Homogeneous spaces with highly localized kernels} \label{sec:HomogSpace}
\setcounter{equation}{0}

Much of recent work on regular domains, such as the sphere, the ball, and the regular simplex falls in 
the framework of homogeneous space. What distinguishes these regular domains is the existence of 
highly localized kernels constructed via orthogonal polynomials. From the prior work on these regular
domains and what we will prove for the conic domains latter sections, it becomes clear that the desired 
estimates for the highly localized kernels share a common formation. In this section, we work with 
homogeneous spaces that are assumed to contain highly localized kernels and carry out our analysis 
on such spaces. 

The first subsection contains basics on homogeneous space and orthogonal polynomials, and it contains 
three assertions that define highly localized kernels precisely. Three examples are given in the second 
subsection, which also serves as a review of the Jacobi polynomials, spherical harmonics, and orthogonal 
polynomials on the ball that will be used later. An auxiliary maximal function is introduced in the third subsection 
and, with the help of highly localized kernels for $\varpi$, is shown to be bounded in weighted $L^p$ space
with respect to a doubling weight. Using this maximal function, Marcinkiewicz-Zygmund inequalities are 
established in the fourth subsection for all doubling weights. The fifth subsection contains the fourth assertion 
that is used to obtain an upper bound for the Christoeffel function. The latter is needed in the sixth subsection, 
where the positive cubature rules are established. With these preparations, the localized tight frame 
is defined and studied in the seventh subsection.  

The development of this section follows and generalizes the study on the unit sphere. Some parts of
the work follow those on the sphere with little extra effort but by no means all. Since the geodesic 
distance on the sphere and the surface measure on the unit sphere will be replaced by a distance 
function and a weight function that offer no specific geometric information, we will provide complete 
proofs in the most part and will be brief only when the proof follows that on the unit sphere fairly straightforwardly. 

\subsection{Localizable homogeneous spaces}\label{sec:localizableHS}

\subsubsection{Homogeneous spaces}
A homogeneous space is a measure space $(\Omega,\mu, \sd)$ with a positive measure $\mu$ and a metric 
$\sd$ such that all open balls 
$$
B(x,r)=\{y \in \Omega: \sd(x,y) < r\},  
$$ 
are measurable and $\mu$ is a regular measure satisfying the doubling property
$$
   \mu (B(x,2r)) \le c \mu (B(x,r)), \qquad \forall x \in  \Omega, \quad \forall r >0,
$$ 
where $c$ is independent of $x$ and $r$. Such a measure $\mu$ is called a doubling measure. 

Let $\Omega$ be a domain in $\RR^d$. We assume that $\Omega$ is equipped with an intrinsic 
distance $\sd$, so that the Lebesuge measure $\sm$ on $\Omega$ is a doubling measure and,
for each $\ve > 0$, $\min_{x \in \Omega} \sm( B(x,\ve)) \ge c_\ve > 0$. Let $\sw$ be a nonnegative 
integrable function defined on $\Omega$. For a given set $E \subset \Omega$, we define 
$$
       \sw(E) = \int_{E} \sw(x) \d \sm(x).
$$
The weight function $\sw$ is called a doubling weight if its satisfies the doubling condition:  
there exists a constant $L > 0$ such that 
$$
   \sw(B(x,2 r)) \le L \sw(B(x,r)), \qquad \forall x \in \Omega, \quad r \in (0, r_0), 
$$
where $r_0$ is the largest positive number such that $B(x,r) \subset \Omega$. If $\sw$ is a 
doubling weight, then $\sd \mu = \sw(x) \d \sm$ is a doubling measure. We are particularly 
interested in the homogeneous space with its measure so defined, which we denote by 
$(\Omega, \sw, \sd)$.

For a doubling weight $\sw$ on $\Omega$ with respect to $\sd$, we denote by $L(\sw)$ the least constant 
$L$ for the doubling condition $\sw(B(x,2 r)) \le L \sw(B(x,r))$. Let $\a(\sw)$ be a positive number such that 
\begin{equation} \label{eq:alpha(w)}
 \sup_{B(x, r) \subset \Omega} \frac{ \sw(B(x,2^m r))}{ \sw(B(x, r))} \le c_{L(\sw)} 2^{\a(\sw)m}, \quad m=1,2,\ldots.
\end{equation}
Then $\a(\sw) \le \log_2 L(\sw)$, as can be seen by iteration. We call $L(\sw)$ the doubling constant and 
$\a(\sw)$ the doubling index. 

\begin{lem}\label{lem:doublingLem}
Let $\sw$ be a doubling weight on $\Omega$. 
\begin{enumerate}[ \quad (i)]
\item If $0 < r < t$ and $x \in \Omega$, then 
$$
   \sw(B(x,t)) \le c_{L(\sw)} \left( \frac{t}{r} \right)^{\a(\sw)} \sw(B(x,r)). 
$$
\item For $x,y \in \Omega$ and $n=1,2,\ldots$,
$$
   \sw(B(x, n^{-1})) \le c_{L(\sw)} (1+n \sd(x,y))^{\alpha(\sw)} \sw(B(y, n^{-1})). 
$$
\end{enumerate}
\end{lem}

\begin{proof}
Assume $2^{m-1} \le t/r \le 2^m$, (i) follows immediately from the definition of $\a(\sw)$. By the triangle 
inequality of the distance function, $\sw(B(x,n^{-1})) \le \sw(B(y, \sd(x,y)+n^{-1}))$, applying (i) to 
bound it by $\sw(B(y, n^{-1}))$ proves (ii).
\end{proof}

Many fundamental results that hold for the Euclidean space can be extended to homogeneous spaces
\cite{Stein}. For example, the Hardy-Littlewood maximal function is defined, for a locally integrable 
function $f$, by 
\begin{equation}\label{def:HL-MF} 
   M_\sw f(x) := \sup_{r> 0} \frac{1}{\sw(B(x,r))} \int_{B(x,r)} |f(y)| \sw(y) \d \sm, \qquad x \in  \Omega,
\end{equation} 
which possesses the usual properties of the maximal function \cite[p. 13]{Stein}. For example, 
for every $f \in L^p(\Omega, \sw)$,
\begin{equation}\label{eq:MaxFuncH}
    \| M_\sm f\|_{p,\sw} \le c_p \|f\|_{p,\sw}, \qquad 1 < p \le \infty,
\end{equation}
where $\|f\|_{p,\sw}$ denotes the $L^p$ norm with respect to the measure $\sw(x) \d \sm$. 

We study homogeneous spaces that contain highly localized kernels, to be defined below. Such 
spaces are known to hold for certain weight functions on several regular domains. We adopt the
convention of denoting the weight function that admits highly localized kernels by $\varpi$ and we also
assume that $\varpi$ is normalized by $\int_\Omega \varpi(x) \d \sm =1$. The kernels are defined via 
orthogonal polynomials with respect to the inner product $\la \cdot,\cdot\ra_\varpi$ defined in terms 
of $\varpi$ in \eqref{def:ipd}.    

\subsubsection{Orthogonal polynomials}
Let $\Pi(\Omega)$ denote the space of polynomials restricted to $\Omega$ and let $\Pi_n(\Omega)$ denote 
its subspace of polynomials of degree at most $n$. If the interior of $\Omega$ is open, then $\Pi(\Omega)$ 
contains all polynomials of degree $n$ in $d$ variables and we write $\Pi_n = \Pi_n(\Omega)$. If $\Omega$ 
is an algebraic surface, such as the sphere $\sph$, then $\Pi_n(\Omega)$ contains all polynomials restricted 
to the surface. 

Let $\varpi$ be a normalized weight function on $\Omega$. Then the bilinear form $\la \cdot,\cdot\ra_\varpi$
\begin{equation}\label{def:ipd}
    \la f, g\ra_{\varpi} := \int_{\Omega} f(x) g(x) \varpi(x) \d x
\end{equation}    
is a well defined inner product on $\Pi_n(\Omega)$.     
Let $\CV_n^d(\Omega,\varpi)$ be the space of orthogonal polynomials of degree $n$ with respect to this inner 
product. We are particularly interested in the case when $\Omega$ is either a quadratic surface, such as 
the unit sphere $\sph$ or the conical surface $\VV_0^{d}$, for which 
\begin{equation} \label{eq:dimVnS}
  \dim \CV_n^d(\Omega,\varpi) = \binom{n+d-2}{n} + \binom{n+d-3}{n-1}, \quad n = 1,2,3,\ldots, 
\end{equation}
where we assume $\binom{n}{k} =0$ if $k < 0$, or a solid domain bounded by a quadratic surface, 
such as the unit ball $\BB^d$ and the solid cone $\VV^d$, for which 
\begin{equation} \label{eq:dimVn}
   \dim \CV_n^d(\Omega,\varpi) = \binom{n+d-1}{n}, \quad n = 0, 1,2,\ldots ,
\end{equation}
The reproducing kernel of $\CV_n^d(\Omega,\varpi)$, 
denoted by $P_n(\varpi; \cdot,\cdot)$, is uniquely determined by
$$
    \int_\Omega f(y) P_n(\varpi; x, y) \varpi(y)\d \sm(y) = f(x), \qquad \forall f\in \CV_n^d(\Omega,\varpi).
$$
If $\{P_{\nu,n}: 1 \le \nu \le \dim \CV_n^d(\Omega,\varpi)\}$ is an orthogonal basis of $\CV_n^d(\Omega,\varpi)$, then 
\begin{equation}\label{eq:reprod-kernel}
P_n(\varpi; x, y) = \sum_{\nu=1}^{\dim\CV_n^d(\Omega,\varpi)} \frac{P_{\nu,n}(x) P_{\nu,n}(y)}{\la P_{\nu,n},P_{\nu,n} \ra_\varpi}. 
\end{equation}
The kernel plays an essential role in the study of the Fourier orthogonal series, since it is the kernel
of the orthogonal projection operator $\proj_n: L^2(\Omega, \varpi)\mapsto \CV_n^d(\Omega, \varpi)$; more
precisely, 
\begin{equation}\label{def:projPn}
  \proj_n(\varpi; f, x) = \int_\Omega f(y) P_n(\varpi; x,y) \varpi(y) \d \sm(y), \quad f\in L^2(\Omega, \varpi).
\end{equation}
The Fourier orthogonal series of $f \in L^2(\Omega,\varpi)$ is defined by 
\begin{equation}\label{def:Fourier}
  f  = \sum_{n=0}^\infty  \proj_n(\varpi; f) = \sum_{n=0}^\infty   \sum_{\nu=1}^{\dim\CV_n^d(\Omega,\varpi)}
   \wh f_{\nu,n} P_{\nu,n}, 
  \quad \wh f_{\nu,n} = \frac{\la f, P_{\nu,n}\ra_\varpi}{\la P_{\nu,n},P_{\nu,n} \ra_\varpi}. 
\end{equation}
For these definitions and orthogonal polynomials of several variables in general, we refer to \cite{DX}. 

Our highly localized kernels are defined by a sampling of the kernels of the Fourier series through a smooth
cut-off function, which we define first.  

\begin{defn}\label{def:cutoff}
A nonnegative function $\wh a \in C^\infty(\RR)$ is said to be admissible if it obeys either one of the conditions
\begin{enumerate}[         $(a)$]
\item $\mathrm{supp}\,  \wh a \subset [0, 2]$ and $\wh a(t) = 1$, $t\in [0, 1]$; or
\item $\mathrm{supp}\,  \wh a \subset [1/2, 2]$. 
\end{enumerate}
\end{defn} 

Let $\wh a$ be an admissible cut-off function. For $n = 0,1,2,\ldots$, we define the kernel function $L_n(\varpi)$ 
on $\Omega\times\Omega$ by
\begin{equation} \label{def:Ln-gen}
   L_n (\varpi; x,y) = \sum_{k=0}^{\infty} \wh a \left(\frac{k}{n}\right) P_n(\varpi; x,y), \qquad x,y \in \Omega,
\end{equation}
Since $\wh a$ is supported on $[0,2]$, this is a kernel of polynomials of degree at most $2n$ in either $x$ or $y$ 
variable. 

\subsubsection{Localizable homogeneous spaces}
The main purpose of this section is to study homogeneous spaces $(\Omega, \varpi, \sd)$ for which the
kernels $L_n(\varpi; \cdot,\cdot)$ are assumed to satisfy the following assertions. 

\begin{defn}\label{def:Assertions}
The kernels $L_n(\varpi; \cdot,\cdot)$, $n=1,2,\ldots$, are called highly localized if they satisfy the following 
assertions: 
\begin{enumerate}[\,  \bf 1]
\item[] \textbf{Assertion 1}. For $\k > 0$ and $x, y\in \Omega$, 
$$
   |L_n(\varpi; x ,y)| \le c_\k \frac{1} {\sqrt{\varpi\!\left(B(x,n^{-1})\right)} \sqrt{\varpi\!\left(B(y,n^{-1})\right)}
      \left(1+n \sd(x,y)\right)^\k}.
$$
\medskip\noindent
\item[] 
\textbf{Assertion 2}.  For $0 < \delta \le \delta_0$ with some $\delta_0<1$ and $x_1 \in B(x_2, \frac{\delta}{n})$, 
$$
   \left|L_n(\varpi; x_1,y) - L_n(\varpi; x_2,y)\right| \le c_\k \frac{n \sd(x_1,x_2)} 
        {\sqrt{\varpi\!\left(B(x_1,n^{-1})\right)} \sqrt{\varpi\!\left(B(x_2,n^{-1})\right)} \left(1+n \sd(x_2,y)\right)^\k}.
$$
\item[]
\textbf{Assertion 3}. For sufficient large $\k >0$, 
\begin{align*}
\int_{\Omega} \frac{ \varpi(y)}{\varpi\!\left(B(y,n^{-1})\right)
    \big(1 + n \sd(x,y) \big)^{\k}}    \d \sm(y) \le c.
\end{align*}
\end{enumerate}
\end{defn}

The third assertion affirms the sharpness of the first two assertions. It implies the following more general 
inequality.

\begin{lem}\label{lem:CorA3}
Let $\varpi$ be a doubling weight that satisfies Assertion 3 with $\k >0$. For $ 0< p < \infty$, let 
$\tau =\k - \f p 2 \a(\varpi) |1-\f{p}{2}| > 0$. Then, for $x\in \Omega$, 
\begin{equation}\label{eq:CorA3}
\int_{\Omega} \frac{ \varpi(y)}{\varpi\!\left(B\left(y,n^{-1}\right)\right)^{\f p 2} \big(1 + n \sd(x,y) \big)^{\tau}} \d \sm(y) \le 
   c \,  \varpi\!\left(B\left(x,n^{-1}\right)\right)^{1-\f p 2}.
\end{equation}
\end{lem}

\begin{proof}
Denote the left-hand side of \eqref{eq:CorA3} by $J_{p,\tau}$. When $p =2$, the inequality \eqref{eq:CorA3} 
is precisely the Assertion 3. The case $p \ne 2$ can be deduced from the case $d=2$. Indeed, by 
Lemma \ref{lem:doublingLem}, it is easy to see that,
\begin{equation}\label{eq:w(x)/w(y)}
 \big(1+ n \sd(x,y) \big)^{-\a(\varpi)} \le 
  \frac{  \varpi\!\left(B\left(y,n^{-1}\right)\right)}{ \varpi\!\left(B\left(x,n^{-1}\right)\right)}\le \big(1+ n \sd(x,y) \big)^{\a(\varpi)}.
\end{equation}
For $ p \ne 2$, using the right-hand side inequality if $0 < p < 2$ and the left-hand side inequality if $p >2$, 
we obtain
$$
  \varpi\!\left(B\left(y,n^{-1}\right)\right)^{\f{p}2} \ge \frac{ \varpi\!\left(B\left(y,n^{-1}\right)\right) 
   \varpi\!\left(B\left(x,n^{-1}\right)\right)^{\f{p}2 -1}}
         {\big(1+ n \sd(x, y) \big)^{\a(w) |1-\frac{p}2|}},
$$
which gives $J_{p,\tau} \le  \varpi\!\left(B\left(x,n^{-1}\right)\right)^{1-\f{p}2} J_{2,\k}$. Hence, \eqref{eq:CorA3} 
for $p \ne 2$ follows from the case $p =2$. The proof is completed. 
\end{proof}

\begin{cor}
For $0 < p < \infty$ and $x \in \Omega$,  the highly localized kernel satisfies
\begin{equation}\label{eq:Ln-bdd}
   \int_{\Omega} \left|L_n (\varpi;x,y) \right|^p \varpi(y) \d \sm(y)
       \le c \left[ \varpi\!\left(B\!\left(x,n^{-1}\right)\right)\right]^{1-p}. 
\end{equation}
\end{cor}

In particular, for $p =1$, it ensures that the Assertion 1 is sharp in the sense that the estimate guarantees 
the boundedness of $\left|L_n (\varpi;x,y) \right|$ and it ensures that Assertion 2 is sharp in the same sense. 

\begin{defn}
The homogeneous space $(\Omega, \varpi, \sd)$ is called localizable if $\varpi$ is a doubling weight that 
admits highly localized kernels. For convenience, we also say that the domain $\Omega$, or the weight
$\varpi$, admits a localizable homogeneous space without specifying $\varpi$, or $\Omega$. 
\end{defn}

Throughout the rest of this section, we assume that the weight function $\varpi$ admits a localizable 
homogeneous space on $\Omega$. 

\subsection{Examples of localizable homogeneous spaces}

In this subsection, we give three examples of localizable homogeneous spaces: the interval $[-1,1]$
with the Jacobi weight, the unit sphere $\sph$ with the surface measure, the unit ball $\BB^d$ with
its classical weight function. These cases have been thoroughly studied fairly recently in the 
literature, and they motivate our definition of the localizable homogeneous spaces. Moreover, their
corresponding orthogonal polynomials serve as the building blocks for orthogonal polynomials on
the conic domains. 

\subsubsection{The interval $[-1,1]$ with the Jacobi weight} 
For $\a, \b > -1$, the Jacobi weight function is defined by 
$$
      w_{\a,\b}(t):=(1-t)^\a(1+t)^\b, \qquad -1 < x <1. 
$$
Its normalization constant $c'_{\a,\b}$, defined by $c'_{\a,\b}  \int_{-1}^1 w_{\a,\b} (x)dx = 1$, is given by
\begin{equation}\label{eq:c_ab}
 c'_{\a,\b} = \frac{1}{2^{\a+\b+1}} c_{\a,\b} \quad\hbox{with} \quad 
   c_{\a,\b} := \frac{\Gamma(\a+\b+2)}{\Gamma(\a+1)\Gamma(\b+1)}.
\end{equation}
The intrinsic distance of the interval $[-1,1]$ is defined by 
\begin{equation}\label{eq:dist[-1,1]}  
\sd_{[-1,1]}(t,s) = \arccos \left(t s + \sqrt{1-t^2} \sqrt{1-s^2}\right),
\end{equation}
which is the projection of the distance $|\t-\phi|$ on the upper half of the unit circle if we set $t  =\cos \t$ 
and $s = \cos \phi$. The space $([-1,1], w_{\a,\b}, \sd_{[-1,1]})$ is a localizable homogeneous space if
$\a, \b \ge -\f12$. 

The orthogonal polynomials with respect to $w_{\a,\b}$ are the Jacobi polynomials $P_n^{(\a,\b)}$, which
are given by the hypergeometric function as \cite[(4.21.2)]{Sz}
$$
  P_n^{(\a,\b)}(x) = \frac{(\a+1)_n}{n!} {}_2F_1 \left (\begin{matrix} -n, n+\a+\b+1 \\ 
      \a+1 \end{matrix}; \frac{1-x}{2} \right).
$$
These polynomials satisfy the orthogonal relations
$$
c_{\a,\b}' \int_{-1}^1 P_n^{(\a,\b)}(x) P_m^{(\a,\b)}(x) w_{\a,\b}(x) \d x = h_n^{(\a,\b)} \delta_{m,n},
$$
where $h_n^{(\a,\b)}$ is the square of the $L^2$ norm that satisfies
$$
  h_n^{(\a,\b)} =  \frac{(\a+1)_n (\b+1)_n(\a+\b+n+1)}{n!(\a+\b+2)_n(\a+\b+2 n+1)}.
$$
The Jacobi polynomials are eigenfunctions of a second order differential operator: 
$$
    \left[(1-x) \partial^2 - (\a-\b+(\a+\b+2)x ) \partial \right] u = - n (n+\a + \b+1)u, 
$$
where $\partial f= f'$ and $u = P_n^{(\a,\b)}$ for $n =0,1,2,\ldots$. 
The highly localized kernel for the Jacobi polynomials will be denoted by 
\begin{equation}\label{def.L}
L_n^{(\a,\b)}(x,y)=\sum_{j=0}^\infty \wh a \Big(\frac{j}{n}\Big)
       \frac{P_j^{(\a,\b)}(x) P_j^{(\a,\b)}(y)} {h_j^{(\a,\b)}},
\end{equation}
where $\wh a$ is an admissible cutoff function. This kernel decays rapidly in terms of the distance
$\sd_{[-1,1]}(\cdot,\cdot)$. Let the ball $B(x,r)$ be defined with respect to this distance. Then the 
Jacobi weight is a doubling weight and satisfies, 
$$
   w_{\a, \b}(B(x, n^{-1})) \sim n^{-1} w_{\a,\b}(n; t),
$$
where $w_{\a,\b}(n; t)$ is defined by 
\begin{equation*} %\label{Jacobi-weight}
w_{\a,\b}(n; t) = (1-t + n^{-2})^{\a+\f12} (1+t + n^{-2})^{\b+\f12}.
\end{equation*}
The following estimate of the localized kernel can be found in \cite{PX1}.

\begin{thm}\label{thm:Jacobi-kernel}
Let $\alpha, \beta \ge -1/2$ and $0<\ve \le 1$. Then, for any $\k >0$ and $n \ge 1$,
\begin{equation*}%\label{eq:JacobiBd1}
|L_n^{(\a,\b)} (t,s)| \le \frac{c\,n}{\sqrt{w_{\a,\b}(n; t)}\sqrt{w_{\a,\b}(n; s)}}
     \left(1+n \sd_{[-1,1]}(t, s)\right)^{- \k},
\end{equation*}
\end{thm}

This verifies Assertion 1 of the Definition \ref{def:Assertions}. The Assertions 2 and 3 can be found 
in \cite{PX1, KPX1}, respectively. The proof of the estimate relies on the estimate for the special case
$$
L_n^{(\a,\b)}(t) = L_n^{(\a,\b)}(t,1), \qquad -1 \le t \le 1.
$$ 
If $s =1$ and $t= \cos \t$, then $\sd_{[-1,1]}(t,1) = \t \sim  \sin \theta /2 \sim \sqrt{1-t}$ and we obtain \cite{BD}, 
\begin{equation} \label{eq:Ln(t,1)}
\left|L_n^{(\a,\b)}(t) \right|  \le c\, \frac{n^{2 \a +2}}{\left(1+ n  \sqrt{1-t}\right)^\k}, \quad -1 \le t \le 1,
\end{equation}
which is used to establish the estimate in Theorem \ref{thm:Jacobi-kernel}. There is also a more 
general estimate stated under weaker assumption on the cut-off function \cite{BD} and
\cite[Theorem 2.6.7]{DaiX}, which will be useful in later development. 

\begin{thm}
Let $\ell$ be a positive integer and let $\eta$ be a function that satisfy, $\eta\in C^{3\ell-1}(\RR)$, 
$\mathrm{supp}\, \eta \subset [0,2]$ and $\eta^{(j)} (0) = 0$ for $j = 0,1,2,\ldots, 3 \ell-2$. Then, 
for $\a \ge \b \ge -\f12$, $t \in [-1,1]$ and $n\in \NN$, 
\begin{equation} \label{eq:DLn(t,1)}
\left| \frac{d^m}{dt^m} L_n^{(\a,\b)}(t) \right|  \le c_{\ell,m,\a}\left\|\eta^{(3\ell-1)}\right\|_\infty 
    \frac{n^{2 \a + 2m+2}}{(1+n\sqrt{1-t})^{\ell}}, \quad m=0,1,2,\ldots . 
\end{equation}
\end{thm}

The Jacobi polynomials with equal parameters are the Gegenbauer polynomials $C_n^\l$, which are of
particular interest for our study. These polynomials are orthogonal with respect to the the weight function 
$$
  w_\l(x) = (1-x^2)^{\l-\f12},   \quad -1< x<1, \quad \l > -\tfrac12.
$$
More precisely, the polynomials $C_n^\l$ satisfy the orthogonal relation
\begin{equation} \label{eq:GegenNorm}
  c_{\l} \int_{-1}^1 C_n^{\l}(x) C_m^{\l}(x) w_\l(x) \d x = h_n^{\l} \delta_{n,m}, 
\end{equation}
where the normalization constant $c_\l$ of $w_\l$ and the norm square $h_n^\l$ are given by
\begin{equation}\label{eq:c_l}
    c_{\l} = \frac{\Gamma(\l+1)}{\Gamma(\f12)\Gamma(\l+\f12)} \qquad \hbox{and} \qquad
      h_n^\l = \frac{\l}{n+\l}C_n^\l(1) = \frac{\l}{n+\l}\frac{(2\l)_n}{n!}.
\end{equation}
The polynomial $C_n^\l$ is a constant multiple of the Jacobi polynomial $P_n^{(\l-\f12,\l-\f12)}$ and
it is also related to the Jacobi polynomial by a quadratic transform \cite[(4.7.1)]{Sz}
\begin{equation} \label{eq:Jacobi-Gegen0}
  C_{2n}^\l (x) = \frac{(\l)_n}{(\f12)_n} P_n^{(\l-\f12,\l-\f12)}(2x^2-1).
\end{equation}

\subsubsection{The unit sphere $\sph$}
With the unit weight function, or the Lebesgue measure $\d \s$, the space $(\sph, \d\s, \d_\SS)$ is a 
localizable homogeneous space, where the metric $\d_\SS = \d_{\sph}$ is the geodesic distance defined 
in \eqref{eq:distSph}. The orthogonal polynomials are spherical harmonics. Let $\CP_n^d$ denote the 
space of homogeneous polynomials of degree $n$ in $d$ variables. A spherical harmonics $Y$ of 
degree $n$ is an element of $\CP_n^d$ that satisfies $\Delta Y =0$, where $\Delta$ is the Laplace 
operator of $\RR^d$. If $Y \in \CP_n^d$, then $Y(x) = \|x\|^n Y(x')$, $x' = x/\|x\| \in \sph$, so that $Y$
is determined by its restriction on the unit sphere. 

Let $\CH_n(\sph)$ denote the space of spherical harmonics of degree $n$. Its dimension $\dim \CH_n(\sph)$ 
is given by \eqref{eq:dimVnS}. Spherical harmonics of different degrees are orthogonal on the sphere. For 
$n \in \NN_0$ let $\{Y_\ell^n: 1 \le \ell \le \dim \CH_n(\sph)\}$ be an orthonormal basis of $\CH_n(\sph)$; then 
$$
   \frac{1}{\o_d} \int_\sph Y_\ell^n (\xi) Y_{\ell'}^m (\xi)\d\s(\xi) = \delta_{\ell,\ell'} \delta_{m,n},
$$
where $\o_d$ denotes the surface area $\o_d = 2 \pi^{\f{d}{2}}/\Gamma(\f{d}{2})$ of $\sph$. Let 
$\sP_n(\cdot,\cdot)$ denote the reproducing kernel of the space $\CH_n(\sph)$. Among many properties 
of the spherical harmonics (see, for example, \cite{DaiX, SW}), one characteristic property is a closed-formula 
for this kernel \cite[(1.2.3) and (1.2.7)]{DaiX},  
\begin{equation} \label{eq:additionF}
  \sP_n(x,y) = \sum_{\ell =1}^{\dim \CH_n(\sph)} Y_\ell^n (x) Y_\ell^n(y) = Z_n^{\f{d-2}{2}} (\la x,y\ra), \quad x, y \in \sph, 
\end{equation}
where $Z_n^\l$ is defined in terms of the Gegenbauer polynomial by 
\begin{equation} \label{eq:Zn}
    Z_n^\l(t) = \frac{n+\l}{\l} C_n^\l(t), \qquad \l = \frac{d-2}{2}. 
\end{equation}
Because of the second equal sign, the identity \eqref{eq:additionF} is often called the {\it addition formula} of 
spherical harmonics. The function in its right-hand side is often called zonal harmonic, since it depends only
on $\la x,y\ra$. This formula allows us to use \eqref{eq:Ln(t,1)} to derive the estimate \eqref{eq:decaySph},
which is Assertion 1, for the highly localized kernel
$$
   \sL_n(x,y) = \sum_{k=0}^\infty \wh a \left(\frac{k}{n}\right) \sP_k(x,y), \qquad x, y \in \sph.
$$
The other two assertions also hold; see, for example, \cite[Thm. 6.5]{DaiX} and \cite[Cor. 2.6.6]{DaiX}. 

Another characteristic property of the spherical harmonics is that they are eigenfunctions of the 
Laplace-Beltrami operator $\Delta_0$ on the sphere, which is the restriction of the Laplace operator 
$\Delta$ on the unit sphere; see, for example,  \cite[Section 1.4]{DaiX} for its explicit expression. More 
precisely, it is known \cite[(1.4.9)]{DaiX} that 
\begin{equation} \label{eq:sph-harmonics}
     \Delta_0 Y = -n(n+d-2) Y, \qquad Y \in \CH_n^d,
\end{equation}
so that the eigenvalues depend only on $n$. This relation, as well as the operator $\Delta_0$, plays an 
important role in the approximation theory on the unit sphere. 

\subsubsection{The unit ball $\BB^d$ with classical weight function}
The classical weight function on the unit ball $\BB^d$ of $\RR^d$ is defined by 
\begin{equation}\label{eq:weightB}
  W_\mu(x) = (1-\|x\|)^{\mu-\f12}, \qquad x\in \BB^d, \quad \mu > -\tfrac12.
\end{equation}
Its normalization constant is $b_\mu^\BB = \Gamma(\mu+\f{d+1}{2}) /(\pi^{\f d 2}\Gamma(\mu+\f12))$.
For the ball $B(x, r)$ defined via the distance function $\sd_\BB$ in \eqref{eq:distBall} on 
the unit ball \cite[Lemma 5.1]{PX2}, 
$$
      W_\mu ( B(x,r)) \sim r^d \left(\sqrt{1-\|x\|^2} + r \right)^{2\mu}, \qquad 0 < r \le 1,
$$
so that $W_\mu$ is a doubling weight. The space $(\BB^d, W_\mu, \sd_\BB)$ is a localizable 
homogeneous space for $\mu \ge 0$, where $\sd_\BB (\cdot,\cdot)$ is the distance defined in 
\eqref{eq:distBall} on the unit ball. 

Let $\CV_n^d(\BB^d,W_\mu)$ be the space of orthogonal polynomials of degree $n$ with respect to 
$W_\mu$. An orthogonal basis of $\CV_n^d(\BB^d, W_\mu)$ can be given explicitly in terms of 
the Jacobi polynomials and spherical harmonics. For $ 0 \le m \le n/2$, let $\{Y_\ell^{n-2m}: 1 \le \ell \le
 \dim \CH_{n-2m}(\sph)\}$ be an orthonormal basis of $\CH_{n-2m}^d$. Define \cite[(5.2.4)]{DX}
\begin{equation}\label{eq:basisBd}
  P_{\ell, m}^n (x) = P_m^{(\mu-\f12, n-2m+\f{d-2}{2})} \left(2\|x\|^2-1\right) Y_{\ell}^{n-2m}(x).
\end{equation}
Then $\{P_{\ell,m}^n: 0 \le m \le n/2, 1 \le \ell \le \dim \CH_{n-2m}(\sph)\}$ is an orthogonal basis of 
$\CV_n^d(W_\mu)$.  Let $\Pb_n(W_\mu; \cdot, \cdot)$ denote the reproducing kernel of the 
space $\CV_n(\BB^d, W_\mu)$. This kernel satisfies an analog of the addition formula \cite{X99}: 
for $x,y \in \BB^d$,
\begin{align} \label{eq:additionBall}
      \Pb_n(W_\mu;x,y) = c_{\mu-\f12} 
         \int_{-1}^1 Z_n^{\mu+\f{d-1}{2}} & \Big(\la x, y \ra + u \sqrt{1-\|x\|^2}\sqrt{1-\|y\|^2} \Big) \\
           &  \times (1-u^2)^{\mu-1}\d u, \notag
\end{align}
where $\mu > 0$ and it holds for $\mu =0$ under the limit
\begin{equation}\label{eq:limitInt}
   \lim_{\mu \to 0+}  c_{\mu-\f12} \int_{-1}^1 f(t) (1-t^2)^{\mu-1} \d u = \frac{f(1) + f(-1)}{2}. 
\end{equation}
This is an analog of the addition formula for the unit ball and it allows us to use \eqref{eq:Ln(t,1)} to 
derive the estimate \eqref{eq:decayBall}, which is Assertion 1, for the highly localized kernel now 
denoted by
$$
   \Lb_n(W_\mu; x,y) = \sum_{k=0}^\infty \wh a \left(\frac{k}{n}\right) \Pb_k(W_\mu; x,y).
$$
All three assertions were proved in \cite{PX2}. For the first one, see also  \cite[Thm. 11.5.3]{DaiX}.
 
The orthogonal polynomials with respect to $W_\mu$ on the unit ball are eigenfunctions of a second 
order differential operator \cite[(5.2.3)]{DX}:
\begin{equation}\label{eq:diffBall}
  \left( \Delta  - \la x,\nabla \ra^2  - (2\mu+d-1) \la x ,\nabla \ra \right)u = - n(n+2\mu+ d-1) u
\end{equation}
for all $u \in \CV_n(\BB^d, W_\mu)$. The differential operator in the left-hand side is the analog of
the Laplace-Beltrami operator on the sphere.

\subsection{An auxiliary maximal function}

Some of the polynomial inequalities in the latter sections will be established for norms defined via
a doubling weight. The study of approximation and polynomial inequality with doubling weights 
was pioneered in \cite{MT} for the interval $[-1,1]$ and developed subsequently in \cite{Dai1} for
the unit sphere $\sph$. The main tool for the latter is an auxiliary maximal function. We define its 
analog on the domain $\Omega$ that admits a localizable homogeneous space.  

\begin{defn}\label{defn:maxf*}
For $\b > 0$, $n \in \NN$ and $f \in C(\Omega)$, a maximal function $f_{\b,n}^\ast$ is defined by 
$$
f_{\b,n}^\ast (x) := \max_{y \in \Omega} |f(y)| \big (1+ n \sd(x,y) \big)^{-\b}, \quad x \in \Omega. 
$$
\end{defn}

The development below follows closely \cite[Sect. 5.2 and 5.3]{DaiX} but with a somewhat more 
streamlined proof.  

We clearly have $|f(x)| \le f_{\b,n}^\ast(x)$ for all $x \in \Omega$. The maximal function $f_{\b,n}^\ast$ 
has an upper bound in the Hardy-Littlewood maximum function $M_\mu$ defined in \eqref{def:HL-MF}. 
To see this, we first prove a lemma. For $f\in C(\Omega)$ and $r>0$, we define 
\begin{equation*}
{\rm osc } (f) (x, r) := \sup_{x_1,x_2 \in B(x, r)} |f(x_1)-f(x_2)|, \qquad   x\in\Omega.
\end{equation*}

\begin{lem}\label{lem:osc(f)}
If $f\in \Pi_n(\Omega)$ and $0 < \delta \le 1$, then for any $\b>0$,
$$
{\rm osc } (f) \left(x, \tfrac{\delta}{n}\right) \leq c_\b \delta f_{\b , n}^\ast (x), \qquad  x\in\Omega,
$$
where the constant $c_\b$ depends only on $d$ and $\b$ when $\b$ is large.
\end{lem}

\begin{proof}
We use the highly localized kernel $L_n$ with an admissible cut-off function of type (a). For $\delta >0$ 
and $x,y \in \Omega$, set
\begin{equation*}%\label{3-4-ch3}
A_{n,\delta} (x,y) := \sup_{z \in B (x,\f\delta n)} \left| L_n (\varpi; x,y) - L_n(\varpi; z,y)\right |.
\end{equation*}
Let $L_n(\varpi)*f$ be defined as in \eqref{def:Ln-operator}. Then $L_n(\varpi)*f = f$ for $f \in \Pi_n(\Omega)$.
If $x_1,x_2 \in B(x, r)$, then by the triangle inequality 
$$
    \left| L_n (\varpi; x_1,y) - L_n(\varpi; x_2,y)\right | \le 2 A_{n,\delta}(x,y). 
$$
Hence, it follows by  \eqref{def:Ln-operator} that 
\begin{align*}
  {\rm osc} (f) \left(x, \tfrac \delta n \right) & \leq  2 \int_{\Omega} |f(y)| 
           A_{n,\delta}(x,y) \varpi(y)  \d \sm(y)\\ 
  & \leq  2 f_{\b, n}^\ast (x)   \int_{\Omega} (1+n \sd(x,y))^{\b} A_{n,\delta} (x,y)\,  \d \sm(y)\\
  & \leq c_\b \delta f_{\b, n}^\ast (x),
\end{align*}
where the last step follows from Assertion 1 and Assertion 3. 
\end{proof}

\begin{thm}\label{thm-3-1-ch3}
Let $\sw$ be a doubling weight. For $f\in \Pi_n(\Omega)$, $\b>0$ and setting $\g = \a(\sw)/\b$, then
\begin{equation}\label{3-3-ch3}
    f_{\b,n}^\ast (x) \leq c_{\b,L(\sw)} \left( M_\sw \big(|f|^\g \big) (x) \right)^{1/\g}, \qquad x \in \Omega.
\end{equation}
\end{thm}

\begin{proof}
We apply Lemma \ref{lem:osc(f)} with $c_\b \delta = \f14$ to obtain, by $(a+b)^\g \le 2^\g (a^\g+b^\g)$, 
\begin{align*}
 \frac{|f(y)|^{\g} }{( 1+ n \sd(x, y))^{\a(\sw)}}  \leq 
       2^{\g} \frac{\min_{ z\in B(y,\f \delta n)} |f(z)|^{\g} }{( 1+ n \sd(x, y))^{\a(\sw)}} + 
          \frac{1}{2^\g} \left( f_{\b, n}^\ast(x) \right)^{\gamma},
\end{align*}
Taking maximum over $x \in \Omega$, the left-hand side becomes $\big( f_{\b, n}^\ast(x) \big)^{\gamma}$
using $\a(\sw) = \b \g$, so that the inequality can be rearranged to give 
\begin{equation} \label{eq:f_bn-bound}
    \left( f_{\b, n}^\ast(x) \right)^{\gamma} \le  c ( 1+ n \sd(x, y))^{-\a(\sw)}\min_{ z\in B( y,\f \delta n)}  |f(z)|^{\gamma},
\end{equation}
where $c = 2^\g /(1 - 2^{-\g})$. We now estimate $\min_{ z\in B( y,\f \delta n)}|f(z)|^{\gamma}$, following the 
proof in the case of the unit sphere  \cite[Thm. 5.2.2]{DaiX}. 

Let $\t = \max\{ \frac{\delta}{n}, \sd(x,y)\}$. Then $B(y, \f\delta n)\subset B(x, 2\t)\subset B(y,3\t)$. By the
doubling condition for $\sw$ and (ii) of Lemma \ref{lem:doublingLem}, 
\begin{align*}
   \sw \big(B (y,\tfrac \delta n)\big) \ge  \frac{1}{c_{L(\sw)}} \( \f { 3\t n} \delta \)^{-\a(\sw)}
       \sw(B (y, 3 \t) )  \ge  \frac{1}{c_{L(\sw)}} \( \f {3\t n} \delta \)^{-\a(\sw)}  \sw(B(x, 2 \t)), 
\end{align*}
which implies, together with $n \t \le 1+n\sd (x,y)$, the estimate
\begin{align*}
 \min_{z\in B(y,\f \delta n)} |f(z)|^\g  
 & \le \frac{1}{\sw(B(y, \f \delta n))} \int_{B(y,\f \delta n)} | f(z)|^{ \g} \sw(z)\, \d \sm(z) \\
  & \le  c_{L(\sw)}   \( \f { 3\t n} \delta \)^{\a(\sw)}  \frac{1}{\sw(B(x, 2 \t))}
        \int_{B (x,2\t)} |f(z)|^{\gamma} \sw(z)\,  \d \sm(z)\\
   & \leq c_{L(\sw)}  \( \f 3\delta \)^{\a(\sw)}(1+n \sd(x,y))^{\a(\sw)} M_{\sw} \left(|f|^{\gamma}\right)(x).
\end{align*}
Together with \eqref{eq:f_bn-bound}, this completes the proof.  
\end{proof}

Let $\|\cdot\|_{p, \sw}$ denote the norm of $L^p(\Omega, \sw)$ and we adopt this notation for $0 < p <1$,
even though it is no longer a norm. The $L^p$ boundedness of $f_{\b,n}^\ast$ follows from that of 
$M_{\mu}(|f|^\g)$ for $\g p > 1$ or $\b > \a(\sw)/p$. 

\begin{cor}\label{cor:fbn-bound}
If $ 0< p\leq \infty$, $ f\in\Pi_n(\Omega)$ and $\b > \a(\sw)/p$, then
$$
  \|f\|_{p, \sw} \leq \|f_{\b,n}^\ast\|_{p,\sw} \leq c  \|f\|_{p,\sw},
$$
where $c$ depends also on $L(\sw)$ and $\b$ when $\b$ is either large or close to $\a(\sw)/p$.
\end{cor}

\subsection{Marcinkiewicz-Zygmund inequality}
These inequalities are between the $L^p$ norm and a discrete $L^p$ norm defined with respect to 
a well-distributed set of points. We start with the definition of well-separated sets of points. 
 
\begin{defn}\label{defn:separated-pts}
Let $\Xi$ be a discrete set in $\Omega$. 
\begin{enumerate} [  \quad (a)]
\item A finite collection of subsets $\{S_z: z \in \Xi\}$ is called a partition of $\Omega$ if $S_z^\circ\cap 
S_y^\circ  = \emptyset$ when $z \ne y$ and $\Omega = \bigcup_{z \in \Xi} S_z$. 
\item Let $\ve>0$. A discrete subset $\Xi$ of $\Omega$ is called $\ve$-separated if $\sd(x,y) \ge\ve$
for every two distinct points $x, y \in \Xi$. 
\item $\Xi$ is called maximal if there is a constant $c_d > 1$ such that 
\begin{equation}\label{eq:def-pts2}
  1 \le  \sum_{z\in \Xi} \chi_{B(z, \ve)}(x) \le c_d, \qquad \forall x \in \Omega,
\end{equation}
where $\chi_E$ denotes the characteristic function of the set $E$.
\end{enumerate}
\end{defn} 

Evidently \eqref{eq:def-pts2} implies $\Omega = \bigcup_{z \in \Xi} B(z,\ve)$. Moreover, it implies 
that the cardinality $\# \Xi$ of $\Xi$ satisfies 
$$
      c_1 \frac{\sm(\Omega)}{\max_{z \in \Xi}  \sm(B(z,\ve))} \le \# \Xi \le c_2 
         \frac{\sm(\Omega)}{\min_{z \in \Xi}  \sm(B(z,\ve))}.
$$

For the unit sphere, all balls of the same radius have the same volume; that is, $\s (B(x,r)) = \s(B(0,1)) r^{d-1}$
for all $x \in \sph$, where $\s$ is the surface measure on $\sph$. This allows us to deduce \eqref{eq:def-pts2} 
from $\Omega = \bigcup_{x \in \Xi} B(x,\ve)$ and the cardinality of the maximal $\ve$-separated set 
$\Xi_\SS$ on the sphere $\sph$ satisfies $c_d' \ve^{-d+1} \le \# \Xi_\SS \le c_d \ve^{-d+1}$; see, 
for example, \cite[p. 114]{DaiX}. 

As we shall show in the next section that the volume of the ball $B(x,r)$ on the conic surface depends on 
the position of $x$. In particular, we no longer have $\varpi(B(x,\ve)) \sim \varpi(B(y,\ve))$ for all 
$x, y$ in the domain in general. 

\iffalse
\begin{prop} \label{prop:pointsSS}
Let $\Xi_{\Omega}$ be a maximal $\ve$ separated on $\Omega$ and let $d = \dim \Omega$. 
Then 
$$
c_d' \ve^{-d+1} \le \# \Xi_{\Omega} \le c_d \ve^{-d+1}, 
$$
where $\# \Xi$ denote the cardinality of $\Xi$. 
\end{prop}
\fi

\begin{thm}\label{thm:osc}
Let $\sw$ be a doubling weight on $\Omega$. Let $\ve = \f \delta n$ for $n = 1,2, \ldots$ and $\delta \in (0,1)$.
If $\Xi$ is a maximal $\ve$-separated subset of $\Omega$, then for all $f\in\Pi_n(\Omega)$ and  
$0 < p < \infty$,
\begin{equation}\label{eq:osc}
    \bigg( \sum_{y \in \Xi} \left|{\rm osc} (f)\left(y, \tfrac \delta n\right)\right|^p
       \sw\!\left (B(y,\tfrac \delta n)\right)\bigg)^{\f1p}  \le c_p \delta \|f\|_{p, \sw},
\end{equation}
where  $c_p$ depends on $p$, when $p$ is close to $0$, and on $d$ and $L(\sw)$.
\end{thm}

\begin{proof}
If $ y \in B(x, \frac{ \delta}{n})$, then $f_{2 \a(\sw)/p, n}^\ast (y) \sim f_{2 \a(\sw)/p, n}^\ast (x)$. 
Hence, it follows from Lemma \ref{lem:osc(f)} that % and \eqref{eq:def-pts1} that
\begin{align*}
\sum_{y \in \Xi} \left|{\rm osc} (f)\left(y, \tfrac \delta n\right)\right|^p  \sw\!\left (B(y,\tfrac \delta n)\right)
 % &  \le  c \sum_{y \in \Xi} \left|{\rm osc} (f)\left(y, \tfrac \delta n\right)\right|^p \varpi\!\left (R(y,\tfrac \delta n)\right)\\
  &  \le c (c_p \delta)^p \sum_{y  \in \Xi} \int_{B(y, \frac \delta n)} 
         \left( f_{2\a(\sw)/p, n}^\ast (x)\right)^p \sw(x) \,  \d \sm(x) \\
       & \leq ( c_p \delta)^p \int_{\Omega}  \left( f_{2\a(\sw)/p,n}^\ast (x)\right )^p \sw(x)\,  \d \sm(x). 
\end{align*}  
The last integral is bounded by $\|f\|_{p, \sw}^p$ by Corollary \ref{cor:fbn-bound}. 
\end{proof}

We now prove the Marcinkiewicz-Zygmund inequality for polynomials.

\begin{thm} \label{thm:MZinequality}
Let $\sw$ be a doubling weight on $\Omega$. Let  $\Xi$ be a maximal $\f \delta n$-separated 
subset of $\Omega$ and let $\delta > 0$ and $\delta \le 1$. 
\begin{enumerate}[$(i)$]
\item For all $0<p<\infty$ and $f\in\Pi_m(\Omega)$ with $n \le m \le c n$,
\begin{equation}\label{eq:MZ1}
  \Bigg(\sum_{z \in \Xi} \Big( \max_{x\in B(z, \f \delta n)} |f(x)|^p \Big)
     \sw\!\left(B(z, \tfrac \delta n) \right)\Bigg)^{\f1 p} \leq c_{\sw,p} \|f\|_{p,\sw},
\end{equation}
where $c_{\sw,p}$ depends on $L(\sw)$ and on $p$ when $p$ is close to $0$.
\item For $0 < r < 1$, there is a $\delta_r > 0$ such that for $\delta \le \delta_r$, $r \le p < \infty$ and 
$f \in \Pi_n(\Omega)$,  
\begin{align}\label{eq:MZ2}
  \|f\|_{p,\sw} \le c_{\sw,r} \bigg(\sum_{z \in\Xi}
       \bigg(\min_{x\in B\bigl(z, \tfrac{\delta}n\bigr)} |f(x)|^p\bigg)
          \sw\bigl(B(z, \tfrac \delta n)\bigr)\Bigg)^{\f1 p},
\end{align}
where $c_{\sw,r}$ depends only on $L(\sw)$ and on $r$ when $r$ is close to $0$. Furthermore,
\begin{equation}\label{eq:MZ3}
    \|f\|_\infty \le c \max_{z\in \Xi} |f(z)|. 
\end{equation}
\end{enumerate}
\end{thm}

\begin{proof}
For every $y \in\Xi$, choose $z_y \in B(y,\f \delta n)$ such that $|f(z_y)| = \sup_{x \in B(y, \f \delta n)} |f(x)|$.
For $f \in \Pi_m(\Omega)$, we have $|f(z_y)| \le f_{\b,m}^\ast(z_y) \le c f_{\b,n}^\ast(y)$. Hence, %by \eqref{eq:def-pts1}, 
\begin{align*}
  \sum_{z \in \Xi} \Big( \max_{x\in B(z, \f \delta n)} |f(x)|^p \Big) \sw\! \left(B(z, \tfrac \delta n) \right)
   & \le c  \sum_{z \in \Xi} \big(f_{\b,n}^\ast(y)\big)^p \sw\! \left(B(z, \tfrac \delta n) \right) \\
   & \le c  \int_\Omega \big( f_{\b,n}^\ast(x)\big)^p \sw(x)  \d \sm (x),
\end{align*}
from which (i) follows from the bound in Corollary \ref{cor:fbn-bound}. We now prove (ii). Since $\Xi$ is 
maximal $\f \delta n$-separated, %and $R(y,\f \delta n) \subset B(y,\f \delta n)$, 
\begin{align*}
 \|f\|_{p, w}^p & \le \sum_{z \in \Xi}  \int_{B(z,\f \delta n)} |f(x)|^p \sw(x)  \d \sm (x) \\
  & \le 2^p  \sum_{z \in \Xi} |\mathrm{osc}(f)\left(z, \tfrac{\delta}{n}\right)|^p \sw\left(B\big(z,\tfrac \delta n\big)\right)
      + 2^p  \sum_{z \in \Xi} \bigg(\min_{x\in B\bigl(z, \tfrac{\delta}n\bigr)} |f(x)|^p\bigg)\sw\left(B\big(z,\tfrac \delta n\big)\right).
\end{align*}
The first term in the right-hand side is bounded by $( 2 c_r \delta)^p \|f\|_{p,w}^p$ by Lemma \ref{lem:osc(f)}. 
Hence, choosing $\delta_r = 1/(4 c_r)$ so that $(2 c_r \delta)^p \le 2^{-p}$ for $\delta \le \delta_r$, we conclude
that 
$$
(1-  2^{-p}) \|f\|_{p,\sw}^p \le  2^p \sum_{z \in \Xi} \bigg(\min_{x\in B\bigl(z, \tfrac{\delta}n\big)} |f(x)|^p\bigg)
   \sw\big(B(z,\tfrac \delta n)\big).
$$
Taking the power of $1/p$ proves \eqref{eq:MZ2}. In particular, since the constant in \eqref{eq:MZ2} is independent 
of $p$, it readily implies \eqref{eq:MZ3}. This completes the proof. 
\end{proof}

For the unit sphere without weight and the unit ball with the classical weight, the MZ inequality was established 
in \cite{BD, MNW, NPW1,PX2}. For the unit sphere with the doubling weight, it was established in \cite{Dai1}.

\subsection{Christoffel function for doubling weight} \label{subsect:CFunction}

Let $\sw$ be a doubling weight on $\Omega$. The Christoffel function $\l_n(\sw;\cdot)$ is defined by 
\begin{align}\label{eq:ChristoffelF}
   \l_n(\sw;x): = \inf_{\substack{g(x) =1 \\ g \in \Pi_n(\Omega)}} \int_{\Omega} |g(x)|^2 \sw(x)  \d \sm(x).
\end{align} 
It is known that the Christoffel function is closely related to the kernel $K_n(\sw;\cdot,\cdot)$ of the partial
sum operator 
$$
   K_n(\sw;x,y) = \sum_{k=0}^n P_k(\sw; x,y), \qquad x, y \in \Omega.
$$
More precisely, it is known \cite[Theorem 3.6.6.]{DX} that, for $n=0,1,2,\ldots$,  
\begin{align}\label{eq:ChristoffelF2} 
   \l_n(\sw;x) = \frac{1}{K_n(\sw; x,x)}, \qquad x \in \Omega. 
\end{align}    

The Christoffel function encodes essential information on weighted approximation by polynomials on 
the domain $\Omega$. It will also be useful in the study of cubature rules in the next subsection, for 
which we need an upper bound of $\l_n(\sw; x)$. We establish this bound for all doubling weights on 
the localizable homogeneous space under the assumption that there exist fast decaying polynomials on 
$\Omega$. More precisely, we make the following assertion. 

\medskip
{\bf Assertion 4}.  {\it Let $\Omega$ be compact. For each $x \in \Omega$, there is a nonnegative 
polynomial $T_x$ of degree at most $n$ that satisfies 
\begin{enumerate}[   (1)]
\item $T_x(x) =1$, $T_x(y) \ge \delta > 0$ for $y \in B(x,\f 1 n)$ for some $\delta$ independent of $n$,
and, for each $\g > 1$,  
$$
     0 \le  T_x(y) \le c_\g (1+ n \sd(x,y))^{-\g}, \qquad y \in \Omega; 
$$
\item there is a polynomial $q_n$ such that $q_n(x) T_x(y)$ is a polynomial of degree at most $r n$,
for some positive integer $r$, in $x$-variable and $c_1 \le q_n(x) \le c_2$ for $x \in \Omega$ for some
positive numbers $c_1$ and $c_2$. 
\end{enumerate}
}

\medskip

The fast decaying polynomials of one variable are studied in \cite{KT}, see also \cite{ST}, which leads 
to fast decaying polynomials on the unit sphere by the addition formula \eqref{eq:additionF}. For the 
unit ball, such polynomials are constructed in \cite{PX2}. We use these polynomials to derive an upper 
bound for the Christoffel function. First, we need a lemma. 

\begin{lem}\label{lem:Ass4Q}
Assume Assertion 4. Let $\a> 0$ be a positive number and let $x \in \Omega$ be fixed. For a doubling 
weight $\sw$ on $\Omega$, there is a polynomial $Q_x$ of degree $n$ such that, for all $y \in \Omega$, 
\begin{equation} \label{eq:Ass4Q}
    c_1 (1+n \sd(x,y))^{\a} \sw(B(y,\tfrac{1}{n}))\le Q_x(y) \le c_2 (1+n \sd(x,y))^{\a} \sw(B(y,\tfrac{1}{n})), 
\end{equation}
where $c_1$ and $c_2$ are positive constant independent of $n$, $x$ and $y$. 
\end{lem}

\begin{proof}
Let $T_x$ and $q$ be polynomials in Assertion 4 of degree at most $\lfloor \frac{n}{2}\rfloor$. 
For fixed $x$, we define $Q_x$ by 
\begin{equation*} 
   Q_x(y) = q_n(y) \int_\Omega  T_y(v) (1+n \sd(x,v))^{\a} \sw(v)\d \sm(v), \qquad y \in \Omega. 
\end{equation*}
Since $q_n(y)T_y(v)$ is a polynomial of degree at most $n$ in $y$, it is easy to see that 
$Q_x$ is a polynomial of degree at most $n$. Using Assertion 4 with $\g > \a + 1$ and using
the triangle inequality $1+n \sd(x,v)\le  (1+n \sd(x,y))(1+n \sd(y,v))$, we obtain
\begin{align*}
   Q_x(y) \, & \le  c (1+n \sd(x,y))^\a \int_\Omega \frac{1}{(1+n \sd(y,v))^{\g-\a}} \sw(v) \d \sm(v)\\
       & \le c  (1+n \sd(x,y))^\a \sw\big(B (y,\tfrac 1 n)\big), 
\end{align*}
since, by the doubling property of $\sm$ and choosing $\g$ sufficiently large, 
\begin{align} \label{eq:int-dm}
  \int_\Omega \frac{ \sw(v)\d \sm(v)}{(1+n \sd(y,v))^{\g-\a}} \, & \le 
   \sum_{k=1}^\infty \int_{B(y,\frac{2^k} n)\setminus B(y,\frac{2^{k-1}} n)} \frac{ \sw(v)\d \sm(v)}{(1+n \sd(y,v))^{\g-\a}} \\
       & \le c  \sum_{k=1}^\infty \frac{1}{(1+ 2^{k-1})^{\g-\a} } \sw\!\left( B\big (y,\tfrac{2^{k}} n\big) \right)\notag \\
       & \le c\, \sw\! \left( B\big (y, \tfrac1 n \big) \right) \sum_{k=1}^\infty \frac{ L(\sw)^{k-1} }{(1+ 2^{k-1})^{\g-\a} } \notag\\
       & \le c\, \sw\! \left( B\big (y, \tfrac1 n \big) \right). \notag
 \end{align}
This establishes the upper bound of $Q_x$. In the other direction, we use $q_n(y) \ge c_1$, 
$ T_y(v) \ge \delta > 0$ for $v \in B(y,\f1n)$ and $ T_y(v) \ge 0$ to obtain
\begin{align*}
  Q_x(y) \, & \ge  (1+n \sd(x,y))^\a \int_{\Omega}  T_y(v) (1+n \sd(y,v))^{- \a}  \sw(v) \d \sm(v) \\
     &  \ge c\delta (1+n \sd(x,y))^\a \int_{B(y,\f1n)} (1+n \sd(y,v))^{- \a}  \sw(v) \d \sm(v) \\
     &  \ge c \delta (1+n \sd(x,y))^\a  \sw\! \left( B\big (y, \tfrac1 n \big) \right),
\end{align*}
since the last integral is trivially bounded below by $c \sw\! \left( B\big (y, \tfrac1 n \big) \right)$. 
This completes the proof. 
\end{proof}

\begin{prop}\label{prop:ChristF1}
Assume Assertion 4 holds. Let $\sw$ be a doubling weight on the domain $\Omega$. Then
$$
    \l_n(\sw; x) \le c\, \sw\! \left(B(x,\tfrac 1n) \right), \qquad x\in \Omega.
$$  
\end{prop}

\begin{proof}
Let $m = \lfloor \frac{n}3 \rfloor$. Let $ T_x(y)$ be the polynomial of degree $m$ in Assertion 4. Let 
$\Xi$ be a maximal $\frac{\delta}{n}$-separated set in $\Omega$ with $\delta \le \delta_0$ as in 
Theorem \ref{thm:MZinequality}. By the inequality \eqref{eq:MZ2} with $p = 1$ and (ii) of 
Lemma \ref{lem:doublingLem} for both $\sm$ and $\sw$, we obtain 
\begin{align*}
\int_\Omega [ T_x(y)]^2 \sw(y)\d \sm(y) \, & 
      \le c \sum_{z \in \Xi}  [ T_x(z)]^2 \sw \big(B(z,\tfrac \delta n)\big) \\
  &  \le c \frac{\sw\big(B(x,\tfrac \delta n)\big)}{[\sm\big(B(x,\tfrac \delta n)\big)]^2}
        \sum_{z \in \Xi}  [ T_x(z)]^2 (1+ n \sd (x,z))^{\a} [\sm\big(B(z,\tfrac \delta n)\big)]^2,     
\end{align*}
where $\a = \a(\sw)+ 2 \a(\sm)$. Let $Q_x$ be the polynomial of degree $m$ that satisfies
\eqref{eq:Ass4Q} with $\sm$ in place of $\sw$. Using the lower bound of $Q_x$ in \eqref{eq:Ass4Q} 
and applying the inequality \eqref{eq:MZ1} with $p=1$ on the polynomial $[ T_x(z)]^2 Q_y(z)$ of degree
 $n$ with respect to the measure $\d \sm$, we obtain
\begin{align*}
\int_\Omega [ T_x(y)]^2 \sw(y)\d \sm(y)
 &  \le c \frac{\sw\big(B(x,\tfrac \delta n)\big)}{[\sm\big(B(x,\tfrac \delta n)\big)]^2}
            \sum_{z \in \Xi} [ T_x(z)]^2 Q_x(z) \sm (B(z,\tfrac{\delta}{n})) \\ 
 &  \le c\, \frac{\sw\big(B(x,\tfrac \delta n)\big)}{[\sm\big(B(x,\tfrac \delta n)\big)]^2} 
           \int_\Omega [ T_x(v)]^2 Q_x(v) \d \sm(v). 
\end{align*}
Using now the upper bound of $Q_x$ in \eqref{eq:Ass4Q}, again with $\sm$ in place of $\sw$, and the 
upper bound of $T_x(v)$ in Assertion 4, and (ii) of Lemma \ref{lem:doublingLem} for the measure $\d \sm$, 
we further deduce that the last integral is bounded by
\begin{align*}
  \int_\Omega [ T_x(v)]^2 Q_x(v) \d \sm(v) & \le c \int_\Omega \frac{\sm(B(v,\tfrac{\delta}{n}))}
       { (1+ n \sd (x,v))^{2 \g - \a }}  \d \sm(v) \\
       & \le c \sm(B(x,\tfrac{\delta}{n})) \int_\Omega \frac{1} { (1+ n \sd (x,z))^{2 \g - \a - \a(\sm)}}  \d \sm(v) \\
       & \le c \big[\sm\big(B(x,\tfrac{\delta}{n})\big)\big]^2,
\end{align*}
where the last step follows from \eqref{eq:int-dm}. Together, from the last two displayed inequalities follows
that 
$$
\int_\Omega [ T_x(y)]^2 \sw(y)\d \sm(y)  \le c \sw\big(B(x,\tfrac \delta n)\big). 
$$
Since $ T_x(x) =1$, this completes the proof by \eqref{eq:ChristoffelF}.
\end{proof}

\begin{prop}\label{prop:ChristF2}
If $\sw$ is a weight function for which Assertion 1 holds, then 
$$
    \l_n(\sw; x) \ge c\, \sw\! \left(B(x,\tfrac 1n) \right), \qquad x\in \Omega.
$$  
\end{prop}

\begin{proof}
Let $L_n(\sw; \cdot,\cdot)$ be the localized kernel defined with a cut-off function of type (a). It 
follows immediately that $K_n(\sw; x,x) \le L_n(\sw; x,x)$. Hence, if $\sw$ admits Assertion 1, then
$$
    K_n(\sw; x,x)  \le L_n(\sw;x,x) \le  \frac{c}{\sw  (B(x,\tfrac 1n))}, 
$$ 
which is equivalent to the lower bound $\l_n(\sw; x)$ by \eqref{eq:ChristoffelF2}.
\end{proof}

Together, the last two propositions give both upper and lower bounds for the Christoffel functions. 
Such bounds are known for regular domains, see for example \cite{DX, KrooLub, X95} and the references 
therein, and it has been established recently for fairly general convex domains \cite{Kroo,Pr1} and 
planar domains with piecewise boundary \cite{Pr2}. None of the previous results, however, apply to 
conic domains that will come out as special cases of the above theorem. 

\subsection{Positive cubature rules}
Let $\sw$ be a doubling weight on $\Omega$. A cubature rule of degree $n$ for $\sw$ is a finite linear
combination of point evaluations that satisfies 
$$
  \int_\Omega f(x) \sw(x)  \d \sm(x) = \sum_{k=1}^N \l_k f(x_k), \qquad  \forall f \in \Pi_n(\Omega),
$$
where $\l_k \in \RR$ and $x_k \in \Omega$. When $\l_k > 0$ for all $k$, the cubature rule is called 
positive. Given a maximal $\ve$-separated subset $\Xi$ on $\Omega$, we establish the existence 
of a positive cubature rule based on the points in $\Xi$, which plays an essential role for discretizing 
our near best approximation $L_n * f$.  

The proof will use the Farkas lemma (cf. \cite[Lemma 6.3.2]{DaiX}) stated below.

\begin{lem}\label{prop:Farkas}
Let $V$ be a finite dimensional real Hilbert space with inner product $\la \cdot,\cdot\ra$. Then
for any points $a^1, a^2, \dots, a^m$ and $\zeta \in V$, exactly of the following two systems has a 
solution: 
\begin{enumerate}[     \rm (1)]
\item $\displaystyle{\sum_{j=1}^m \mu_j  a^j = \zeta}, \qquad  0 \le \mu_1,\mu_2, \ldots, \mu_m \in \RR$.
\item $\la a^j,x\ra \ge 0$, $j =1,2,\ldots, m$ and $\la \zeta, x \ra < 0$ for some $x \in V$. 
\end{enumerate}
\end{lem}

We use the lemma to prove the existence of the positive cubature rules. For the unit sphere and the unit 
ball, this was established in \cite{MNW} and later in \cite{BD, NPW1, PX2}. We follow the proof in  
\cite[Theorem 6.3.3]{DaiX} for the doubling weight on the unit sphere.
% see also \cite{PX1}. 

\begin{thm}\label{thm:cubature}
Let $\sw$ be a doubling weight on $\Omega$. Let $\Xi$ 
be a maximum $\frac{\delta}{n}$-separated subset of $\Omega$. There is a $\delta_0 > 0$ 
such that for $0 < \delta < \delta_0$ there exist positive numbers $\l_z$, $z \in \Xi$, so that 
\begin{equation}\label{eq:CFgeneral}
    \int_{\Omega} f(x) \sw(x)  \d \sm(x) = \sum_{z \in \Xi }\l_z f(z), \qquad 
            \forall f \in \Pi_n(\Omega).
\end{equation}
Moreover, $\l_z \ge c_1 \sw\!\left(B(z, \tfrac{\delta}{n})\right)$ and, if Assertion 4 holds, then 
$\l_z \sim \sw\!\left(B(z, \tfrac{\delta}{n})\right)$.
\end{thm}
 
\begin{proof}
Let $\varpi$ be the weight function that admits highly localized kernels on $\Omega$. We use 
Lemma \ref{prop:Farkas} with the inner product defined by $\la \cdot,\cdot\ra_{\varpi}$. Recall that 
$K_n(\varpi; \cdot,\cdot)$ denotes the reproducing kernel of $\Pi_n(\Omega)$. Let $\{z_k: 1 \le k \le N\}$
be an enumeration of $\Xi$. We define the functions $a^j$ and $\zeta$ by 
\begin{align*}
  a^j(x) \, & = K_n\big(\varpi; x, z_j\big), \qquad 1 \le j \le N, \\
  \zeta(x) & = \int_{\Omega} K_n\big(\varpi; x, y\big) \sw(y)  \d \sm (y) - 
     \frac{1}{2 c_d} \sum_{j=1}^N K_n\big(\varpi; x, z_j\big) \sw\big(B(z_j, \tfrac{\delta}{n}) \big),
\end{align*}
where $c_d$ is the constant in \eqref{eq:def-pts2}. For every $f \in \Pi_n(\Omega)$, the reproducing property of 
$K_n(\varpi; \cdot,\cdot)$ implies that 
\begin{align*}
   \la f, a^j\ra_{\varpi}  & =  f(z_j), \qquad 1 \le j \le N, \\ 
   \la f, \zeta\ra_{\varpi} & = \int_{\Omega}f(y) \sw(y)  \d \sm(y) 
            - \frac{1}{2 c_0} \sum_{j=1}^N f (z_j) \sw\!\big(B(z_j), \tfrac{\delta}{n}\big).
\end{align*}
Assume $\la f, a^j\ra_\varpi \ge 0$ for all $j$, so that $f$ is nonnegative on $\Xi$.  Then, by
\eqref{eq:def-pts2},
\begin{align*}
  \int_\Omega   f(x) & \sw(x)  \d \sm(x)  = \sum_{j=1}^N \int_{B(z_j,\frac{\delta}{n})}  f(x)
       \frac{\sw(x)}{\sum_{z \in \Xi} \chi_{B(y,\frac{\delta}{n})}(x)}  \d \sm(x) \\
    & \ge \frac{1}{c_d}  \sum_{j=1}^N f(z_j) \int_{B(z_j,\frac{\delta}{n})} \sw(x)  \d \sm(x)
        - \sum_{j=1}^N \int_{B(z_j,\frac{\delta}{n})} \mathrm{osc}(f)(x,\tfrac \delta n) \sw(x)  \d \sm(x) \\
    & \ge  \frac{1}{c_d} \sum_{j=1}^N f(z_j)  \sw\big(B(z_j,\tfrac{\delta}{n})\big) - c_1\delta \int_{\Omega} |f(x)| 
      \sw(x)  \d \sm(x)
\end{align*}
by \eqref{eq:osc}. Hence, using the Marcinkiewicz-Zygmund inequality in (ii) of Theorem \ref{thm:MZinequality}
with $p = 1$, we conclude that 
\begin{align*}
    \la f, \zeta\ra_{\varpi}  \ge \left((2 c_d)^{-1} - c_{\sw,1} c_1 \delta\right)\sum_{j=1}^N f(z_j)  
       \sw\big(B(z_j,\tfrac{\delta}{n})\big)  \ge 0
\end{align*}
if we choose $\delta \le (2 c_d)^{-1} /(c_{\sw,1} c_1)$. Consequently, (2) of Farkas lemma does not hold so that
(1) must hold, which shows that there exist $\mu_j \ge 0$ such that 
$$ 
      \int_\Omega   f(x)  \sw(x)  \d \sm(x)  - \frac{1}{2 c_0} \sum_{j=1}^N f (z_j) \sw\!\big(B(z_j, \tfrac{\delta}{n})\big)
          = \sum_{j=1}^N \mu_j f(z_j). 
$$ 
This establishes the cubature rule \eqref{eq:CFgeneral} with $\l_{z_j} = (2c_0)^{-1}\sw(B(z_j, \tfrac{\delta}{n}))+\mu_j$.

It follows immediately that $\l_z \ge c \sw(B(z, \tfrac{\delta}{n}))$ for all $z \in \Xi$. To prove the upper bound 
of this inequality, we set $m = \lfloor \frac{n}{2} \rfloor$ and use the reproducing kernel $K_m(\sw; \cdot,\cdot)$ of
$\Pi_n(\Omega)$ with respect to $\sw$. Since $K_m(\sw; z,z) > 0$, we define a polynomial 
$$
     q_z(x) =  \frac{K_m(\sw; x,z)}{K_m(\sw; z,z)}, \qquad z \in \Xi.
$$
For each $z$, the polynomial $q_z^2$ is a polynomial of degree at most $n$ and $q_z(z) =1$. 
Hence, by \eqref{eq:CFgeneral}, 
$$
   \l_z  \le \sum_{y \in \Xi} \l_y  [q_z(y)]^2 
       = \int_{\Omega}  [ q_z(x) ]^2 \sw(x)  \d \sm(x) = \int_{\Omega}   \frac{K_m(\sw; x,z)^2}{K_m(\sw; z,z)^2} \sw(x)  \d \sm(x).
$$
Hence, using the reproducing property of the kernel $K_m(\sw; \cdot,\cdot)$, 
$$
  \int_\Omega K_m(\sw; x,z)^2 \sw(x)  \d \sm(x) =K_m(\sw; z,z),
$$
we conclude that $\l_z \le c [K_m(\sw; z,z)]^{-1}= \l_n(\sw; z)$. Hence, $\l_z \le c\, \sw(B(z, \tfrac{\delta}{n}))$ by
Proposition \ref{prop:ChristF1}.
\end{proof}

\subsection{Tight polynomial frames}
The highly localized kernels are powerful tools when they exist. We use them to study approximation behavior
of the integral operator $L_n(\varpi)*f$ that has $L_n\big(\varpi; \cdot,\cdot)$ as its kernel, or more precisely, 
\begin{equation}\label{eq:Lnf}
   L_n(\varpi)* f (x) =  \int_{\Omega} f(y) L_n\big(\varpi; x, y) \varpi(y)  \d \sm(y). 
\end{equation}
The fast decaying of the kernel indicates that $L_n(\varpi)*f$ provides a good approximation to the 
function $f$. Its approximation property will be studied in the next section. For now, we observe that 
$L_n(\varpi)*f$ is a polynomial of degree at most $2n$ by the property of the cut-off function. 
%$L_n(\varpi)*g = g$ for any polynomial $g \in \Pi_n(\Omega)$ of degree $n$;

We now use the operator $L_n(\varpi)* f$ and the positive cubature rules based on the discrete set
$\Xi$ to construct tight polynomial frames in $L^2(\Omega, \varpi)$. The construction is fairly 
standard by now, so is the proof of the tight frame. See, for example, \cite{DaiX, NPW1, PX1, PX2} 
for the case of the interval, the unit sphere and the unit ball. We recall the procedure and will be brief in proof. 

Let $(\Omega, \varpi, \sd)$ be a localizable homogeneous space. For the highly localized kernels
$L_n(\varpi;x,y)$, we assume its cut-off function $\wh a$ is of type (b) and satisfies 
\begin{align} \label{eq:a-frame}
\begin{split} 
 \wh a(t)\ge \rho > 0, \qquad & \mbox{if $t \in [3/5, 5/3]$},\\
 [\wh a(t)]^2 + [\wh a(2t)]^2 =1, \qquad & \mbox{if $t \in [1/2, 1]$.}
\end{split}
\end{align}
If $g$ is a nonnegative even function in $C^\infty(\RR)$, so that $\mathrm{supp}\, g = [-1, 1]$, $g(0) =1$ and 
$|g(t)|^2 + |g(t+1)|^2 =1$ on  $[-1, 0]$, then $\wh a(t) := g(\log_2t)$ has the desired properties. The last 
assumption on $\wh a$ implies 
\begin{equation}\label{eq:a3}
\sum_{j=0}^\infty  \left[ \wh a\left( \frac{t}{2^{j} } \right) \right]^2 = 1, \quad t \in [1, \infty).
\end{equation}
 
Let $\sw$ be the doubling weight that admits Assertion 4. Then the cubature rules in 
Theorem \ref{thm:cubature} are well established for $\sw$. For $j =0,1,\ldots,$ let 
$\ve_j = \frac \delta {2^{j}}$ and let $\Xi_j$ be a maximal $\ve_j$-separated subset in $\Omega$, 
where $\delta$ is chosen so that the cubature rule in Theorem \ref{thm:cubature} holds; 
that is, there are $\l_{z,j} > 0$ for $z \in \Xi_j$ such that 
\begin{equation}\label{eq:cuba-frame}
  \int_\Omega f(x) \sw(x)  \d \sm(x) = \sum_{z \in \Xi_j} \Xi_{z,j} f(z), \qquad f \in \Pi_{2^j} (\Omega); 
\end{equation}
moreover, $\l_{z,j} \sim \sw(B(z, 2^{-j}))$ since we assume that Assertion 4 holds for $\sw$. In terms of 
the kernel $L_n(\sw;\cdot,\cdot)$ defined via the cut-off function $\wh a$, we define 
$$
  F_0(x,z):=1, \quad \hbox{and}\quad F_j(x,z): = L_{2^{j-1}}(\sw; x, z), \quad j = 1, 2, 3, \ldots. 
$$
Accordingly, we define $F_j * f = L_{2^{j-1}} * f$, that is, 
$$
   F_j * f (x): = \int_\Omega f(y) F_j(x,y) \sw(y)  \d \sm(y), \qquad j = 0,1,2,\ldots. 
$$

Now, for $z \in \Xi_j$ and $j = 1, 2,\ldots$, we define our frame elements by 
$$ 
      \psi_{z,j}(x):= \sqrt{\l_{z,j}} F_j(x, z). 
$$ 
Then $\Psi:= \big\{\psi_{z,j}: z \in \Xi_j, \quad 1 \le j \le \infty\big\}$ is a frame system. 

By the orthogonality of the reproducing kernel,  it follows readily that
\begin{align*}
F_j*F_j * f = \sum_{k=1}^{2^{j}} \left|\wh a\left (\frac{k}{2^{j-1}}\right)\right|^2 \proj_k(\sw; f).
\end{align*}
Hence, the following semi-discrete Calder\'{o}n type decomposition follows from \eqref{eq:a3}, 
\begin{equation}\label{eq:f=F*F*f}
   f  =  \sum_{k=0}^\infty \proj_k(\sw;f) = \sum_{j=0}^\infty F_j* F_j * f, \qquad f\in L^2(\Omega,\sw). 
\end{equation}
Furthermore, the system $\Psi$ is a tight frame in $L^2(\Omega, \sw)$ norm. 
 
\begin{thm}\label{thm:frame}
Assume that $\Omega$ admits a localizable homogenous space. Let $\sw$ be a doubling weight 
satisfying Assertion 4. If $f\in L^2(\Omega, \sw)$, then
\begin{equation} \label{eq:f=frame}
   f =\sum_{j=0}^\infty \sum_{z \in\Xi_j}
            \langle f, \psi_{z, j} \rangle_\sw \psi_{z,j}  \qquad\mbox{in $L^2(\Omega, \sw)$}
\end{equation}
and
\begin{equation} \label{eq:tight-frame}
\|f\|_{2, \sw}  = \Big(\sum_{j=0}^\infty \sum_{z \in \Xi_j} |\langle f, \psi_{z,j} \rangle_\sw|^2\Big)^{1/2}.
\end{equation}
\end{thm}
 
\begin{proof}
Since $F_j(x,\cdot) F_j(\cdot,y)$ is a polynomial of degree $2^j$, we can apply the cubature rule \eqref{eq:cuba-frame}
to discretize $F_j*F_j$, which gives 
$$
     F_j * F_j(x,y) = \int_\Omega F_j(x,z)F_j(z,y) \sw(z)  \d \sm(z) =   \sum_{z \in \Xi_j} 
        \psi_{z,j}(x)\psi_{z,j}(y),
$$
from which the identity \eqref{eq:f=frame} follows from \eqref{eq:f=F*F*f} right away. Furthermore, 
let $S_R f$ be the sum $S_R f = \sum_{j=0}^R \sum_{z \in\Xi_j} \langle f, \psi_{z, j} \rangle_\sw \psi_{z,j}$. Then, 
it follows immediately
$$
   \la f, S_R f \ra_\sw = \sum_{j=0}^R  \sum_{z \in\Xi_j} \left |\langle f, \psi_{z, j} \rangle_\sw\right|^2.
$$
Taking the limit $R\to \infty$ we obtain the tight frame identity \eqref{eq:tight-frame}.
\end{proof}

For the weight function $\varpi$ that admits the highly localized kernels, the frame element $\psi_{z,j}$ 
has near exponential rate of decay away from its center with respect to the distance $\sd(\cdot, \cdot)$ 
on $\Omega$. 

\begin{thm}\label{p:localization}
Let $(\Omega,\varpi, \sd)$ be a localizable homogeneous space. There, there is a constant $c_\k >0$ 
depending only on $\k$, $d$, $\varpi$ and $\wh a$ such that for $z \in \Xi_j$, $j=0, 1, \dots$,  
\begin{equation} \label{est.needl}
   |\psi_{z,j}(x)| \le c_\k \frac{1}{\sqrt{\varpi(B(z, 2^{-j}))} (1+ 2^j \sd(x,z))^\k}, \quad x\in \Omega.
\end{equation}
\end{thm}

This follows readily from the fast decaying of the kernel and $\l_{z,j} \sim \sw(B(z, 2^{-j}))$. It shows
that the frame $\Psi$ is highly localized, which makes it a powerful tool for various applications, such
as decompositions of functions spaces and computational harmonic analysis, on more specific domains.  

%%%%%%%%%%%%
%%     section 3     %%    
%%%%%%%%%%%%
\section{Polynomial Approximation on homogeneous spaces}
\setcounter{equation}{0}
We consider approximation by polynomials in the space $L^p(\Omega,\varpi)$ when the weight function
$\varpi$ posses two additional properties that are analogs of two characteristics of spherical harmonics: 
addition formula for the reproducing kernels and a differential operator that has orthogonal polynomials 
as eigenfunctions. With appropriate weight functions, these properties are shared by orthogonal polynomials
on the unit sphere, on the unit ball, and on the simplex. The addition formula premises that the reproducing 
kernels possess a one-dimensional character, which leads to a convolution structure that allows us to reduce
much of the study of the Fourier orthogonal series to that of the Fourier-Jacobi series of one variable. 
The differential operator that has orthogonal polynomials as eigenfunctions leads to a $K$-functional. 
Together they provide us with tools for characterizing the best approximation by polynomials, which is a 
central problem in approximation theory. The goal of this section is to develop this framework, which  
extends the results on the unit sphere and the unit ball and it is applicable on the conic domains. 

The two characteristic properties are defined and discussed in the first subsection. The first one is 
used to define the convolution structure that leads to the definition of the modulus of smoothness in the 
second subsection. The second one is used to define a $K$-functional in the third subsection, where the main 
results on the characterization of the best polynomial approximation are stated and discussed. The 
near-best approximation operator, used to provide a direct estimate, is studied in the fourth subsection 
with the help of the differential operator, where some of the results are established with norms defined 
via a doubling weight. The Bernstein inequality, essential for proving the inverse estimate, is proved in 
the fifth subsection, together with the Nikolskii inequality. Finally, the proof of the main results is given in 
the seventh subsection. 

\subsection{Addition formula and differential operator} \label{sec:two-properties}

Let $\varpi$ be a weight function defined on $\Omega$. For studying approximation by polynomials, 
we require two more properties on the orthogonal structure of $L^2(\Omega,\varpi)$. These properties
are analogs to the two characteristics of spherical harmonics on the unit sphere: the Laplace-Beltrami 
operator \eqref{eq:sph-harmonics} and the addition formula \eqref{eq:additionF}, which serve as the 
quintessential examples in the study \cite{X20a,X20b} as well as in this subsection.  These properties 
are shared by several other domains, including the unit ball, the simplex, and conic domains. Below we 
assume these properties and use them to carry out our subsequent study.

The first property that we assume is an analog of the Laplace-Beltrami operator. 
Recall that $\CV_n(\Omega, \varpi)$ denote the space of orthogonal polynomials of degree $n$ with 
respect to $\varpi$ on $\Omega$. 
 
\begin{defn}\label{def:LBoperator}
Let $\varpi$ be a weight function on $\Omega$. We denote by $\fD_\varpi$ the second order 
derivation operator that has orthogonal polynomials with respect to $\varpi$ as eigenfunctions; 
more precisely, 
\begin{equation}\label{eq:LBoperator}
  \fD_{\varpi} Y = - \mu(n) Y, \qquad \forall\, Y \in \CV_n(\Omega,\varpi),
\end{equation}
where $\mu$ is a nonnegative quadratic polynomial. 
\end{defn}

For the unit sphere $\sph$, the operator $\fD_\varpi$ is the Laplace-Beltrami operator given by 
\eqref{eq:sph-harmonics}. For the unit ball $\BB^d$, $\fD_\varpi$ is the operator given by 
\eqref{eq:diffBall}. In both cases, the operator is a differential operator. For the unit sphere 
and the unit ball with reflection invariant weight functions, $\fD_\varpi$ is a differential-difference 
operator, called Dunkl Laplacian \cite{Dunkl, DX}. For $d =2$ and $\Omega$ is a domain with non-empty
interior, the existence of a second order differential operator that satisfies \eqref{eq:LBoperator} was 
characterized in \cite{KS}. There are essentially five cases with positive weight functions on regular 
domains but only two are compact domains: disk and triangle. The characterization for $d > 2$ remains
an open problem. The operators for the conic domains are the recent additions, discovered in \cite{X20a,X20b} 
and will be recalled in later sections. 
 
The second property that we assume is the addition formula, which gives a closed-form formula
for the reproducing kernel $P_n(\varpi; \cdot,\cdot)$ of $\CV_n(\Omega, \varpi)$.

\begin{defn}\label{defn:additionF}
Let $\varpi$ be a weight function on $\Omega$. The reproducing kernel $P_n(\varpi; \cdot,\cdot)$ 
is said to satisfy an addition formula if, for some $\a \ge \b \ge - \f12$,
\begin{equation*}% \label{eq:7Pn}
   P_n (\varpi; x, y) = \int_{[-1,1]^m} Z_n^{(\a,\b)} \big(\xi(x, y; u) \big) \d \tau (u), \qquad
          Z_n^{(\a,\b)}(t) = \frac{P_n^{(\a,\b)}(1) P_n^{(\a,\b)}(t)}{h_n^{(\a,\b)}},
\end{equation*}
where $m$ is a positive integer; $\xi(x, y; u)$ is a function of $u \in [-1,1]^m$, symmetric 
in $x$ and $y$, and $\xi(x, y; u) \in [-1,1]$; moreover, $\d \tau$ is a probability measure on 
$[-1,1]^m$, which can degenerate to have a finite support. 
\end{defn}

The classical addition formula is the one for the unit sphere $\sph$, given in \eqref{eq:additionF}, which 
is the degenerate case, with $\a=\b = \frac{d-3}2$. The one for the classical weight $\varpi_\mu$ on the 
unit ball $\BB^d$ is given by \eqref{eq:additionBall}, which has $m =1$ and $\a = \b = \mu+\frac{d-2}{2}$. 
Similar formula holds for reflection invariant weight functions on the unit sphere and on the unit ball 
(cf. \cite[p. 221 and p. 265]{DX}). Furthermore, for the simplex in $\RR^d$ (cf. \cite[p. 275]{DX}) and the
conic domains $\VV_0^{d+1}$ and $\VV^{d+1}$, the polynomial $Z_n^{(\a,\b)}$ is given by $Z_{2n}^\l$ for 
some $\l \ge 0$, so that, by the quadratic transform \eqref{eq:Jacobi-Gegen0}, it is given by the Jacobi 
polynomial with $\a = \l-\f12$ and $\b= -\f12$. 

The addition formula premises a one-dimensional structure for the reproducing kernel. In all known cases,
the estimates of the highly localized kernel $L_n(\varpi; \cdot,\cdot)$ are established with the help of this
formula, since the addition formula implies 
\begin{equation}\label{eq:Ln-LnJacobi}
   L_n(\varpi; x, y)  = \int_{[-1,1]^m} L_n^{(\a,\b)} \big(\xi(x, y; u) \big) \d \tau (u),
\end{equation}
where the function $t \mapsto L_n^{(\a,\b)}(t) = L_n^{(\a,\b)}(t,1)$ is the kernel defined in 
\eqref{def.L} for the Jacobi polynomials. 

In the rest of this section, we assume that the orthogonal polynomials with respect to $\varpi$ posses
both the derivation operator and the addition formula. We do not, however, require that $\varpi$ admits
highly localized kernels that satisfy Assertions 1--3. In fact, some of our results hold if $\varpi$ satisfies
Assertions 1 and 3, but need not satisfy Assertion 2. 

\subsection{Convolution structure}
Making use of the one-dimensional structure premised by the addition formula, we define a convolution 
operator on $\Omega$. 
 
\begin{defn}\label{defn:7convol}
Assume $\varpi$ admits the addition formula. For $f \in L^1(\Omega, \varpi)$ and 
$g \in L^1([-1,1],w_{\a,\b})$, we define the convolution of $f$ and $g$ by 
$$
  (f \ast_\varpi  g)(x) :=   \int_{\Omega} f(y) T^{(\a,\b)} g (x,y) \varpi(y) \d \sm(y),
$$
where the operator $g\mapsto T^{(\a,\b)} g$ is defined by 
\begin{align*} 
   T^{(\a,\b)} g(x,y) := \int_{[-1,1]^m}  g \big( \xi (x, y; u)\big)  \d \tau(u).
\end{align*}
\end{defn}

The addition formula shows that the projection operator $\proj_n (\varpi; f)$ and the reproducing kernel 
$P_n(\varpi; \cdot,\cdot)$ satisfy 
\begin{equation}\label{eq:proj=f*g}
     \proj_n(\varpi; f, x) = f*_\varpi Z_n^{(\a,\b)} \quad\hbox{and}\quad P_n(\varpi; x,y) = T^{(\a,\b)} \left( Z_n^{(\a,\b)}\right).
\end{equation}

The operator $T^{(\a,\b)}$ is defined in \cite{X20b} with a more generic orthogonal polynomial in place of 
$P_n^{(\a,\b)}$, where the following boundedness of the operator is established \cite[Lemma 6.3]{X20b}. 

\begin{lem} \label{lem:translateT}
Let $g \in L^1([-1,1],w_{\a,\b})$. Then, for each $Q_n \in \CV_n (\Omega, \varpi)$, 
\begin{equation}\label{eq:FHformula}
     \int_{\Omega}  T^{(\a,\b)} g (x,y) Q_n (y) \varpi(y) \d y  =  \Lambda_n^{(\a,\b)} (g) Q_n(x),
\end{equation}
where 
$$
 \Lambda_n^{(\a,\b)} (g) =  c_{\a,\b} \int_{-1}^1 g(t) R_n^{(\a,\b)}(t)  w_{\a,\b}(t)\d t 
 \quad \hbox{with} \quad R_n^{(\a,\b)}(t) : = \frac{P_n^{(\a,\b)}(t)}{P_n^{(\a,\b)}(1)}
$$
Furthermore, for $1\le p \le \infty$ and $x \in \Omega$, 
\begin{equation}\label{eq:Tbd}
  \left \| T^{(\a,\b)} g (x, \cdot)\right\|_{L^p(\Omega, \varpi)} \le \|g \|_{L^p([-1,1],w_{\a,\b})}.
\end{equation}
\end{lem}
 
The convolution operator $f \ast_\varpi g$ is also defined in \cite{X20b} and shown to satisfy
the usual Young's inequality. 

\begin{thm} \label{thm:Young}
Let $p,q,r \ge 1$ and $p^{-1} = r^{-1}+q^{-1}-1$. For $f \in L^q(\Omega, \varpi)$ and
$g \in L^r([-1,1]; w_{\a,\b})$, 
\begin{equation} \label{eq:Young}
  \|f \ast_\varpi g\|_{L^p (\Omega, \varpi)} \le \|f\|_{L^q(\Omega,\varpi)}\|g\|_{L^r( [-1,1];\varpi)}.
\end{equation}
\end{thm}

As an application of these results, we consider the Ces\`aro $(C,\delta)$ means 
$S_n^\delta (\varpi;f)$ of the Fourier orthogonal series with respect to $\varpi$. For $\delta > 0$, 
the operator $S_n^\delta(\varpi)$ is defined by 
$$
 S_n^\delta (\varpi;f) := \f{1}{\binom{n+\delta}{n}} \sum_{k=0}^n \binom{n-k+\delta}{n-k} \proj_k(\varpi; f).
$$
It can be written as an integral operator with the kernel $K_n^\delta(\varpi; \cdot,\cdot)$ being 
the $(C,\delta)$ mean of the reproducing kernel $P_n(\varpi;\cdot,\cdot)$. By the addition formula, 
the kernel $K_n^\delta(\varpi; \cdot,\cdot)$ can be written as 
$$
K_n^\delta(\varpi; x,y) = \int_{[-1,1]^m} k_n^{(\a,\b), \delta}(\xi(x,y;u), 1) \d \tau(u),
$$
where $k_n^{(\a,\b), \delta}(s,t)$ denotes the kernel of the $(C,\delta)$ mean of the Fourier-Jacobi series. 
In particular, using the positivity of the kernel $k_n^{(\a,\b), \delta}$ \cite{Ask} and the boundedness
of the kernel \cite[Theorem 9.1.3]{Sz}, we obtain the following theorem. 
 
\begin{thm}\label{thm:cesaro}
Let $\varpi$ be a weight function that admits the addition formula. The Ces\`aro $(C,\delta)$ means 
of the Fourier orthogonal series with respect to $\varpi$ satisfy 
\begin{enumerate}
\item If $\delta \ge \a+\b+2$, then $S_n^\delta(\varpi; f)$ is a nonnegative operator;
\item If $\delta > \a+\f12$, then for $n= 0,1,2,\ldots$,
$$
    \left \|S_n^\delta(\varpi; f)\right \|_{p,\varpi} \le  \|f\|_{p,\varpi}, \qquad 1 \le p \le \infty.
$$
\end{enumerate}
\end{thm} 

To further explore the one-dimensional structure that leads to the definition of the convolution, we
define an operator $S_{\t,\varpi} f$ as follows. 

\begin{defn}
Assume the addition formula holds for the weight function $\varpi$. For $0\le \t \le \pi$, the translation 
operator $S_{\t,\varpi}$ is defined by 
\begin{equation}\label{eq:Stheta}
    \proj_n (\varpi; S_{\t, \varpi} f) = R_n^{(\a,\b)}(\cos \t) \proj_n(\varpi; f), 
      \quad n = 0,1,2,\ldots.
\end{equation}
\end{defn}

As a corollary of Theorem \ref{thm:cesaro}, a function $f \in L^1(\Omega,\varpi)$ is uniquely determined
by its orthogonal projections $\proj_n (\varpi; f)$, $n \ge 0$. Hence, the operator $S_{\t,\varpi}$ is well 
defined for all $f\in L^1(\Omega, \varpi)$. For the unit sphere with the Lebesgue measure $\d\s$, the 
operator $S_\t$ is an integral operator given by \cite[(2.16)]{DaiX}
$$
    S_{\t, \d\s} f(x) = \frac{1}{\o_{d-1}(\sin \t)^{d-1}} \int_{c(x,\t)} f(y) \d \ell(y),
$$
where $c(x,\t) = \{y: \la x,y\ra =\cos \t\}$ is the spherical cap and $\d \ell$ denotes the Lebesgue measure 
on the cap. For the unit ball $\BB^d$ with the classical weight $W_\mu$, the operator $S_{\t,W_\mu}$ 
can also be written as an explicit integral operator \cite[Theorem 11.2.6]{DaiX}.

\begin{prop}\label{prop:Stheta}
The operator $S_{\t,\varpi}$ satisfies the following properties:
\begin{enumerate}[    \rm (i)]
\item For $f \in L^2(\Omega, \varpi)$ and $g \in L^1([-1,1], w_{\a,\b})$, 
$$
   (f*_\varpi g)(x) = c_{\a,\b} \int_0^\pi S_{\t, \varpi} f(x) g(\cos\t) w_{\a,\b}(\cos \t) \sin \t \d \t. 
$$
\item $S_{\t, \varpi} f$ preserves positivity; that is, $S_{\t, \varpi}  f \ge 0$ if $f \ge 0$. 
\item For $f\in L^p(\varpi; \Omega)$, if $1 \le p \le \infty$, or $f \in C(\Omega)$ if $p =\infty$, 
$$
  \|S_{\t, \varpi} f \|_{p,\varpi} \le \|f\|_{p, \varpi} \quad \hbox{and} \quad \lim_{\t\to 0} \|S_{\t, \varpi}  f - f\|_{\varpi,p} =0.
$$
\end{enumerate}
\end{prop}

\begin{proof}
From the Fubini theorem and  \eqref{eq:FHformula}, it is easy to see that 
\begin{align*}
  \proj_n (\varpi; f*_\varpi g) \, & = \Lambda_n^{(\a,\b)}(g) \proj_n (\varpi; f) \\ 
     & = c_{\a,\b} \int_0^\pi \proj_n (\varpi; S_{\t,\varpi} f) g(\cos \t) w_{\a,\b}(\cos \t) \sin \t \d \t, 
\end{align*}
from which (i) follows. To prove (ii), we let $g_n$ be a non-negative function such that $g_n(\cos \t)$ is 
supported on $[-\f1n, \f1n]$ and $\int_0^\pi g_n(\cos \t) w_{\a,\b}(\t) \sin \t \d \t = 1$. Then, using
the expression in (i), $f \ast_\varpi g_n$ converges to $S_{\t,\varpi} f$, which proves (ii). Moreover, 
by \eqref{eq:Young}, $\|f\ast_\varpi g_n \|_{p, \varpi} \le \|f\|_{p,\varpi}$, so that by the Fatou lemma, 
$\|S_{\t, \varpi} f \|_{p,\varpi} \le \|f\|_{p, \varpi}$. Finally, if $f_n = S_n^\delta(\varpi;f)$ with $\delta \ge 
\a+\b+2$, then $f_n \ge 0$ and $f_n \to f$ in $L^p(\Omega, \varpi)$; furthermore, by \eqref{eq:Stheta},
$S_{\t,\varpi} f_n$ converges to $f_n$ for $\t \to 0$; consequently, by the triangle inequality, 
$\|S_{\t, \varpi}  f - f\|_{p,\varpi}$ converges to $0$ when $\t \to 0$. This completes the proof. 
\end{proof}

The operator $S_{\t,\varpi}$ can be used to define a modulus of smoothness. For $r > 0$, we defined 
the $r$-th difference operator 
$$
  \triangle_{\t,\varpi}^r = \left(I - S_{\t,\varpi}\right)^{r/2}  =\sum_{n=0}^\infty (-1)^n \binom{r/2}{n} (S_{\t,\varpi})^n,
$$
where $I$ denote the identity operator, in the distribution sense, by
\begin{equation}
\proj_n \left(\varpi; \triangle_{\t,\varpi}^r f\right)= \left(1- R_n^{(\a,\b)}(\cos \theta)\right)^{r/2} \proj_n(\varpi; f),
         \quad n=0,1,2,\cdots.
\end{equation}

\begin{defn} \label{def:modulus}
Let $r >0$ and $0 < \t < \pi$. For $f\in L^p(\Omega,\varpi)$ and $1\leq p<\infty$ or $f\in C(\Omega)$ and $p=\infty$,
the weighted $r$th order modulus of smoothness is defined by
\begin{equation}\label{moduli1}
\omega_r(f,t)_{p,\varpi}:= \sup_{0< \theta \le t} \left\|\triangle_{\t,\varpi}^r f\right\|_{p,\varpi}, \quad  0 < t <\pi.
\end{equation}
\end{defn}

For the unit sphere with the Lebesuge measure, the definition of this modulus of smoothness is classical. 
For weighed approximation on the unit sphere, the unit ball, and the simplex, it is given in \cite{X05}. 
The proof of the following lemma is standard, see \cite[Proposition 10.1.2]{DaiX} for example, 
and will be omitted. 

\begin{prop}\label{prop-1-2-ch10}
Let $f\in L^p(\Omega,\varpi)$ if $1\leq p<\infty$ and $f\in C(\Omega)$ if $p=\infty$. Then
\begin{enumerate}
\item $\o_r(f,t)_{p,\varpi}\leq 2^{r+2} \|f\|_{p,\varpi}$;
\item $\omega_r(f,t)_{p,\varpi} \to 0$ if $t \to 0+$;
\item $\omega_r(f,t)_{p,\varpi}$ is monotone nondecreasing on $(0,\pi)$;
\item $\omega_r(f+g,t)_{p,\varpi} \le \omega_r(f,t)_{\kappa,p}+ \omega_r(g,t)_{p,\varpi}$;
\item For $0 < s < r$,
$$
\omega_r(f,t)_{p,\varpi} \le 2^{(r-s)+2} \omega_s(f,t)_{p,\varpi}.
$$
\end{enumerate}
\end{prop}

\subsection{Characterization of best approximation}\label{sec:bestapp}

Let $\sw$ be a doubling weight on $\Omega$. For $f\in L^p(\Omega, \sw)$, we denote by 
$E_n(f)_{p, \sw}$ the best approximation to $f$ from $\Pi_n(\Omega)$ in the norm $\|\cdot\|_{p, \sw}$ 
of $L^p(\Omega, \sw)$; that is, 
$$
      E_n(f)_{p, \sw}:= \inf_{g \in \Pi_n(\Omega)} \|f - g\|_{p, \sw}, \qquad 1 \le p \le \infty.
$$
A central problem of approximation theory is to character this quantity by the smoothness of the functions, 
usually in terms of a modulus of smoothness or a $K$-functional, which are often equivalent. In this 
subsection, we state our main results on the characterization of the best approximation. 

The modulus of smoothness is already defined. We now define the $K$-functional via the differential 
operator $\fD_\varpi$ in \eqref{eq:LBoperator}. Since the operator $-\fD_\varpi$ has nonnegative 
eigenvalues, it is a non-negative operator. A function $f \in L^p(\Omega;\varpi)$ belongs to the Sobolev 
space $W_p^r (\Omega; \varpi)$ if there is a function $g \in L^p(\Omega; \varpi)$, which we denote by 
$(-\fD_\varpi)^{\f r 2}f$, such that  
\begin{equation} \label{eq:fracDiff}
  \proj_n\left(\varpi; (-\fD_\varpi)^{\f r 2}f\right) = \mu(n)^{\f r 2} \proj_n(\varpi; f),
\end{equation}
where we assume that $f, g \in C(\Omega)$ when $p = \infty$. The fractional differential operator 
$(-\fD_\varpi)^{\f r 2}f$ is a linear operator on the space $W_p^r (\Omega; \varpi)$ defined by 
\eqref{eq:LBoperator}.

Let $\sw$ be a doubling weight. We denote by $W_p^r(\Omega, \sw)$ the Sobolev space that
consists of $f \in L^p(\Omega, \sw)$ such that $(-\fD_\varpi)^{\f r 2}f \in L^p(\Omega, \sw)$. The  
space is well defined as we shall see later in Theorem \ref{thm:BernsteinLB}. We define our $K$-functional
as follows. 

\begin{defn}
Let $r > 0$ and $1 \le p \le \infty$. The $r$-th $K$-functional of $f\in L^p(\Omega, \sw)$ is defined by 
$$
   K_r(f,t)_{p,\sw} : = \inf_{g \in W_p^r(\Omega, \sw)}
      \left \{ \|f-g\|_{p,\sw} + t^r\left\|(-\fD_\varpi)^{\f r 2}g \right\|_{p,\sw} \right \}.
$$
\end{defn}

We first state of our characterization theorem in terms of the $K$-functional, which contains two parts, 
the first part is the direct estimate or the Jackson inequality, and the second part is the inverse estimate. 
For the first part, we need another assertion to ensure that a Bernstein inequality holds. This is 
Assertion 5 and it will be stated in Subsection \ref{set:Bernstein}.

\begin{thm}\label{thm:Enf-Kfunctional}
Let $\varpi$ be a weight function that admits Assertions 1, 3 and 5. Let $f \in L^p(\Omega, \varpi)$ 
if $1 \le p < \infty$ and $f\in C(\Omega)$ if $p = \infty$. Then for $r > 0$ and $n =1,2,\ldots$, 
\begin{enumerate} [   (i)]
\item the direct estimate   
$$
    E_n(f)_{p,\varpi} \le c K_r (f;n^{-1})_{p,\varpi};
$$
\item the inverse estimate  
$$
   K_r(f;n^{-1})_{p,\varpi} \le c n^{-r} \sum_{k=0}^n (k+1)^{r-1}E_k(f)_{p, \varpi}.
$$
\end{enumerate}
\end{thm}

It is worth mentioning that the direct estimate holds for the weight norm of $\varpi$, whereas the inverse 
estimate holds for all doubling weight. We can also replace the above characterization by the modulus of 
smoothness $\varpi_r (f;t)_{p,\varpi}$, since the two quantities will be shown to be 
equivalent.

\begin{thm}\label{thm:K=omega}
Let $\varpi$ be a weight function that admits Assertions 1, 3 and 5. Let $f \in L^p(\Omega,\varpi)$ if 
$1 \le p < \infty$ and $f\in C(\Omega)$ if $p = \infty$. If $0 < \t \le \pi/2$ and $r > 0$, then 
\begin{equation}\label{eq:K=omega}
   c_1 K_r(f;\t)_{p,\varpi} \le \o_r(f;\t)_{p,\varpi} \le c_2 K_r(f;\t)_{p,\varpi}. 
\end{equation}
\end{thm}
 
For the unit sphere and the unit ball, these theorems are known to hold when $\varpi$ being reflection 
invariant weight functions. For the unit sphere with the Lebesgue measure, the characterization of best
approximation has a long history. Starting from \cite{BBP, P} for $r =2$, the problem was studied by a number
of authors and finally completed in \cite{Rus}; we refer to \cite[p. 102 and p. 263]{DaiX} for historical notes.
These results are extended to the weighted case, with reflection invariant weight functions, on the unit 
sphere, the unit ball, and the simplex, in \cite{X05}. 

Another pair of $K$-functional and modulus of smoothness was used to characterize the best approximation
for the unit sphere with the Lebesgue measure in \cite{DaiX2}. The approach appears to be very much 
domain-specific and relies on the geometric and the differential structure on the unit sphere.  

We will give the proof of these two theorems in Subsection \ref{sect:bestapprox-proof}, after proving several 
preliminary results that are of interest in their own right. 

\subsection{Near best approximation operator}
In this subsection we study the approximation property of the operator $L_n(\varpi)*f$, defined in \eqref{eq:Lnf}, 
that has the highly localized kernel as its kernel. By \eqref{eq:Ln-LnJacobi}, the operator can be written in 
terms of the convolution operator by 
$$
  L_n(\varpi) *f =  f \ast_\varpi L_n^{(\a,\b)}, \qquad n =0,1,2,\ldots.
$$

We now show that $L_n(\varpi)*f$ provides near-best approximation to $f$. 

\begin{thm} \label{thm:nearbest}
Let $\wh a$ be admissible of type $(a)$. Assume $\varpi$ admits Assertions 1 and 3. Then the operator 
$L_n *f$ satisfies 
\begin{enumerate}[\quad (a)]
\item $L_n(\varpi)*f$ is a polynomial of degree at most $2n$;
\item $L_n(\varpi)*g = g$ for any polynomial $g \in \Pi_n(\Omega)$ of degree $n$;
\item For $f \in L^p(\Omega, \varpi)$, $1 \le p \le \infty$,
\begin{equation*} %\label{est.operators}
\|L_n(\varpi)*f\|_{p, \varpi} \le c \|f\|_{p, \varpi} \quad \hbox{and} \quad
\|L_n(\varpi)*f - f \|_{p, \varpi} \le c\, E_n (f)_{p, \varpi}.
\end{equation*}
\end{enumerate}
\end{thm}

\begin{proof}
The first two properties are immediate consequences of the definition. For the third property, in the case $p=1$
we obtain 
$$
\|L_n(\varpi)*f\|_{1,\varpi} \le c  \|f\|_{1,\varpi} \max_{x\in \Omega}  \int_{\Omega} \left| L_n\big(\varpi; x,y)\right|
    \varpi(y) \d \sm(y).
$$
By Assertion 1 and \eqref{eq:CorA3} with $p =1$ and a large $\tau$, the integral in the right-hand side is 
bound by
\begin{align*}
 \max_{x\in \Omega} \int_{\Omega} \frac{ \varpi(y) \d \sm(y)}
       {\sqrt{\varpi\left(B\left(y,n^{-1}\right)\right)}\sqrt{\varpi\left(B\left(x,n^{-1}\right)\right)}       
        \big(1 + n \sd(x,y) \big)^{\tau}}  \le c,
\end{align*}
which shows that $\|L_n(\varpi) *f\|_{p, \sw}$ is bounded for $p =1$. The same inequality also shows 
that the boundedness holds for $p = \infty$. The case $1< p < \infty$ follows from the Riesz-Thorin interpolation 
theorem. By item (b), the boundedness of $\|L_n(\varpi)*f\|_{p,\varpi}$ implies, by the triangle inequality, 
$$
  \|L_n(\varpi)*f - f \|_{p, \varpi} \le \|L_n(\varpi)* (f - g) \|_{p, \varpi} + \|f - g \|_{p, \varpi} \le c \|f-g\|_{p,\varpi}
$$
for any polynomial $g \in \Pi_n(\Omega)$. Taking infimum over $g$ completes the proof. 
\end{proof}
 
Because of property (c), we call the operator $L_n(\varpi)*f$ near best approximation polynomial. Such 
operators have been used extensively in approximation theory and computational harmonic analysis on the unit 
sphere and the unit ball (see \cite{DaiX}, for example), and in various other domains.  

\begin{thm} \label{thm:f-LnfD}
Let the assumption be the same as in the previous theorem. For $f \in W_p^r(\Omega, \varpi)$, $1\le p\le \infty$, 
\begin{equation} \label{eq:f-LnfD}
 \|f - L_n(\varpi)*f\|_{p,\varpi} \le c n^{-r} \left\| (-\fD_{\varpi})^{\f r 2}f \right\|_{p,\varpi}, 
 \qquad n =1,2,\ldots.
\end{equation}
\end{thm}

\begin{proof}
Let $\ell$ be a positive integer. Without loos of generality, we can assume $n > \ell$. Since 
$L_n(\varpi)*f = f$ if $f$ is a polynomial of degree $n$, we can write
\begin{align*}
 f - L_n(\varpi)*f \, & = \sum_{k=n+1}^\infty \left(1 - \wh a \left(\frac{k}{n} \right) \right)\proj_k (\varpi; f) \\
  & = \sum_{k=n+1}^\infty\! \left(1 - \wh a \!\left(\frac{k}{n} \right) \right) \mu(k)^{-\f r 2} 
        \proj_k\! \left(\varpi;F\right),
\end{align*}
where $F =\left(-\fD_\varpi\right)^{\f r 2} f$. Summation by parts $\ell+1$ times, we obtain 
$$
  f -L_n(\varpi)*f = \sum_{k=n+1}^\infty b_k  \proj_k\! \left(\varpi;F\right)
      = \sum_{k=n+1}^\infty \left( \triangle^{\ell+1} b_k\right) A_k^\ell S_k^\ell \left(\varpi; F\right),
$$
where $b_k = (1 - \wh a(\frac{k}{n} )) \mu(k)^{-\f r 2}$, $A_k^\ell = \binom{k+\ell}{k} \sim k^\ell$ 
and $S_k^\ell \left(\varpi; F\right)$ denotes the $k$-th Ces\`aro $(C,\ell)$ mean of the Fourier orthogonal 
series with respect to $\varpi$. Since $\wh a$ is in $C^\infty$ and its support is $[0,2]$, it is easy to see 
that $|\triangle^{\ell+1} b_k| \le c k^{-r-\ell-1}$. Hence, it follows that 
$$
 \sum_{k=n+1}^\infty \left| \triangle^{\ell+1} b_k\right| A_k^\ell  \le 
    c \sum_{k=n+1}^\infty k^{-r-\ell-1}  k^\ell  \le c n^r.
$$
Since the choice of $\ell$ implies that $\|S_k^\ell \left(\varpi; F\right)\|_{p, \varpi} \le \|F\|_{p, \varpi}$ by the 
convergence of the C\`esaro means, we obtain
$$
\|f -  L_n(\varpi)*f\|_{p, \varpi} \le cn^{-r} \|F\|_{p,\varpi},
$$
which is what we need to prove. 
\end{proof}

\subsection{Bernstein and Nikolskii inequalities}\label{set:Bernstein}

The Bernstein inequality is essential for studying approximation by polynomials. We establish
such an inequality for the deviation operator $\fD_\varpi$ defined by \eqref{eq:LBoperator}. 
First we prove a proposition that is of independent interest.  

\begin{prop}\label{prop:Tfw-norm}
Let $(\Omega, \varpi, \sd)$ be a localizable homogeneous space and assume Assertion 4 holds. 
Let $G_n(\cdot,\cdot): \Omega \times \Omega\mapsto \RR$ be a kernel such that $G_n(x,y) = G_n(y,x)$ 
for all $x, y\in \Omega$ and $G_n$ is a polynomial of degree $n$ in either of its variables. Let 
$T: f\mapsto Tf$ be the operator defined by 
$$
       T f(x) = \int_\Omega f(y)  G_n(x,y) \varpi(y) \d \sm(y). 
$$
If $f$ is a polynomial of degree $n$, then for $1 \le p \le \infty$ and any doubling weight $\sw$, 
\begin{equation}\label{eq:Tfw-norm}
  \|T f\|_{p,\sw} \le c \|f\|_{p,\sw}  
       \max_{z\in\Omega} \int_\Omega | G_n(x,z)| (1+ n \sd (z,x))^{\a(\sw)+\a(\varpi)} \varpi(x) \d \sm (x).
\end{equation}
\end{prop}

\begin{proof}
Since $T$ is a linear integral operator, if $f$ is a polynomial of degree $n$ then $T f$ is a polynomial of the 
same degree. Let $\Xi$ be an $\ve$-separated set of $\Omega$ for $\ve = \frac{\delta}{2n}$ so
that the Marcinkiewicz-Zygmund inequalities in Theorem \ref{thm:MZinequality} hold for all polynomials 
of degree $2n$. The inequality is trivial for $p = \infty$ since it holds without $ (1+ n \sd (z,x))^{\a(\sw)}$
term. For $1\le p < \infty$, by \eqref{eq:MZ2},
\begin{align*}
   \|T f\|_{p,\sw}^p \le c \sum_{z \in \Xi} | T f(z)|^p \sw \left (B(z,\tfrac{\delta} {2n})\right). 
\end{align*}
Since $f(y) G_n(z,y)$ is a polynomial of degree $2n$ in $y$ variable, applying \eqref{eq:MZ2} again, we obtain
$$
 | T f(z)| \le c \sum_{u \in \Xi}  |f(u)| |G_n(u, z)| \varpi\left (B(u,\tfrac{\delta} {2n})\right). 
$$
Assume $1\le p < \infty$. By the H\"older's inequality,  
\begin{align*}
   |Tf(z)|^p \,& \le c \sum_{u \in \Xi} |f(u)|^p |G_n(z,u)|\varpi\left (B(u,\tfrac{\delta} {2n})\right)
       \left( \sum_{u \in \Xi} |G_n(z,u)| \varpi\left (B(u,\tfrac{\delta} {2n})\right)\right)^{\f p q} \\
         & \le c \sum_{u \in \Xi} |f(u)|^p |G_n(z,u)|\varpi\left (B(u,\tfrac{\delta} {2n})\right)
            \|G_n(z,\cdot)\|_1^{\f p q}, 
\end{align*}
where we used \eqref{eq:MZ1} in the second step. Using (ii) of Lemma \ref{lem:doublingLem} for 
both $\varpi$ and $\sw$, we see that 
\begin{align*}
   & \sum_{z\in \Xi} |T f(z)|^p  \sw\!\left(B(z,\tfrac{\delta} {2n})\right) 
     \le c \max_{z \in \Xi} \|G_n(z,\cdot)\|_1^{\f p q} \\
   & \qquad  \times \sum_{u \in\Xi} |f(u)|^p   \sw\!\left(B(u,\tfrac{\delta} {2n})\right)   
      \sum_{z\in \Xi}  |G_n(z,u)| (1+ n \sd (z,u))^{\a(\sw)+\a(\varpi)} \varpi\!\left (B(z,\tfrac{\delta} {2n})\right).  
 \end{align*}
By Lemma \ref{lem:Ass4Q}, there is a nonnegative polynomial $Q_u$ of degree $n$ that satisfies 
\eqref{eq:Ass4Q} with $\a =\a(\sw)+ \a(\varpi)$, which allows us to apply \eqref{eq:MZ1} on the polynomial 
$z \mapsto G(z,u) Q_u(z)$ to show that the last sum in the right-hand side is bounded by 
\begin{align*}
  c \sum_{z\in \Xi} |G_n(z,u)| Q_u(z)  \varpi\! \left (B(z,\tfrac{\delta} {2n})\right)  
   & \le c  \int_\Omega | G_n(x,u)| Q_u(x) \varpi(x) \d \sm (x)  \\
   & \le c  \int_\Omega | G_n(x,u)| (1+ n \sd (x,u))^{\a(\sw)+\a(\varpi)} \varpi(x) \d \sm (x).  
\end{align*}
Putting these together and using \eqref{eq:MZ1} and \eqref{eq:MZ2}, we have proved that 
\begin{align*}
 \|T f\|_{p,\sw}^p \,&  \le c  \|f\|_{p,\sw}^p  \max_{z \in \Xi} \|G_n(z,\cdot)\|_1^{\f p q}  \\
     & \qquad  \times 
     \max_{u\in \Xi} \int_\Omega | G_n(x,u)| (1+ n \sd (x,u))^{\a(\sw)+\a(\varpi)} \varpi(x) \d \sm (x)\\
       & \le c  \|f\|_{p,\sw}^p  \left(\max_{u \in \Xi}
         \int_\Omega | G_n(x,u)| (1+ n \sd (x,u))^{\a(\sw)+\a(\varpi)} \varpi(x) \d \sm (x) \right)^{1+\f p q},
\end{align*}
where we have used the assumption that $G_n$ is symmetric in its variables. Since $1+ \f p q = p$, this
proves the stated inequality.
\end{proof}

For $r > 0$, we denote by $L_{n}^{(r)}(\varpi; \cdot,\cdot)$ the kernel defined by 
\begin{equation} \label{eq:Ln_r}
   L_{n}^{(r)}(\varpi; x,y) = \sum_{n=0}^\infty \wh a\left( \frac{k}{n} \right) [\mu(k)]^{r/2} P_k(\varpi; x,y),
\end{equation}
which is the kernel $\fD_\varpi^{r/2} L_n(x,y)$ with $\fD_\varpi^{r/2}$ applying on $x$ variable. Our Bernstein 
inequality is proved under the following assumption on the decaying of this kernel. 

\medskip
{\bf Assertion 5.} For $r > 0$ and $\k > 0$, the kernel $ L_{n}^{(r)}(\varpi)$ satisfies, for $x,y \in \Omega$,
$$
  \left|  L_{n}^{(r)}(\varpi; x,y) \right| \le c_\k \frac{n^r}{\sqrt{\varpi(B(x,n^{-1}))}\sqrt{\varpi(B(y,n^{-1}))}}
     (1+n\sd(x,y))^{-\k}.
$$
\smallskip

\begin{thm} \label{thm:BernsteinLB}
Let $\varpi$ be a weight functions that admits Assertions 1--5. Let $\sw$ be a doubling weight on $\Omega$. 
If $r > 0$, $1 \le p \le \infty$ and $f\in \Pi_n(\Omega)$, then
$$
  \| (-\fD_\varpi)^{\f r 2}f \|_{p,\sw} \le c n^r \|f\|_{p,\sw}. 
$$
\end{thm}
 
\begin{proof}
Since $\sL_n(\varpi)* f$ reproduces polynomials of degree $n$, we can write 
$$
  (-\fD_\varpi)^{\f r 2}f  =  (-\fD_\varpi)^{\f r 2} \left( L_n(\varpi) *f \right ) =  
  \int_{\Omega} f(y) L_{n}^{(r)}(\varpi; x,y)\varpi(y) \d \sm(y),
$$
where $ L_{n}^{(r)} (\varpi)$ is defined by \eqref{eq:Ln_r}. Applying Proposition \ref{prop:Tfw-norm},
we conclude that
\begin{align*}
   \left \| (-\fD)^{\f r 2}f \right\|_{p,\sw} \, & \le c \|f\|_{p,\sw} 
      \max_{x \in \Omega} \int_{\Omega} \left| L_{n}^{(r)} (\varpi; x,y)\right| (1+n\sd(x,y))^{\a(\sw)+\a(\varpi)} 
        \varpi(y) \d \sm(y).
\end{align*}
By Assertion 5, the integral in the right-hand side is bounded by 
$$
  c_\k n^{r}  \int_{\Omega} \frac{\varpi(s)}{\sqrt{\varpi(B(x,n^{-1}))}\sqrt{\varpi(B(y,n^{-1}))}
     \left(1+n\sd(x,y)\right)^{\k- \a(\sw)-\a(\varpi)}}  \d \sm(y) \le c\, n^r
$$
using \eqref{eq:CorA3} with $p=1$ and $\k > \a(\sw)+\frac{5}{4}\a(\varpi)$. 
\end{proof}

For the Laplace-Beltrami operator on the unit sphere, this Bernestin inequality is classical if $\sw$ is the Lebesgue
measure, whereas the version with the doubling weight was proved in \cite{Dai1}. For the unit ball, the inequality 
was established in \cite{X05} for the classical weight function with a different proof that applies to reflection 
invariant weight functions on the unit sphere and the unit ball. 
 
We can also prove a Nikolskii type inequality for the doubling weight on $\Omega$ that admits a 
localizable homogeneous space. 

\begin{thm} \label{thm:Nikolskii}
Let $\sw$ be a doubling weight on $\Omega$. If $0 < p < q \le \infty$ and $f \in \Pi_n(\Omega)$, then 
\begin{equation} \label{eq:Nikolskii}
    \|f\|_{q,\sw} \le c n^{ (\frac 1 p - \f 1 q) \a(\sw)} \|f\|_{p,\sw}.
\end{equation}
\end{thm}

\begin{proof}
The proof is another one that follows as in the case of the unit sphere. 
The main work lies in proving the case $q = \infty$, for which we choose a maximal $\frac{\delta}{n}$-separated 
subset of $\Omega$, so that 
\begin{align*}
 \|f\|_\infty \le c \max_{z\in \Xi} |f(z)|\, & \le c \left(\min_{z\in \Xi} B\Big(z, \frac{\delta}{n}\Big)\right)^{-\f 1 p}
   \Bigg(\sum_{z\in \Xi} \sw \left(B\Big(z,\frac{\delta}{n}\Big) \right)|f(z)|^p  \Bigg)^{\f 1 p} \\
    & \le c \|f\|_{p,\sw}\max_{z\in \Xi}  \left( B\Big(z, \frac{1}{n}\Big)\right)^{-\f 1 p}.
\end{align*}
Since $\Omega$ is compact, there is a positive number $\rho_0 > 0$ such that, for each $z \in \Xi$, 
$1= \sw(\Omega) = \sw(B(z, \rho))$ for some positive number $\rho \le \rho_0$. Let $m$ be a positive integer 
such that $2^{m-1} \le \rho n \le 2^m$. Then, by (i) of Lemma \ref{lem:doublingLem},
$$
   1 = \sw(B(z,\rho)) \le c_{L(\sw)} 2^{m\a(\sw)}  \sw(B(z,\tfrac1n)) \le c n^{\a(\sw)} \sw(B(z,\tfrac1n)),
$$
which implies that $\sw(B(z,\tfrac1n)) \le c n^{-\a(\sw)}$. Combing the two inequalities, we conclude that 
$\|f\|_\infty  \le c \|f\|_{p,\sw} n^{\frac1 p\a(\sw)}$. This proves \eqref{eq:Nikolskii} for $q = \infty$. 

The case $q < \infty$ reduces to that of $q= \infty$ since, using  \eqref{eq:Nikolskii} for $q =\infty$, 
$$
  \|f\|_{q,\sw}^q \le \|f\|_\infty^{q-p}  \|f\|_{q,\sw}^p \le c \left(\|f\|_{p,\sw} n^{\frac1 p\a(\sw)}\right)^{q-p} \|f\|_{q,\sw}^p
     = c n^{q(\f1 p - \f 1 q)\a(\sw)}\|f\|_{p,\s}^q,
$$
which is  \eqref{eq:Nikolskii} for $q < \infty$. 
\end{proof}

For the unit sphere, this inequality was established in \cite{DaiWang}. The above proof uses essentially the
same argument as seen in \cite[Theorem 5.5.1]{DaiX}. For the unit ball with the classical weight, the inequality
was proved in \cite{KPX2}.

\subsection{Proof of the main results} \label{sect:bestapprox-proof}

We are now ready to prove Theorems \ref{thm:Enf-Kfunctional} and \ref{thm:K=omega}. The proof follows 
the same procedure used in \cite{Rus, X05}, which is summarized in \cite[section 10]{DaiX}. Since most 
technical parts are essentially the same, we shall be brief. 

\medskip\noindent
{\it Proof of Theorem \ref{thm:Enf-Kfunctional}}. 

To prove the direct estimate, we use Theorems \ref{thm:nearbest} and \ref{thm:f-LnfD} to obtain 
\begin{align*}
  \left \|f- L_n\left (\varpi; f\right) \right\|_{p,\varpi}
  \,& \le  2  \|f- g\|_{p,\varpi} + \left\|g- L_n\left(\varpi; g\right) \right\|_{p,\varpi} \\
    & \le  2  \|f- g\|_{p,\varpi} + c n^{-r}\left\| (-\fD_\varpi)^{\f r 2}g \right\|_{p,\varpi}.
\end{align*}
Taking infimum over $g$ proves (i). 

To prove the inverse estimate, we choose $m$ such that $2^{m-1} \le n < 2^m$. Let us set 
$L_{2^{-1}}(\varpi)*f =0$. Then, choosing $g = L_{2^m}(\varpi)*f$, we obtain by Theorem \ref{thm:nearbest},
\begin{align*}
  K_r\left(f,n^{-1}\right)_{p,\sw}&\,  \le \|f - L_{2^m} (\varpi)*f\|_{p,\sw} + 
      2^{- (m-1) r} \left \|(-\fD_\varpi)^{\f r2}  L_{2^m} (\varpi)*f \right \|_{p,\sw} \\
    & \le  c E_n(f)_{p,\sw} +  2^{- (m-1) r} \sum_{j=0}^m  \left \|(-\fD_\varpi)^{\f r2}
        \left[ L_{2^j} (\sw)*f- L_{2^{j-1}} (\sw)*f \right]\right\|_{p,\sw}.
\end{align*}    
Applying the Bernstein inequality Theorem \ref{thm:BernsteinLB} and using the triangle inequality with
Theorem \ref{thm:nearbest}, we conclude that 
\begin{align*}
  K_r\left(f,n^{-1}\right)_{p,\varpi}&\,  \le  c E_n(f)_{p,\varpi} +  2^{- m r} 
      \sum_{j=0}^m 2^{-j r} \left \| L_{2^j} (\varpi)*f- L_{2^{j-1}} (\varpi)*f\right \|_{p,\sw} \\
    & \le c\, 2^{-m r} \sum_{j=0}^m 2^{j r} E_{2^{j-1}}(f)_{p,\varpi} 
    \le c\, n^{-r}  \sum_{k=0}^m (k+1)^{r-1}E_k(f)_{p, \varpi}. 
\end{align*}
This completes the proof. 
\qed

\medskip

The proof of the equivalence of the $K$-functional and the modulus of smoothness relies on two technical
lemmas. Recall that $\wh a$ is a cut-off function and $\mu(j)$ is the eigenvalue in \eqref{eq:LBoperator}.

\begin{lem}\label{lem:Rus1}
Let $\a \ge \b \ge -\f12$. For $r>0$,  $0 \le \t \le 3n^{-1}$, and any $\ell\in\NN$,
\begin{equation}\label{eq:Rus1a}
 \sum_{ j=0}^{2n} \Biggl| \triangle^{\ell+1}
\biggl[\biggl( \f { 1-R^{(\a,\b)}_j(\cos\t)}{ \mu(j) \t^2 }\biggr)^r \wh a \left( \f jn\right)\biggr]\Biggr| ( j+1)^\ell
   \leq c,
\end{equation}
and
\begin{equation}\label{eq:Rus1b}
 \sum_{ j=0}^{2n}\Biggl| \triangle^{\ell+1} \biggl[\biggl( \f { \mu(j) \t^2 }
     { 1-R_j^{(\a,\b)}(\cos\t)}\biggr)^r \wh a \left( \f j n\right)\biggr]\Biggr| ( j+1)^\ell \leq c,
\end{equation}
where the difference $\triangle^{\ell+1}$ is acting on $j$ and $c$ depends only on $\a, \b, \ell, r$.
\end{lem}

\begin{lem}\label{lem:Rus2}
Let $\a \ge \b \ge -\f12$ and let $\ell = \lceil  \a \rceil$. If $r>0$, $k^{-1}\leq \t \leq \f \pi2$,  $j\in \NN_0$
and $0\leq j\leq \ell+1$,  then for any $m\in\NN$,
\begin{equation}\label{eq:Rus2}
 \biggl | \triangle^j \biggl( \f { ( 1-(1-R_k^{(\a,\b)}(\cos\t))^r)^{m+\ell+1}}{(1-R_k^{(\a,\b)}(\cos\t))^r}\biggr) \biggr| \leq 
    c_{m,r} (k\t)^{-m\a}\t^j.\end{equation}
\end{lem}

For $\a = \b = \l$, these two lemma are Lemma 10.4.3 and Lemma 10.4.4 in \cite{DaiX}. The proof relies on
properties of the Jacobi polynomials and holds readily for $\a \ge \b \ge -\f12$. 

\medskip\noindent
{\it Proof of Theorem \ref{thm:K=omega}}. 

Choose $n$ such that $\t^{-1} \le n \le 3 \t^{-1}$. Let $g$ be a polynomial in $\Pi_n(\Omega)$. Then 
so is $\left(I - S_{\t,\varpi}\right)^{r/2} g$ in $\Pi_n(\Omega)$ 
by \eqref{eq:Stheta}. Let $\wh a$ be a cut-off function of type (a), so that $L_n(\varpi)$ reproduces polynomials
of degree $n$. Using \eqref{eq:fracDiff}, it follows readily that 
$$
   \left(I - S_{\t,\varpi}\right)^{r/2} g = \t^r \sum_{j=0}^{2n} \biggl(\f {1-R^{(\a,\b)}_j (\cos\t)}{ \mu(j)\t^2}\biggr)^{r/2}
          \wh a \Big(\f j n\Big) \proj_j\left(\varpi;  (-\fD_\varpi)^{r/2}g\right).
$$
Following the proof of Theorem \ref{thm:f-LnfD} to take a summation by parts and use the boundedness of
the Ces\`aro means, we can use \eqref{eq:Rus1a} to obtain 
\begin{align}\label{eq:I-S<D}
  \left\|\left(I - S_{\t,\varpi}\right)^{r/2} g \right\|_{p,\varpi} \le c n^{-r} \left\| (-\fD_{\varpi})^{r/2} g\right\|_{p,\varpi}. 
\end{align}
Using this inequality with $g = L_{\lfloor n/2 \rfloor} (\varpi)* f$ and the Jackson estimate in 
Theorem \ref{thm:Enf-Kfunctional}, we obtain
\begin{align*}
 \left \|( I-S_{\t,\varpi})^{r/2} f \right \|_{p,\varpi}  & \le c \left \|f-L_{\lfloor \f n 2 \rfloor} (\varpi)* f\right \|_{p, \varpi} +
       \left \|( I-S_{\t,\varpi})^{r/2} L_{\lfloor \f n 2 \rfloor} (\varpi)* f \right \|_{p, \varpi} \\
 & \le c K_r(f; n^{-1})_{p,\varpi} + c n^{-r} \left\|(-\fD)^{r/2} L_{\lfloor \f n 2 \rfloor} (\varpi)* f\right\|_{p,\varpi}.
\end{align*}
By the triangle inequality and applying the Bernstein inequality on $L_n(\varpi)*(f-g)$, we obtain
\begin{align*}
 n^{-r} \left\|(-\fD)^{r/2} L_{\lfloor \f n 2 \rfloor} (\varpi)* f\right\|_{p,\varpi}
 & \le c  \left\|f-g\right\|_{p,\varpi}+ n^{-r} \left\|(-\fD)^{r/2} L_{\lfloor \f n 2 \rfloor} (\varpi)* g\right\|_{p,\varpi} \\
 & \le c  \left( \left\|f-g\right\|_{p,\varpi}+ n^{-r} \left\|(-\fD)^{r/2} g\right\|_{p,\varpi}\right),
\end{align*}
where we used, by \eqref{eq:LBoperator}, $(-\fD)^{r/2} L_{\lfloor \f n 2 \rfloor} (\varpi)* g
 =  L_{\lfloor \f n 2 \rfloor} (\varpi)* (-\fD)^{r/2}g$ in the second step. Hence, taking infimum over $g$,
it follows from the above two inequalities and $\t \sim n^{-1}$ that
\begin{equation*} 
\left \|(I-S_{\t,\varpi})^{r/2} f \right \|_{p, \varpi} \leq c K_{r}(f, \t)_{p,\varpi},
\end{equation*}
from which the right-hand inequality of \eqref{eq:K=omega} follows.  

In the other direction, we can follow the proof of inequality \eqref{eq:I-S<D} and use \eqref{eq:Rus1b} instead of
\eqref{eq:Rus1a}, to establish 
\begin{align*}
 n^{-r} \left\| (-\fD_{\varpi})^{r/2} L_n(\varpi)*f \right\|_{p,\varpi} \le c  \left\|\left(I - S_{\t,\varpi}\right)^{r/2} f \right\|_{p,\varpi}. 
\end{align*}
Hence, in order to prove the left-hand inequality of \eqref{eq:K=omega}, it suffices to prove that
\begin{align}\label{eq:I-Ln<I-S}
\left\|f- L_n(\varpi)*f \right\|_{p,\varpi} \le c  \left\|\left(I - S_{\t,\varpi}\right)^{r/2} f \right\|_{p,\varpi}.
\end{align}
Using $(1-r)^{-1} = \sum_{i=0}^{m+\ell} r^i + r^{m+\ell}(1-r)^{-1}$ with $r = 1- \big(1-R^{(\a,\b)}_j (\cos\t)\big)^{r/2}$, 
one can write, as seen in \cite[p. 253]{DaiX}, that 
\begin{align*}
  f - L_n(\varpi)* f &\, =  \sum_{i=0}^{m+\ell} (I-L_n(\varpi)) * \left(I-(I-S_{\t,\varpi})^{r/2}\right)^i F\\
     & + \sum_{j=n}^\infty \left(1-\wh a \left(\f j n \right) \right)
      \frac{ \big(1- \big(1-R^{(\a,\b)}_j (\cos\t)\big)^{r/2}\big)^{m+\ell+1}F}{(1-R_j^{(\a,\b)}(\cos \t))^{r/2}} \proj_j(\varpi; F), 
\end{align*}
where $F  = \left(I - S_{\t,\varpi}\right)^{r/2} f$. While the first term is bounded by $\|F\|_{p,\varpi}$ by the 
boundedness of $L_n(\varpi)*f$ and $(I-(I-S_{\t,\varpi})^{r/2})F$, the second one can be shown to be bounded
by $\|F\|_{p,\varpi}$ by using \eqref{eq:Rus2} with the summation by parts and the boundedness of Ces\`aro 
means that we have used several times. This proves \eqref{eq:I-Ln<I-S} and completes the proof. 
\qed

%%%%%%%%%%%%
%%     section 4     %%    
%%%%%%%%%%%%
\section{Homogeneous space on conic surfaces}\label{sec:coneV0}
\setcounter{equation}{0}

In this section we work in the setting of homogeneous space on the conic surface
$$
     \VV_{0}^{d+1}= \{(x,t): \|x\| = t, \, x \in \RR^d, \, 0 \le t \le 1\}. 
$$
We shall verify that the framework in the previous two sections is applicable on this domain for the
weight function $t^{-1} (1-t)^\g$, which has a singularity at the apex. The verification is highly non-trivial 
because of new phenomena and obstacles encountered. While the structure of orthogonal polynomials 
on the conic surface shares characteristic features of spherical harmonics, the conic surface is markedly 
different from that of the unit sphere because of its apex and its boundary.  

Our first task is to understand, in the first subsection, the intrinsic distance function on the conic surface, 
which turns out to be incomparable to the Euclidean distance around the apex. In the second subsection, 
we show that the Jacobi weight function $\sw_{\b,\g}(t) = t^\b (1-t)^\g$ is a doubling weight with 
respect to the intrinsic distance. The orthogonal structure with respect to the Jacobi weight is reviewed
in the third subsection, which is used to verify Assertions 1-3 of the highly localized kernels for the weight 
$\sw_{-1,\g}$ in the fourth subsection. Construction of $\ve$-separated set of $\VV_0^{d+1}$ is 
provided in the fifth subsection and used to state the Marcinkiewicz-Zygmund inequality. Assertion 4
is verified in the sixth subsection, which ensures that the positive cubature rules and the tight localized
frames can both be stated for the conic surface. In the seventh subsection, Assertion 5 is verified and
the characterization of the best approximation by polynomials is stated.  

\subsection{Distance on the surface of the cone}
Our first task is to define an appropriate distance function on the surface of the cone. Unlike the sphere, 
the surface $\VV_0^{d+1}$ has a boundary at $t =1$ and a singularity at the apex $t=0$. Our
distance function should measure the distance between points near the boundary or the apex 
and the distance between interior points differently. 
This is a well-known phenomenon as we have already seen for the interval $[-1,1]$. More generally, 
the distance on the interval $[a,b]$ is given by the change of variables $x \in [a,b] \mapsto y \in [-1,1]$,  
$$
  \sd_{[a,b]}(x_1,x_2) = \frac{b-a}{2} \arccos \left( y_1 y_2 + \sqrt{1-y_1^2}\sqrt{1-y_2^2}\right),
$$
where $y_i = -1+2 \frac{x_i-a}{b-a}$. In particular, the distance function for $[0,1]$ is given by
\begin{align*}
  \sd_{[0,1]}(x_1,x_2) \, & = \frac{1}{2} \arccos \left( (2x_1-1) (2y_1-1) + 4 \sqrt{x_1(1-x_1)} \sqrt{y_1(1-y_1)}\right)\\
    & =  \arccos \left(\sqrt{x_1 x_2} + \sqrt{(1-x_1)(1-x_2)} \right),
\end{align*}
where the second identity follows from $\arccos (\a) = \frac12 \arccos(2\a^2-1)$ . In particular,  
setting $x_i = \cos^2 \frac{\t_i}{2} \in [0,1]$, $0\le \t_i \le \pi$, we obtain $\sd_{[0,1]}(x_1,x_2) = \frac12|\t_1-\t_2|$. 

\begin{defn}
For $(x,t)$ and $(y,s)$ on $\VV_0^{d+1}$, define 
\begin{equation}\label{eq:distV0}
  \sd_{\VV_0} ((x,t), (y,s)): =  \arccos \left(\sqrt{\frac{\la x,y\ra + t s}{2}} + \sqrt{1-t}\sqrt{1-s}\right).
\end{equation}
\end{defn}

\begin{prop}
The function $\sd_{\VV_0^{d+1}} (\cdot, \cdot)=\sd_{\VV_0}(\cdot,\cdot)$ defines a distance on the  
surface of the cone $\VV_0^{d+1}$. 
\end{prop}

\begin{proof}
Evidently $\sd_{\VV_0}(\cdot,\cdot)$ is symmetric. Since $\|x\| = t$ and $\|y\| = s$, it follows readily that
$0 \le \sqrt{\frac{\la x,y\ra + t s}{2}} + \sqrt{1-t}\sqrt{1-s} \le 1$, so that $\sd_{\VV_0} ((x,t), (y,s)) \ge 0$
and, furthermore, $\sd_{\VV_0} ((x,t), (x,t))= \arccos 1 =0$. Hence, we only need to prove that it satisfies
the triangle inequality. 

For $d =2$, we write $(x,t) = (x_1,x_2,t) \in  \VV_0^3$  with $x_1^2+x_2^2 = t^2$. For $(x,t), (y,s) \in \VV_0^3$,
it is easy to verify the identity
$$
   \left( \sqrt{ (t+x_1)(s+y_1) } + \mathrm{sign}(x_2 y_2) \sqrt{ (t-x_1)(s-y_1)}\right)^2 
     = \frac{t s + x_1 y_1 + x_2 y_2}{2}.
$$
Hence, setting $z_{x,t} = (\sqrt{t+x_1}, \mathrm{sign}(x_2) \sqrt{t-x_1}, \sqrt{1-t}$) and define $z_{y,s}$ similarly,
then they are elements in $\SS^2$ and it follows that 
$$
 \sd_{\VV_0} ((x,t), (y,s))  = \arccos  (\la z_{x,t}, z_{y,s}\ra) = \sd_{\SS^2} (z_{x,t}, z_{y,s}), 
$$
where $\sd_{\SS^2}(\cdot,\cdot)$ is the geodesic distance of the unit sphere $\SS^2$. In particular, 
$ \sd_{\VV_0} (\cdot,\cdot)$ satisfies the triangle inequality. For $d> 2$, given three distinct point in 
$\VV_0^{d+1}$, written as $(t_i \xi_i, t_i)$, $1 \le i \le 3$, where $\xi \in \sph$, we can find a rotation 
in $\RR^d$ so that $\xi_i = (\eta_i, 0)$ with $\eta_i \in \SS^2$. Hence, the triangle inequality for $d > 2$
follows from the triangle inequality for $d =2$. This completes the proof. 
\end{proof}

The distance function $\sd_{\VV_0}(\cdot,\cdot)$ is closely related to the distance function $\sd_{[0,1]}(\cdot,\cdot)$ 
of the interval $[0,1]$ and the geodesic distance $\sd_{\SS}(\cdot,\cdot)$ of the unit sphere $\sph$. 

\begin{prop}\label{eq:cos-dist}
For $d \ge 2$ and $(x,t), (y,s) \in \VV_0^{d+1}$, write $x = t\xi$ and $y = s \eta$ with $\xi,\eta\in \sph$. Then
\begin{equation}\label{eq:d=d+d}
   1- \cos \sd_{\VV_0} ((x,t), (y,s)) =1-\cos \sd_{[0,1]}(t,s) + 
      \sqrt{t}\sqrt{s} \left[1-\cos \left(\tfrac{1}{2} \sd_{\SS}(\xi,\eta) \right)\right].
\end{equation}
In particular, 
\begin{equation} \label{eq:d2=d2+d2}
   c_1 \sd_{\VV_0} ((x,t), (y,s)) \le \sd_{[0,1]}(t,s)  + (t s )^{\f14} \sd_{\SS}(\xi,\eta) \le  c_2 \sd_{\VV_0} ((x,t), (y,s)).
\end{equation}
\end{prop}

\begin{proof}
Using $\arccos (\a) = \frac12 \arccos(2\a^2-1)$,  we deduce  
$$
 \sqrt{\frac{ts +\la x, y\ra}{2}} = \sqrt{t s} \sqrt{\frac{1+\la \xi, \eta\ra}{2}} = \sqrt{ts} \cos \left(\tfrac{1}{2} \arccos \la \xi, \eta\ra \right).
$$
Consequently, in terms of the geodesic distance on the unit sphere, we can write 
\begin{equation*} %\label{eq:distV0B}
   \sd_{\VV_0} ((x,t), (y,s)) = \arccos \left[\sqrt{t}\sqrt{s} \cos \left(\tfrac{1}{2} \sd_{\SS}(\xi,\eta) \right)
       + \sqrt{1-t}\sqrt{1-s} \right].
\end{equation*}
In particular, it follows that
$$
  1- \cos  \sd_{\VV_0} ((x,t), (y,s))  = 1- \sqrt{t}\sqrt{s}- \sqrt{1-t}\sqrt{1-s} +\sqrt{t}\sqrt{s}
         \left[1- \cos \left(\tfrac{1}{2} \sd_{\SS}(\xi,\eta) \right)\right],
$$
which is the identity \eqref{eq:d=d+d}. From this identity, \eqref{eq:d2=d2+d2} follows from 
$1-\cos \t = 2 \sin^2\frac{\t}{2}$, $\frac{1}{\pi} \t \le \sin \f{\t}{2} \le \f{\t}{2}$ for $0 \le \t \le \pi$, and 
$(a+b)^2/2 \le a^2+b^2 \le (a+b)^2$ for $a,b \ge 0$. 
\end{proof}

The line segment from the apex to a point $(\xi,1)$, $\xi\in \sph$, on the top boundary of the cone $\VV_0^{d+1}$ 
can be parametrized by $l_\xi= \{(t\xi,t): 0 \le t \le 1\}$. For two points $(t\xi,t)$ and $(s \xi, s)$ on $l_\xi$, the
identity \eqref{eq:d=d+d} shows that 
$$
 \sd_{\VV_0} ((t\xi,t), (s\xi,s))  = \tfrac{|\t-\phi|}{2} =  \sd_{[0,1]}(t,s)
$$
if $t = \cos^2 \frac{\t}{2}$ and $s =  \cos^2 \frac{\phi}{2}$. Moreover, the top boundary of $\VV_0^{d+1}$ is the 
unit sphere $\sph$, or $\{(\xi,1): \xi \in \sph\}$. For $(\xi_1,1)$ and $(\xi_2,1)$ on this boundary, 
\eqref{eq:d=d+d} gives
$$
   \sd_{\VV_0}( (\xi_1,1),(\xi_2,1)) =  \frac12 \sd_{\SS}(\xi_1,\xi_2), 
$$
or half of the geodesic distance on the unit sphere. 

\begin{rem}\label{rem:distV0}
It is well-known that the geodesic distance $\sd_\SS(\cdot,\cdot)$ of $\sph$ is proportional to the Euclidean 
distance; that is, $\|\xi - \eta\| \sim \sd_{\SS}(\xi,\eta)$ for all $\xi,\eta \in \sph$. For the surface of the cone 
$\VV_0^{d+1}$, however, this is no longer true when the points are near the apex. Indeed, for $(x,t), (y,s) 
\in \VV_0^{d+1}$ with $x = t\xi$, $y =s \eta$, $\xi,\eta \in \sph$, 
$$
  \|(x,t) - (y,s)\|^2 = 2 (t-s)^2 + 2 t s \big(1- \la \xi,\eta \ra\big) =2 (t-s)^2 + 2 t s (1-\cos \sd_{\SS}(\xi,\eta)).
$$
In particular, if $t = s$, we see that $\|(t\xi,t) - (t\eta,t)\| \sim t \sd_{\SS}(\xi,\eta)$, whereas we have
$d_{\VV_0} \big( (t \xi,t),(t\eta,t) \big) \sim \sqrt{t} \sd_{\SS}(\xi,\eta)$ by \eqref{eq:d2=d2+d2}. Hence
the two distances are not compatible when $t$ is small. 
\end{rem}

We will also need the following lemma in the estimate of the kernels. 

\begin{lem} \label{lem:|s-t|}
For $(x,t), (y,s) \in \VV_0^{d+1}$, 
$$
  \big| \sqrt{t} - \sqrt{s} \big|\le \sd_{\VV_0} ((x,t), (y,s)) \quad \hbox{and} \quad 
      \big| \sqrt{1-t} - \sqrt{1-s} \big| \le \sd_{\VV_0} ((x,t), (y,s)).
$$
\end{lem} 

\begin{proof}
Let $t = \cos^2 \frac{\t}{2}$ and $s = \cos^2 \frac{\phi}2$, $0 \le \t, \phi \le \pi$. Since $|\la x,y\ra| \le ts$, we obtain
$$
\cos \sd_{\VV_0} ((x,t), (y,s)) \le \sqrt{t s} + \sqrt{1-t}\sqrt{1-s} = \cos \tfrac{\t-\phi}2,  
$$
so that $\sd_{\VV_0} ((x,t), (y,s))  \ge  \frac12 |\t-\phi|$. Elementary trigonometric identities shows that
$$
 \big | \sqrt{t} - \sqrt{s} \big| = \big |\cos \tfrac{\t}{2} - \cos \tfrac{\phi}{2} \big | \le 2 \sin \tfrac{|\t-\phi|}4 
   \le \tfrac{|\t-\phi|}2 \le \sd_{\VV_0} ((x,t), (y,s)).  
$$
The inequality for $\big | \sqrt{1-t} - \sqrt{1-s} \big|$ follows from $\big |\sin \tfrac{\t}{2} - \sin \tfrac{\phi}{2} \big | 
\le 2 \sin \tfrac{|\t-\phi|}4$. 
\end{proof}

\subsection{A family of doubling weights}
For the conic  surface, balls are conic  caps. For $r > 0$ and $(x,t)$ on $\VV_0^{d+1}$, we denote the conic cap 
centered at $(x,t)$ with radius $r$ by 
$$
      \sc((x,t), r): = \left\{ (y,s) \in \VV_0^{d+1}: \sd_{\VV_0} \big((x,t),(y,s)\big)\le r \right\}.
$$   
A weight function $\sw$ is a doubling weight if it satisfies 
$$
   \sw\big(\sc((x,t), 2 r)\big) \le L \, \sw\big(\sc((x,t), r)\big), \quad r >0.
$$
In comparison with the spherical cap on $\sph$, the geometry of $\sc(x,t)$ is more complicated. 
Denote the surface measure on $\sph$ by $\d \s_\SS$. 

\begin{lem}\label{lem:cone-cap}
For $r > 0$, $t, s \in [0,1]$, define $\tau_r(t,s) = (\cos r - \sqrt{1-t}\sqrt{1-s} )/\sqrt{ts}$ and $\t_r(t,s) = \arccos \tau_r(t,s)$. 
Then, for $(x,t) \in \VV_0^{d+1}$ with $x = t \xi$, $\xi \in \sph$,  
\begin{align*}
   \sw \big(\sc((x,t), r)\big)&  = \int_{\sd_{[0,1]}(t, s)\le r} s^{d-1} 
        \int_{\sd_{\SS}(\xi,\eta) \le \tfrac12 \t_r(t,s)} \sw(s\eta, s) \d \s_{\SS}(\eta)\d s.
\end{align*}
\end{lem}

\begin{proof}
 From $\sd_{\VV_0}((x,t),(y,s)) \le r$, we obtain $\sd_{[0,1]}(t, s) \le r$ by 
\eqref{eq:d=d+d} and, with $x = t \xi$ and $y = s \eta$, it follows from \eqref{eq:distV0} that 
$$
  \sd_{\SS} (\xi,\eta) \le \arccos \left(2 [\tau_r(t,s)]^2 -1\right) = \tfrac12 \arccos \tau_r(t,s) = \tfrac12 \t_r(t,s).  
$$
Hence, the stated identity follows from $\d \s(y,s) = s^{d-1} \d \s_\SS(\eta) \d s$. 
\end{proof}

For $\b > -d$ and $\g > -1$, consider the Jacobi weight function defined on the cone
$$
  \sw_{\b,\g} (t) = t^\b (1-t)^\g, \quad 0 < t <1.  
$$
Let $\bs_{\b,\g}$ be the normalization constant so that $\bs_{\b,\g} \sw_{\b,\g}$ has unit integral on 
$\VV_0^{d+1}$. Setting $y = s \eta$, $\eta \in \sph$, then 
$$
  \bs_{\b,\g}^{-1} = \int_0^1 s^{d+\b-1}(1-s)^\g \d s \int_{\sph}\d \s_\sph (\xi)
        =  \omega_d \frac{\Gamma(\b+d) \Gamma(\g+1)}{\Gamma(\b+\g+d+1)},
$$
where $\omega_d$ is the surface are of $\sph$.

\begin{prop}\label{prop:capV0}
Let $r > 0$ and $(x,t) \in \VV_0^{d+1}$. Then for $\b > - d$ and $\g > -1$, 
\begin{align}\label{eq:capV0}
 \sw_{\b,\g}\big(\sc((x,t), r)\big):=  \, & \bs_{\b,\g} \int_{ \sc((x,t), r)} \sw_{\b,\g}(s) \d \s(y,s) \\
     \sim &\, r^d (t+ r^2)^{\b+\f{d}{2}} (1-t+ r^2)^{\g+\f12}. \notag
\end{align}
In particular, $\sw_{\b,\g}$ is a doubling weight and the doubling index $\a( \sw_{\b,\g})$, defined in 
\eqref{eq:alpha(w)}, is give by $\a(\sw_{\b,\g}) = d + 2 \max\{0, \b+\f{d}2\} + 2 \max\{0,\g+\f12\}$.
\end{prop}

\begin{proof} 
Without loss of generality, we assume $r$ is bounded by a small positive number $r \le \delta$; 
for example, $\delta = \f{\pi}{12}$ will do. By rotation symmetry, we could choose $x = t e_1$, where 
$e_1=(1,0,\ldots,0)$. Then, by Lemma \ref{lem:cone-cap},
\begin{align*}
   \sw_{\b,\g}\big(\sc((x,t), r)\big) = \o_{d-1} \int_{\sd_{[0,1]}(t, s)\le r}  s^{d-1} \sw_{\b,\g}(s) \int_0^{\tfrac12 \t_r(t,s)}
            (\sin \t)^{d-2} \d \t \d s,
\end{align*}
where we have used the identity (cf. \cite[(A.5.1)]{DaiX}) 
\begin{equation} \label{eq:intSS}
  \int_{\sph} g(\la \xi,\eta\ra) d\s(\eta) %= \o_{d-1} \int_{-1}^1 g(u) (1-u^2)^{\f{d-3}{2}} \d u
      = \o_{d-1} \int_0^\pi g (\cos \t) (\sin\t)^{d-2} \d \t
\end{equation} 
with $\o_{d-1}$ being the surface are of $\SS^{d-2}$. Since $\t \sim \sin \t \sim \sqrt{1-\cos \t}$,
it follows that 
\begin{align} \label{eq:sw(cap)}
 \sw_{\b,\g}\big(\sc((x,t), r)\big)\, &\sim \int_{\sd_{[0,1]}(t, s)\le r}  s^{d-1} \sw_{\b,\g}(s) 
      \big(1-\tau_r(t,s) \big)^{\f{d-1}2} \d s. 
\end{align}
By its definition, $\tau_r(t,s) \ge 0$ and, furthermore, $\tau_r(t,s) \le 1$ since we can write
\begin{align} \label{eq:tau_rV0}
   1-\tau_r(t,s) = \frac{\cos \sd_{[0,1]}(t,s) - \cos r}{\sqrt{t}\sqrt{s}}.
\end{align}
We need to consider three cases.

\medskip\noindent
{\it Case 1}. Assume $3 r^2 \le t \le 1- 3 r^2$. By Lemma \ref{lem:|s-t|}, this implies that 
$s \sim t + r^2$ and $1-s \sim 1-t + r^2$, which allows us to conclude that 
\begin{align*}
 \sw_{\b,\g}\big(\sc((x,t), r)\big) & \, \sim (t+ r^2)^{\b+\f{d}{2}} (1-t+ r^2)^{\g+\f12} \\
     & \times \int_{\sd_{[0,1]}(t, s)\le r} \big(\cos (\sd_{[0,1]} (t,s)) - \cos r \big)^{\f{d-1}2} 
    \frac{\d s}{\sqrt{s (1-s)}}. 
\end{align*}
Setting $t = \sin^2 \frac{\t}{2}$ and $s = \sin^2 \frac{\phi}{2}$ so that $\sd_{[0,1]} (t,s) = |\t-\phi|/2$
and the last integral is easily seen to be
$$
  \int_{|\t-\phi|\le 2 r} \big(\cos \tfrac{\t-\phi}{2} - \cos r \big)^{\f{d-1}2} \d \phi
    =  c \int_{|\zeta|\le 2 r} \big(\sin \tfrac{\zeta-r}{2} \sin \tfrac{\zeta+r}{2} \big)^{\f{d-1}2} \d \zeta  \sim r^d.
$$
This completes the proof of the first case. 

\medskip\noindent
{\it Case 2.} $0 \le t \le 3 r^2$. For $(y,s) \in \sc((x,t),r)$, we also have $\sqrt{s} \le \sqrt{t}+ r \le (1+ \sqrt{3})r$ by 
Lemma \ref{lem:|s-t|}, which shows, in particular, that $s \le c (t +r^2)$. Evidently $1-s \sim 1-t \sim 1$ in this
case. Furthermore, let $t = \sin^2 \frac{\t}{2}$ and $s = \sin^2 \frac{\phi}{2}$; then 
$$
|s - t| = \left|\sin \frac{\t-\phi}{2} \sin \frac{\t+\phi}{2} \right| \le c  r^2  
$$
since $|\t-\phi| \le 2 \sd_{[0,1]}(t,s)\le r$ and $\t \le c r$. Now, we have a trivial upper bound $1-\tau_r(t,s) \le 2$, 
which leads to, by \eqref{eq:sw(cap)},
\begin{align*}
 \sw_{\b,\g}\big(\sc((x,t), r)\big) \le c  \int_{\sd_{[0,1]}(t, s)\le r/2 }  s^{d+ \b -1} \d s 
       \sim \int_0^{r^2} \phi^{d+ \b -1} d\phi  \sim r^{2\b+ 2 d},
\end{align*}
which proves the upper bound in \eqref{eq:capV0}. For the lower bound, we consider a subset of $\sc((x,t),r)$ 
with $\d_{[0,1]}(t,s) \le r/2$. Using the upper bound of $s$ and $t$, we then deduce 
$$
  1-\tau_r(t,s)= \frac{\cos d_{[0,1]}(t,s) - \cos r}{\sqrt{t}\sqrt{s}} \ge \frac{\cos \tfrac{r}{2} - \cos r}{ (3+\sqrt{3}) r^2} \ge \frac2{\pi^2},
$$
where in the last step we have used the monotonicity of the function over $0 \le r \le \pi/12$, which shows then
\begin{align*}
 \sw_{\b,\g}\big(\sc((x,t), r)\big)\,& \ge c \int_{\sd_{[0,1]}(t, s)\le r/2 }  s^{d+ \b -1} \d s 
     \sim r^{2\b+ 2 d}. 
\end{align*}
This completes the proof of this second case.  

\medskip\noindent
{\it Case 3.} $1- t \le 3 r^2$. In this case, we clearly have $s \sim t  \sim 1$ for $(y,s) \in \sc((x,t),r)$. Since
$\sd_{[0,1]}(t,s) = \sd_{[0,1]}(1-t,1-s)$, changing variable $s\mapsto 1-s$ in  \eqref{eq:sw(cap)}, we obtain
\begin{align*}
 \sw_{\b,\g}\big(\sc((x,t), r)\big) \sim c  \int_{\sd_{[0,1]}(t, s)\le r}  (1-s)^\g \d s 
        = c \int_{\sd_{[0,1]}(1-t, s)\le r} s^\g \d s, 
\end{align*}
where the last integral can be estimated as in Case 2. This completes the proof of \eqref{eq:capV0}.
\end{proof}

It is worthwhile to mention that the proof relies on the geometry of $\sc((x,t),r)$ when $t \le c r^2$, which
we describe in the following remark. 

\begin{rem}
By \eqref{eq:cos-dist}, $(y,s) \in \sc((x,t),r)$ in equivalent to 
$$
   2 \sin^2 \frac{\sd_{\SS}(\xi,\eta)}{2} = \frac{\cos d_{[0,1]}(t,s) - \cos d_{\VV_0}\big((x,t), (y,s)\big)}{\sqrt{t}\sqrt{s} }
       \le \frac{\cos d_{[0,1]}(t,s) - \cos r}{\sqrt{t}\sqrt{s}}. 
$$
If $t \le c r^2$, then the proof of the Case 2 shows, by $\sin \t \le \t$, that the above inequality holds 
whenever $\sd_{\SS}(\xi,\eta) \le \frac4{\pi^2}$. In other words, for any $s$ that satisfies $\sd_{[0,1]}(t,s) \le r$,
the set $\sc((x,t), r)$ contains a large spherical cap $\{\eta: \sd_{\SS} (\xi,\eta) \le \frac4{\pi^2}\}$. 
\end{rem}

\begin{cor}
For $d\ge 2$, $\b > -d$ and $\g > -1$, the space $(\VV_0^{d+1},  \sw_{\b,\g}, \sd_{\VV_0})$ is a 
homogeneous space. 
\end{cor}

For convenience, we will introduce and use the function $\sw_{\b,\g,d}(n;t)$ defined by 
\begin{equation}\label{eq:w_bg}
  \sw_{\b,\g,d} (n; t):=  n^{d} \sw_{\b,\g}\big(\sc((x,t), n^{-1})\big)
        = \big(t+n^{-2}\big)^{\b+\f{d}{2}} \big(1-t+n^{-2}\big)^{\g+\f12}. 
\end{equation}
For any $\xi$ in the unit sphere $\SS^{m-1}$, a spherical cap $\sc(\xi,r)$ has $\s (\sc(\xi,r)) = c r^{m-1}$ 
so that, for a doubling weight $w$ on the unit sphere $\SS^{m-1}$, the function 
$w(n;\xi) =  n^{m-1} \s (\sc(\xi,\tfrac 1 n))$ is an approximation to $w(\xi)$ by the Lebesgue differentiation 
theorem. The function $\sw_{\b,\g,d}(n; t)$ in 
\eqref{eq:w_bg}, however, is not an approximation to $\sw_{\b,\g}(t)$ on the conic  surface since for the 
Lebesgue measure $\d \s$ on $\VV_0^{d+1}$,
$$
  \s \big(\sc((x,t), n^{-1})\big) \sim n^{-d} (t+ n^{-2})^{\f d 2} (1 - t + n^{-2})^{\f12} 
$$
by \eqref{eq:capV0}, whereas it is the ratio 
$$
  \sw_n(x,t) = \frac{\sw(\sc((x,t), n^{-1}))}{\s(\sc((x,t), n^{-1}))}
$$ 
that provides an approximation to $\sw(t)$ on $\VV_0^{d+1}$. 

For the unit sphere, it can be verified that $w_n(\xi) = w(n;\xi)$ is a doubling weight on $\SS^{m-1}$ whenever 
$w$ is, and this property has been used to show, for example, that $\|f\|_{p,w} \sim \|f\|_{p, w_n}$ for all 
polynomials of degree at most $n$. For a doubling measure $\sw$ on the conic surface, it is not clear, however,
if $\sw_n$ defined above is itself a doubling weight on $\VV_0^{d+1}$. 

\subsection{Orthogonal polynomials on the conic surface}\label{sec:OPconicSurface}
Orthogonal structure with respect to $\sw_{\b,\g}$ on the conic surface was
studied in \cite{X20a}. For $\b > - d$ and $\g > -1$, define the inner product 
$$
\la f, g\ra_{\sw} =\bs_{\b,\g} \int_{\VV_0^{d+1}} f(x,t) g(x,t) \sw_{\b,\g} \d \s(x,t),
$$ 
where $\d \s$ denote the surface measure on $\VV_0^{d+1}$, which is well-defined for all polynomials 
restricted on the conic surface. Let $\CV_n(\VV_0^{d+1},\sw_{\b,\g})$ be the space of orthogonal
polynomials of degree $n$. Since $\VV_0^{d+1}$ is a quadratic surface in $\RR^{d+1}$, the dimension 
of the space $\CV_n(\VV_0^{d+1}, \sw_{\b,\g})$ is the same as the dimension of the spherical harmonics 
of degree $n$ on $\SS^d$; that is, $\dim \CV_0(\VV_0^{d+1},\sw_{\b,\g}) =1$ and 
$$
   \dim \CV_n(\VV_0^{d+1},\sw_{\b,\g})  = \binom{n+d-1}{n}+\binom{n+d-2}{n-1},\quad n=1,2,3,\ldots.
$$
Furthermore, let $\Pi_n(\VV_0^{d+1})$ denote the space of polynomials of degree at most $n$ restricted 
on $\VV_0^{d+1}$, then it is the union of $\CV_n(\VV_0^{d+1},\sw_{\b,\g})$ and 
$$
    \dim \Pi_n(\VV_0^{d+1}) = \binom{n+d}{n}+\binom{n+d-1}{n-1}.
$$

An orthogonal basis of $\CV_n(\VV_0^{d+1}, \varphi_{\b,\g})$ can be given in terms of the Jacobi polynomials
and spherical harmonics. Let $\CH_m(\sph)$ be the space of spherical harmonics of degree $m$ in $d$
variables. Let $\{Y_\ell^m: 1 \le \ell \le \dim \CH_m(\sph)\}$ denote an orthonormal basis of $\CH_m(\sph)$. 
Then the polynomials
\begin{equation*} %\label{eq:sfOPbasis}
  \sS_{m, \ell}^n (x,t) = P_{n-m}^{(2m + \b + d-1,\g)} (1-2t) Y_\ell^m (x), \quad 0 \le m \le n, \,\, 
      1 \le \ell \le \dim \CH_m(\sph),
\end{equation*}
consist of an orthogonal basis of $\CV_n(\VV_0^{d+1}, \varphi_{\b,\g})$. More precisely, 
$$
  \la  \sS_{m, \ell}^n,  \sS_{m', \ell'}^{n'} \ra_{\sw_{b,\g}} = \sH_{m,n}^{\b,\g} \delta_{n,n'} \delta_{m,m'} \delta_{\ell,\ell'}
$$
where the norm $\sH_{m,n}^{\b,\g}$ of $\sS_{m,\ell}^n$ is given by
\begin{equation*}% \label{eq:sfOPNorm}
  \sH_{m,n}^{\b,\g} = \frac{c_{\b+d-1,\g}}{c_{2m + \b+d-1,\g}} h_{n-m}^{(2m+\b+d-1,\g)},
\end{equation*}
where $c_{\a,\b}$ is the normalization constant in \eqref{eq:c_ab} and $h_m^{(\a,\b)}$ is the norm square of
the Jacobi polynomial. 
We call $S_{m, \ell}^n$ the Jacobi polynomials on the conic  surface. The reproducing kernel of the space 
$\CV_n(\VV_0^{d+1}, \sw_{\b,\g})$ is denoted by $\sP_n(\sw_{\b,\g};\cdot,\cdot)$, which can be written 
in terms of the above basis, 
$$
\sP_n\big(\sw_{\b,\g}; (x,t),(y,s) \big) = \sum_{m=0}^n \sum_{k=1}^{\dim \CH_m^d}
    \frac{  \sS_{m, \ell}^n(x,t)  \sS_{m, \ell}^n(y,s)}{\sH_{m,n}^{\b,\g}},
$$
and it is the kernel of the orthogonal projection operator $\proj_n(\sw_{\b,\g}): L^2(\VV_0^{d+1},\sw_{\b,\g}) \to 
\CV_n(\VV_0^{d+1}, \varphi_{\b,\g})$,  
$$
\proj_n(\sw_{\b,\g};f) = \int_{\VV_0^{d+1}} f(y,s) \sP_n\big(\sw_{\b,\g}; \,\cdot, (y,s) \big)  \sw_{\b,\g}(s) \d\s(y,s).
$$

The case $\b = -1$ turns out to be the most interesting case, for which there is a second order differential 
operator that has orthogonal polynomials as eigenfunctions, akin the Laplace-Beltrami operator for the spherical harmonics. 

\begin{thm}\label{thm:Delta0V0}
Let $\Delta_0^{(\xi)}$ denote the Laplace-Beltrami operator in $\xi \in \sph$. Define 
$$
\Delta_{0,\g}:= \left(t(1-t)\partial_t^2 + \big( d-1 - (d+\g)t \big) \partial_t+ t^{-1} \Delta_0^{(\xi)}\right) 
$$
for $\g > -1$. Then the polynomials in $\CV_n(\VV_0^{d+1}, \sw_{-1,\g})$ are eigenfunctions of $\Delta_{0,\g}$, 
$$
    \Delta_{0,\g} u =  -n (n+\g+d-1) u, \qquad \forall u \in \CV_n(\VV_0^{d+1}, \sw_{-1,\g}).
$$
\end{thm}

The reproducing kernel enjoys an addition formula that is a mixture of the addition formula for the spherical
harmonics and the Jacobi polynomials. The formula has the most elegant form when $\b = -1$, 
which is stated below. 

\begin{thm}  \label{thm:sfPbCone2}
Let $d \ge 2$ and $\g \ge -\f12$. Then, for $(x,t), (y,s) \in \VV_0^{d+1}$,
\begin{align} \label{eq:sfPbCone}
 \sP_n \big(\sw_{-1,\g}; (x,t), (y,s)\big) =  b_{\g,d}  \int_{[-1,1]^2} & Z_{2n}^{\g+d-1} \big( \zeta (x,t,y,s; v)  \big) \\
  & \times  (1-v_1^2)^{\f{d-4}{2}} (1-v_2^2)^{\g-\f12} \d v, \notag
\end{align} 
where $b_{\g,d}$ is a constant so that $\sP_0\big(\sw_{-1,\g}; (x,t), (y,s)\big) =1$ and 
$$
 \zeta (x,t,y,s; v)  = v_1 \sqrt{\tfrac{st + \la x,y \ra}2}+ v_2 \sqrt{1-t}\sqrt{1-s};
$$
moreover, the identity holds under the limit \eqref{eq:limitInt} when $\g = -\f12$ and/or $d = 2$. 
\end{thm} 

In particular, the orthogonal structure for the weight function $\sw_{-1,\g}$ on the conic surface satisfies 
both characteristic properties specified in Definition \ref{def:LBoperator} and Definition \ref{defn:additionF}.
We show that $\sw_{-1,\g}$ admits highly localized kernels in the next subsection. 

\iffalse
For $\delta > 0$, the Ces\`aro $(C,\delta)$ means $\sS_n^\delta (\sw_{\b,\g};f)$ of the Fourier series in the Jacobi
polynomials on the conic surface is defined by 
$$
 S_n^\delta (\sw_{,\b,\g};f) := \f{1}{\binom{n+\delta}{n}} \sum_{k=0}^n \binom{n-k+\delta}{n-k} \proj_k(\sw_{\b,\g}; f).
$$
The following theorem is proved in \cite{Xu20a}.

\begin{thm}\label{thm:cesaroV0}
For $\b \ge -1$ and $\g \ge -\f12$, define $\l_{\b,\g}: = \b+\g+d$. Then, the Ces\`aro $(C,\delta)$ 
means for $\sw_{\b,\g}$ on $\VV_0^{d+1}$ satisy 
\begin{enumerate} [\quad 1.]
\item if $\delta \ge \l_{\b,\g} + 1$, then $S_n^\delta(\sw_{\b,\g}; f)$ is nonnegative if $f$ is nonnegative;
\item $S_n^\delta (\sw_{\b,\g}; f)$ converge to $f$ in $L^1(\VV_0^{d+1}, \sw_{\b,\g})$ norm or
$C(\VV_0^{d+1})$ norm if $\delta > \l_{\b,\g}$ and only if $\delta > \l_{\b,\g}$ when $\g = -\f12$. 
\end{enumerate}
\end{thm}
\fi

\subsection{Highly localized kernels}
Let $\wh a$ be an admissible cut-off function. For $(x,t)$, $(y,s) \in \VV_0^{d+1}$, define the kernel 
$\sL_n(\sw_{-1,\g})$ by
$$
   \sL_n\big(\sw_{-1,\g}; (x,t),(y,s)\big) = \sum_{j=0}^\infty \wh a\left( \frac{j}{n} \right) \sP_j\big(\sw_{-1,\g}; (x,t), (y,s)\big). 
$$
We show that this kernel is highly localized. It is worthwhile to point out that $\sw_{-1,\g}$ does not include 
the constant weight function, or the Lebesgue measure on the conic surface. 

We shall need $\sw_{-1,\g}\big(\sc((x,t), r)\big)$ with $r = n^{-2}$.
Following the notation \eqref{eq:w_bg}, we define
\begin{equation}\label{eq:w(n;t)}
     \sw_{\g,d} (n; t) = \big(1-t+n^{-2}\big)^{\g+\f12}\big(t+n^{-2}\big)^{\f{d-2}{2}}. 
\end{equation}

\begin{thm} \label{thm:kernelV0}
Let $d\ge 2$ and $\g \ge -\f12$. Let $\wh a$ be an admissible cutoff function. Then
%for any $\sigma>0$ there exists a constant $c_\sigma$ such that
for any $\k > 0$ and $(x,t), (y,s) \in \VV_0^{d+1}$, 
\begin{equation*}%\label{V0-bound}
\left |\sL_n (\sw_{-1,\g}; (x,t), (y,s))\right|
\le \frac{c_\k n^d}{\sqrt{ \sw_{\g,d} (n; t) }\sqrt{ \sw_{\g,d} (n; s) }}
\big(1 + n \sd_{\VV_0}( (x,t), (y,s)) \big)^{-\k}.
\end{equation*}
\iffalse
\item if $\wh a$ is of type (c) and $0 < \varepsilon \le 1$, then for $(x,t), (y,s) \in \VV_0^{d+1}$, 
\begin{align*}%\label{simplex-bound1}
\left |\sL_n (\sw_{-1,\g}; (x,t), (y,s))\right|
\le \, & \frac{c_\s n^d}{\sqrt{\sw_{\g,d}(n; t)}\sqrt{\sw_{\g,d}(n; s)}} \\
  & \times \exp\left \{-\frac{c' n \sd_{\VV_0}((x,t), (y,s))}{[\ln(e+n \sd_{\VV_0}((x,t), (y,s)))]^{1+\ve}}\right \}.
\end{align*} 
\fi
\end{thm}

\begin{proof}
We only prove the case when $d > 2$ and $\g > -\f12$. The remaining cases $\g = -\f12$ and/or $d =2$ follow
similarly and are easier. Using the addition formula for the reproducing kernel and, by the quadratic transform 
\eqref{eq:Jacobi-Gegen0},
\begin{equation} \label{eq:Jacobi-Gegen}
   Z_{2n}^\l (x) = \frac{P_n^{(\l-\f12,-\f12)}(1)P_n^{(\l-\f12,-\f12)}(2x^2-1)}{h_n^{(\l-\f12,-\f12)}}, % = Z_n^{(\l-\f12,-\f12)}(x),
\end{equation}
we can write $\sL(\sw_{\b,\g})$ in terms of the kernel $L_n ^{(\l-\f12,-\f12)}$ of the Jacobi polynomials, where
$\l = \g+d -1$. Then
\begin{align}\label{eq:Ln-intV0}
\sL_n (\sw_{-1,\g}; (x,t), (y,s) )=  
    c_{\g}  \int_{[-1,1]^2} & L_n ^{(\l-\f12,-\f12)}\big(2 \zeta (x,t,y,s; v)^2-1 \big)\\
  &  \times    (1-v_1^2)^{\f{d-2}2-1}(1-v_2^2)^{\g-\f12} \d v. \notag
\end{align}
Let $\theta(x,t,y,s;v)= \arccos (2\zeta(x,t,y,s; v)^2 -1)$. Then 
$$
  1- \zeta(x,t,y,s;t)^2  = \frac12 (1 - \cos \theta(x,t,y,s;v)) =  \sin^2 \frac{\theta(x,t,y,s;v)}{2}
       \sim \theta(x,t,y,s;v)^2.
$$
We apply the estimate \eqref{eq:Ln(t,1)} for $L_n^{\a,\b}$ with $\alpha=\l-1/2$, $\beta= -1/2$
to obtain
\begin{align*}
 \left| \sL_n (\sw_{-1,\g}; (x,t), (y,s)) \right| \le c n^{2 \l +1} \int_{[-1,1]^2} 
 & \frac{1} { \left(1+ n\sqrt{1- \zeta (x,t,y,s; v)^2}\right)^{\k+2\g+d+1} }\\
       &   \times  (1-v_1^2)^{\f{d-2}2-1}(1-v_2^2)^{\g-\f12} \d v.
\end{align*}
From its definition, it is easy to verify that $|\zeta(x,t,y,s;v)|\le 1$ and 
\begin{align*}%\label{eq:1-xi}
   1- \zeta (x,t,y,s; v) =\, & 1- \cos \sd_{\VV_0} ((x,t), (y,s))  \\
         + & \sqrt{\frac{\la x,y\ra + t s}{2}}(1 - v_1)+\sqrt{1-s}\sqrt{1-t} (1 - v_2). \notag
\end{align*}
Consequently, since $t s + \la x,y\ra \ge 0$, it follows that 
\begin{align*}
   1 - \zeta (x,t,y,s; v) \, &\ge 1- \cos \sd_{\VV_0} ((x,t), (y,s)) \\
      & = 2 \sin^2 \frac{ \sd_{\VV_0} ((x,t), (y,s)) }{2} \ge \frac{2}{\pi^2} [\sd_{\VV_0}((x,t),(y,s))]^2. 
\end{align*}
Applying this inequality, we obtain the estimate  
\begin{align*}
 \left| \sL_n (\sw_{-1,\g}; (x,t), (y,s) )\right| \, & \le c n^{2 \l +1}
      \frac{1}{\left (1+  n \sd_{\VV_0}((x,t),(y,s))\right)^\k} \\
     & \times  b_{\g,d} \int_{[-1,1]^2} 
                     \frac{(1-v_1^2)^{\f{d-2}2-1}(1-v_2^2)^{\g-\f12}} {\left(1+n\sqrt{1- \zeta (x,t,y,s; v)}\right)^{2\g+d+1}} \d v,
\end{align*}
where $b_{\g,d} = c_{\g-\f12,\g-\f12} c_{\f{d-4}{2},\f{d-4}{2}}$. The estimate of the last expression is the 
crux of the proof and it is contained in the lemma 
below. 
\end{proof}

\begin{lem} \label{lem:kernelV0}
Let $d \ge 2$ and $\g > -\f12$. Then, for $\b \ge 2\g+ d+1$, 
\begin{align*}
 b_{\g,d} \int_{[-1,1]^2} & \frac{(1-v_1)^{\f{d-2}2-1}(1-v_2)^{\g-\f12}} 
        {\big(1+n\sqrt{1- \zeta (x,t,y,s; v)}\,\big)^{\b}} \d v\\
    & \qquad \le \frac{cn^{- (2\g+d-1)}} 
   {\sqrt{\sw_{\g,d}(n; t)}\sqrt{\sw_{\g,d}(n; s)}\big(1+n \sd_{\VV_0}((x,t),(y,s))\big)^{\b - 3\g- \frac{3d+1}{2}}} .
\end{align*}
\end{lem}
 
\begin{proof}
Using the lower bound of $1- |\zeta(x,t,y,s;v)|$ in the proof of the previous theorem, the left-hand side of the
stated inequality has the upper bound
$$
 \frac{c}{\big(1+n  \sd_{\VV_0}((x,t),(y,s)) \big)^{\b - 2 \g- d-1}}
         \int_{[-1,1]^2} 
                 \frac{(1-v_1)^{\f{d-2}2-1}(1-v_2)^{\g-\f12}} {(1+n \sqrt{1- \zeta (x,t,y,s; v)})^{2\g+d+1}} \d v. 
$$
Denote the above integral by $I(x,t,y,s)$. To complete the proof, we need to show that 
$$
\frac{ I(x,t,y,s)}{\big(1+  n \sd_{\VV_0}((x,t),(y,s))\big)^{\g+\f{d-1}{2}}} \le \frac{cn^{- (2\g+d-1)}} 
   {\sqrt{\sw_{\g}(n; t)}\sqrt{\sw_{\g}(n; s)}}.
$$
We need a lower bound for $1-\zeta$. Let $\a =  \sqrt{\frac{\la x,y\ra +ts}{2}}$ and $\b = \sqrt{1-t}\sqrt{1-s}$. Then
\begin{align*}
  1- \zeta (x,t,y,s; v) \,& = 1- \a + (1-v_1) \a - v_2 \b  \ge 1 -\a + \frac12(1-v_1)\a  - v_2 \b  \\
                         & = 1-  \frac{1+v_1}{2} \a- \frac{1-v_1}{2} \sqrt{ts} + \frac{1-v_1}{2}\sqrt{ts}  -v_2 \b.
\end{align*}
Using $\a \le \sqrt{ts}$ and $1- \sqrt{t s} - \sqrt{1-t}\sqrt{1-s} \ge 0$, we then obtain
\begin{align}\label{eq:1-zeta-lwd}
  1- \zeta (x,t,y,s; v) \, & \ge 1-\sqrt{ts} - \b + \frac12(1-v_1)\sqrt{ts}  + (1-v_2) \b \\
         & \ge \frac12 (1 - v_1)\sqrt{t}\sqrt{s}  +(1 - v_2) \sqrt{1-s}\sqrt{1-t}. \notag
\end{align}
Using this inequality in $I(x,t,y,s)$ and then making a change of variable $v_1 \mapsto 2 u_1-1$ and 
$v_2 \mapsto 2 u_2-1$, we obtain 
\begin{align*}
 I(x,t,y,s) \le c \int_{[0,1]^2} & \frac{(1-v_1)^{\f{d-2}2-1}(1-v_2)^{\g-\f12}} 
        {\big(1+n\sqrt{1-  (1 - v_1)\sqrt{t}\sqrt{s} - 2 (1 - v_2) \sqrt{1-s}\sqrt{1-t}}\,\big)^{2\g+d+1}} \d v. 
\end{align*}
This integral can be estimated by using the inequality \cite[(13.5.8)]{DaiX} 
\begin{equation}\label{eq:B+At}
  \int_0^1 \frac{(1-t)^{a-1} \d t}{(1+n \sqrt{B+A(1-t)})^b} \le c \frac{n^{-2 a}}{A^a (1+n\sqrt{B})^{b-2a-1}},
\end{equation}
which holds for $A>0$, $B\ge 0$, $a>0$ and $b\ge 2 a +1$. Applying this inequality with 
$A = \sqrt{1-t}\sqrt{1-s}$, $a = \g+\f12$, we obtain
$$
 I(x,t,y,s) \le \frac{cn^{-2\g-1}}{(\sqrt{1-t}\sqrt{1-s})^{\g+\f12}} \int_0^1 
               \frac{(1-v_1^2)^{\f{d-2}2-1}} {\left(1+n\sqrt{\sqrt{t}\sqrt{s}(1 - v_1)}\right)^{d-1}} \d v_1. 
$$
Applying \eqref{eq:B+At} one more time with $B=0$, $a = \f{d-2}{2}$ and $b= d-1$, we see that the last
integral is bounded by $c n^{-(d-2)}$, hence,
\begin{align*}
   I(x,t,y,s) \,&  \le \frac{cn^{- (2\g+d-1)}}{(\sqrt{1-t}\sqrt{1-s})^{\g+\f12} (\sqrt{t}\sqrt{s})^{\f{d-2}2}}  \\
        &  \le \frac{cn^{- (2\g+d-1)}}{(\sqrt{1-t}\sqrt{1-s}+n^{-2} )^{\g+\f12} (\sqrt{t}\sqrt{s}+n^{-2})^{\f{d-2}2}}, 
\end{align*}
where the second inequality follows since $I(x,t,y,s) \le 1$ holds trivially by the choice of $b_{\g,\d}$. Now, 
using the elementary identity \cite[(11.5.13)]{DaiX}
\begin{equation} \label{eq:ab+=a+b+}
   (a+n^{-1})(b+n^{-1}) \le 3 (ab+n^{-2})(1+n|b-a|)
\end{equation}
with $a = \sqrt{t}$ and $b=\sqrt{s}$ as well as with $a = \sqrt{1-t}$ and $b=\sqrt{1-s}$, we obtain 
$$
 I(x,t,y,s) \le \frac{cn^{- (2\g+d-1)}\big(1+  n \big|\sqrt{t}-\sqrt{s}\big|\big)^{\f{d-2}2} \big(1+  n \big|\sqrt{1-t}-\sqrt{1-s}\big|\big)^{\g+\f12}} 
   {\sqrt{\sw_{\g,d}(n; t)}\sqrt{\sw_{\g,d}(n; s)}}.
$$
Finally, by Lemma \ref{lem:|s-t|}, we conclude that 
$$
 I(x,t,y,s) \le \frac{cn^{- (2\g+d-1)}(1+  n \sd_{\VV_0}((x,t),(y,s)))^{\g+\f{d-1}{2}}} 
   {\sqrt{\sw_{\g,d}(n; t)}\sqrt{\sw_{\g,d}(n; s)}},
$$
which is what we need to complete the proof. 
\end{proof}

The following corollary, following immediately from \eqref{eq:1-zeta-lwd}, will be used later. 

\begin{cor} \label{cor:kernelV0}
Let $d \ge 2$ and $\g > -\f12$. Then, for $\b \ge 2\g+ d+1$, 
\begin{align*}
 b_{\g,d} \int_{[-1,1]^2} & \frac{(1-v_1)^{\f{d-2}2-1}(1-v_2)^{\g-\f12}} 
        {\left(1+n\sqrt{1-  (1 - v_1)\sqrt{t}\sqrt{s}  -  (1 - v_2) \sqrt{1-s}\sqrt{1-t}}\,\right)^{2\g+\d+1}} \d v\\
    & \qquad \le \frac{cn^{- (2\g+d-1)}\big(1+n  \sd_{[0,1]}((x,t),(y,s))\big)^{\g+ \frac{d-1}{2}}} 
   {\sqrt{\sw_{\g,d}(n; t)}\sqrt{\sw_{\g,d}(n; s)}} .
\end{align*}
\end{cor}

By \eqref{eq:w_bg}, we have shown that $\sw_{-1,\g}$ admits Assertion 1 of the highly localized kernel. 
Our next result shows that it also admits Assertion 2. 

\begin{thm} \label{thm:L-LkernelV0}
Let  $\wh a$ be an admissible cutoff function. Let $d\ge 2$ and $\g \ge -\f12$. For 
$(x_i,t_i), (y,s) \in \VV_0^{d+1}$ and $(x_1,t_1) \in \sc \big((x_2,t_2), c^* n^{-1}\big)$ with $c^*$ small 
and for $\k > 0$,  
\begin{align}\label{V0-bound}
     &  \left |\sL_n (\sw_{-1,\g}; (x_1,t_1), (y,s))-\sL_n (\sw_{-1,\g}; (x_2,t_2), (y, s))\right| \\
     & \qquad \le c_\k \frac{  n^{d+1} \sd_{\VV_0}((x_1,t_1), (x_2, t_2))}{\sqrt{ \sw_{\g,d} (n; s) }\sqrt{ \sw_{\g,d} (n; t_2) }
     \big(1 + n \sd_{\VV_0}( (y,s), (x_2, t_2)) \big)^{\k}}. \notag
\end{align}
\end{thm}

\begin{proof}
Denote the left-hand side of \eqref{eq:L-LkernelV0} by $K$. Let $\partial L(u) = L'(u)$. Using the integral expression 
\eqref{eq:Ln-intV0} of $\sL_n (\sw_{-1,\g})$, we obtain 
\begin{align} \label{eq:L-LkernelV0}
K &
 \le 2 \int_{[-1,1]^2} \big\| \partial L_n^{\l-\f12,-\f12} \big(2(\cdot)^2-1\big)\big\|_{L^\infty(I_v)} 
     \big |\zeta_1(v)^2 - \zeta_2(v)^2| \\ % (x_1,t_1,y,s;v)^2 - \zeta(x_2,t_2,y, s;v)^2\big| \\
  & \qquad \qquad\qquad
         \times  (1-v_1^2)^{\frac{d-2}{2}-1} (1-v_2^2)^{\g-\f12} \d v, \notag
\end{align}
where $\zeta_i(v) = \zeta (x_i,t_i,y,s;v)$, and $I_v$ is the interval with end points $\zeta_1(v)$ 
and $\zeta_2(v)$. Since $|\zeta(\cdot)| \le 1$, $|\zeta_1(v)^2 - \zeta_2(v)^2| \le 2 |\zeta_1(v)- \zeta_2(v)|$. 
We claim that 
\begin{align} \label{eq:zeta1-zeta2}
 |\zeta_1(v)- \zeta_2(v)| \le  c\, \sd_{\VV_0}\big ((x_1,t_1),(x_2,t_2)\big)
       \big( \Sigma_1 + \Sigma_2(v_1) + \Sigma_3 (v_2)\big), 
\end{align}
where 
\begin{align*}
  \Sigma_1 \, & =  \sd_{\VV_0}\big((x_i, t_i),(y,s)\big) +  \sd_{\VV_0}\big ((x_1,t_1),(x_2,t_2)\big), \\
%  \sd_{\VV_0}\big((x_1,t_1),(y,s)\big)+\sd_{\VV_0}\big((x_2,t_2),(y,s)\big),  \\
  \Sigma_2(v_1) \, & = (1-v_1)\sqrt{s}, \\
  \Sigma_3 (v_2)\, & = (1 - v_2)\sqrt{1-s}. 
\end{align*}
To see this, we use \eqref{eq:d=d+d} and writing $x_i = t_i \xi_i$ and $y = s \eta$ to obtain 
\begin{align*}
 \zeta_1(v)- \zeta_2(v) \, & = \cos \sd_{\VV_0}\big((x_1,t_1),(y,s)\big) - \cos \sd_{\VV_0}\big((x_2,t_2),(y,s)\big) \\
    &     +  (1-v_1) \left(\sqrt{ t_2 s} \cos \frac{\sd_{\SS}(\xi_2,\eta)}{2} - 
         \sqrt{ t_1 s} \cos \frac{\sd_{\SS}(\xi_1,\eta)}{2} \right) \\
     &   +  (1- v_2) \left(\sqrt{1- t_2} - \sqrt{1-t_1}\right) \sqrt{1-s}. 
\end{align*}
Denote temporarily $\a_i =  \sd_{\VV_0}((x_1,t_1),(y,s))$ for $i =1, 2$. Hence, using the identity 
\begin{align*}
  \cos \a_1 -  \cos \a_2  \, &= 2 \sin \frac{\a_1 - \a_2}{2} \sin \frac{\a_1+\a_2}2 \\
    & =   2 \sin \frac{\a_1 - \a_2}{2} \left( 2 \sin \frac{\a_i}{2} + \sin \frac{|\a_1-\a_2|}{2}\right),
\end{align*} 
it follows readily that 
$$
   |\cos \a_1 -  \cos \a_2| \le |\a_1-\a_2| \left ( |\a_1| + \tfrac 12 |\a_1 - \a_2| \right).
$$
By the triangle inequality of $\sd_{\VV_0}$, $|\a_1-\a_2| \le \sd_{\VV_0}((x_1,t_1),(x_2,t_2))$, this gives
the estimate for the $\Sigma_1$ term. Moreover, assuming $t_2 \ge t_1$, for example, and applying similar 
argument for $\sd_\SS$, we then obtain, using Lemma \ref{lem:|s-t|} and \eqref{eq:d2=d2+d2}, that 
\begin{align*}
   \bigg | \sqrt{ t_1 } \cos \frac{\sd_{\SS}(\xi_1,\eta)}{2} - &  \sqrt{ t_2 } \cos \frac{\sd_{\SS}(\xi_2,\eta)}{2} \bigg | \\
          & \le \left | \sqrt{t_2} - \sqrt{t_1} \right | + \sqrt{t_1} 
               \left | \cos \frac{\sd_{\SS}(\xi_1,\eta)}{2} -\cos \frac{\sd_{\SS}(\xi_2,\eta)}{2} \right | \\
  & \le \sd_{\VV_0}((x_1,t_1),(x_2,t_2)) + \f12 (t_1 t_2)^{\f14} \sd_{\SS}(\xi_1,\xi_2) \\
  & \le c\, \sd_{\VV_0}((x_1,t_1),(x_2,t_2)),
\end{align*}
which verifies the $\Sigma_2(v_1)$ term. The third term with $\Sigma_3(v_2)$  follows from using 
Lemma \ref{lem:|s-t|} one more time. This verifies the claim \eqref{eq:zeta1-zeta2}. 

Since $\max_{r\in I_v} |1+n \sqrt{1- r}|^{-\s}$ is attained at one of the end points of the interval, it follows 
from \eqref{eq:DLn(t,1)} with $m =1$.  
$$
\big\| \partial L_n^{\l-\f12,-\f12} \big(2(\cdot)^2-1\big)\big\|_{L^\infty(I_v)} 
 \le c \left[  \frac{n^{2 \l + 3}}{\big(1+n\sqrt{1-\zeta_1(v)^2} \big)^{\k}} +  \frac{n^{2 \l + 3}}{\big(1+n\sqrt{1-\zeta_2(v)^2}\big)^{\k}} \right].
$$
Consequently, we see that $K$ is bounded by a sum of integrals
\begin{align*}
  K  \le c \, \sd_{\VV_0}\big((x_1,t_1),(x_2,t_2)\big) & \int_{[-1,1]^2}\left[  \frac{n^{2 \l + 3}}
    {\big(1+n\sqrt{1-\zeta_1(v)^2} \big)^{\k}} +  \frac{n^{2 \l + 3}}{\big(1+n\sqrt{1-\zeta_2(v)^2}\big)^{\k}} \right ] \\
 &   \quad \times \big(\Sigma_1+\Sigma_2(v) + \Sigma_3(v) \big) (1-v_1^2)^{\frac{d-2}{2}-1} (1-v_2^2)^{\g-\f12} \d v.
\end{align*}
Since $(x_1,t_1) \in \sc \big((x_2,t_2), c^* n^{-1}\big)$, $\Sigma_1$ is bounded by 
$\Sigma_1 \le c n^{-1}  \big(1 + n \sd_{\VV_0}\big((x_i, t_i),(y,s)\big) \big)$. Hence, we obtain  from Lemma \ref{lem:kernelV0} that 
\begin{align*}
 \int_{[-1,1]^2} & \frac{n^{2 \l + 3}}{\big(1+n\sqrt{1-\zeta_i(v)^2} \big)^{\k}} \Sigma_1 
        (1-v_1^2)^{\frac{d-2}{2}-1} (1-  v_2^2)^{\g-\f12} \d v\\
 & \le  c  \frac{  n^{d+1} }{\sqrt{ \sw_{\g,d} (n; s) }\sqrt{ \sw_{\g,d} (n; t_i) }
     \big(1 + n \sd_{\VV_0}( (y,s), (x_i, t_i)) \big)^{\k(\g,d)-1}},
\end{align*} 
where $\k(\g,d) = \k - 3 \g - \frac{3d+1}{2}$ for either $i=1$ or $i=2$. Since  
$\sw_{\g,d}(n,t_1) \sim \sw_{\g,d}(n,t_2)$ and $\sd_{\VV_0}((x_1, t_1),(y,s))
 + n^{-1} \sim \sd_{\VV_0}((x_2, t_2),(y,s)) + n^{-1}$ by Lemma \ref{lem:|s-t|},
we can replace $(x_1,t_1)$ in the right-hand side by $(x_2,t_2)$. This shows that the integral containing 
$\Sigma_1$ has the desired estimate. 

For the remaining integrals, the same consideration shows that we only need to consider those containing 
$\zeta_2(v)$. For the integral that contains $\Sigma_2(v_1)= (1-v_1)\sqrt{s}$, the factor $(1-v_1)$ increases 
the power of the weight to $(1-v_1)^{\f d 2}$, so that we can apply Lemma \ref{lem:kernelV0} with 
$\frac{d-2}{2}$ replaced by $\frac{d}{2}$, which leads to 
\begin{align*}
 \int_{-1}^1 & \frac{n^{2 \l + 3}}{\big(1+n\sqrt{1-\zeta_2(v)^2} \big)^{\k}} \Sigma_2(v_1) 
        (1-v_1^2)^{\frac{d-2}{2}-1} (1-  v_2^2)^{\g-\f12} \d v\\
 & \le  c  \frac{  n^{d+1} n^{-1}\sqrt{s} }{\sqrt{ \sw_{\g,d+1} (n; s) }\sqrt{ \sw_{\g,d+1} (n; t_2) }
   \big(1 + n \sd_{\VV_0}( (y,s), (x_2, t_2)) \big)^{\k(\g,d+1)}} \\
  & \le  c  \frac{  n^{d+1}  }{\sqrt{ \sw_{\g,d} (n; s) }\sqrt{ \sw_{\g,d} (n; t_2) }
     \big(1 + n \sd_{\VV_0}( (y,s), (x_2, t_2)) \big)^{\k(\g,d+1)}}, 
\end{align*} 
where the last step follows from the inequality $n^{-1} \sqrt{s}\le (\sqrt{t_2}+n^{-1}) (\sqrt{s}+n^{-1})$. 
The integral that contains $\Sigma_3(v_2)$ can be estimated similarly by applying Lemma \ref{lem:kernelV0} 
with $\g$ replaced by $\g +1$ and using $n^{-1} \sqrt{1-s}\le (\sqrt{1-t_2}+n^{-1}) (\sqrt{1-s}+n^{-1})$. 
This completes the proof.
\end{proof} 
 
The first two assertions for the highly localized kernels are established for $\sw_{-1,\g}$ when $\g \ge -\f12$. 
The case of $p=1$ of the following lemma establishes Assertion 3. Recall that $\sw_{\b,\g,d}(n; t)$ is defined 
in \eqref{eq:w_bg}. 

\begin{lem}\label{lem:intLn}
Let $d\ge 2$, $\b> -d$ and $\g > -1$. For $0 < p < \infty$, assume 
$\k > \frac{2d}{p} + (\g+\b+\f{d+1}{2}) |\f1p-\f12|$. Then for $(x,t) \in \VV_0^{d+1}$,  
\begin{align}\label{eq:intLn1}
\int_{\VV_0^{d+1}} \frac{ \sw_{\b,\g}(s)  \d \s(y,s) }{  \sw_{\b,\g,d} (n; s)^{\f{p}2}
    \big(1 + n \sd_{\VV_0}( (x,t), (y,s)) \big)^{\k p}} 
    \le c n^{-d} \sw_{\b,\g,d} (n; t)^{1-\f{p}{2}}.
\end{align}
\end{lem}

\begin{proof}
Let $J_{p}$ denote the left-hand side of \eqref{eq:intLn1}. By Lemma \ref{lem:CorA3}, it is sufficient to 
estimate $J_{2}$. Let $x = t \xi$ and $y = s\eta$. The definition of $\sd_{\VV_0}(\cdot,\cdot)$ shows that it 
is a function of $\la \xi,\eta\ra$. Hence, by \eqref{eq:intSS}, we deduce that 
\begin{align*}
  J_{2,\k} \, &\le c   \int_0^1 \int_{-1}^1 \frac{ s^{d-1} \sw_{\b,\g}(s)  (1-u^2)^{\f{d-3}{2}} }{\sw_{\b,\g,d} (n; s)
        \big(1 + n \arccos \big( \sqrt{ts} \sqrt{\frac{1+u}{2}} + \sqrt{1-t}\sqrt{1-s} \big)\big)^{2\k} } \d u  \d s \\
        & \le c  \int_0^1 \int_{0}^1 \frac{s^{d-1} \sw_{\b,\g}(s) v^{d-2}(1-v^2)^{\f{d-3}{2}} }{\sw_{\b,\g,d} (n; s)
        \left(1 + n \sqrt{1- \sqrt{ts} v -  \sqrt{1-t}\sqrt{1-s}} \right)^{2k}} \d v  \d s,
\end{align*}
where the second step follows from changing variable $ \sqrt{\frac{1+u}{2}} \mapsto v$ and the relation
$\t \sim \sin \f{\t}{2} \sim \sqrt{1-\cos \t}$. Making a further changing of variable $v \mapsto z/\sqrt{s}$ gives
\begin{align*}
  J_{2} \,& \le c    \int_0^1 \int_{0}^{\sqrt{s}} \frac{s\, \sw_{\b,\g}(s) z^{d-2}(s- z^2)^{\f{d-3}{2}} }{\sw_{\b,\g,d} (n; s)
        \left(1 + n \sqrt{1- \sqrt{t}\,z -  \sqrt{1-t}\sqrt{1-s}} \right)^{2\k} }  \d z  \d s \\
        & \le c   \int_0^1 \int_{0}^{\sqrt{s}} \frac{(s- z^2)^{\f{d-3}{2}} }{ (1-s+n^{-2})^{\f12} 
       \left(1 + n \sqrt{1- \sqrt{t}\, z -  \sqrt{1-t}\sqrt{1-s}} \right)^{2\k} } \d z  \d s,        
\end{align*}
where we have used $s z^{d-2} \le s^{\frac{d}{2}} \le (s+n^{-2})^{\frac{d}{2}}$. One more change of variable 
$s\mapsto 1-w^2$ with $\d s = w \d w$ and 
$w \le (w^2+n^{-2})^\f12$, we obtain
\begin{align*}
   J_{2}  \le c    \int_0^1 \int_{0}^{\sqrt{1-w^2}} \frac{(1-w^2- z^2)^{\f{d-3}{2}} }{
    \left(1 + n \sqrt{1- \sqrt{t} z -  \sqrt{1-t} \, w} \right)^{2\k}} \d z  \d w, 
\end{align*}
which is an integral over the positive quadrant $\{(z,w) \in \BB^2: w \ge 0,  z \ge 0\}$ of the unit disk 
$\BB^2$. Setting $p = \sqrt{t} z + \sqrt{1-t} \, w$ and $q = - \sqrt{1-t} z + \sqrt{t}w$ in the
integral, which is an orthogonal transformation, and enlarging the integral domain while taking into
account that $p \ge 0$, it follows that 
\begin{align*}
  J_{2} \, & \le c    \int_0^1 \frac{1} {\left(1 + n \sqrt{1- p} \right)^{2\k} } 
        \int_{- \sqrt{1-p^2}}^{\sqrt{1-p^2}}(1-p^2- q^2)^{\f{d-3}{2}} \d q  \d p \\
       & \le c   \int_0^1 \frac{(1-p^2)^{\f{d-2}{2}} } {\left(1 + n \sqrt{1- p} \right)^{2\k} }  
          \le c  n^{-d} \int_0^n \frac{ r^{d-1} }{ (1 + r )^\k } \d r \le c n^{-d}
\end{align*} 
by setting $r = n \sqrt{1-p}$ and recalling that $\k > d$. This completes the proof. 
\end{proof}

\begin{prop}\label{prop:intLn}
Let $d\ge 2$ and $\g \ge - \f12$. For $0 < p < \infty$ and $(x,t) \in \VV_0^{d+1}$,  
\begin{equation*}
   \int_{\VV_0^{d+1}} \left| \sL_n\big(\sw_{-1,\g};(x,t),(y,s)\big) \right|^p \sw_{-1,\g}(s) \d \s(y,s)
       \le c \left(\frac{n^d}{\sw_{\g,d} (n; t)}\right)^{p-1}. 
\end{equation*}
\end{prop}

The proof is immediate by (i) of Theorem \ref{thm:kernelV0} and Lemma \ref{lem:intLn}. 

\begin{cor}
For $\g \ge -\f12$, the space $(\VV_0^{d+1}, \sw_{-1,\g}, \sd_{\VV_0})$ is a localizable homogeneous space. 
\end{cor}

\subsection{Maximal $\ve$-separated sets and MZ inequality}\label{sec:ptsV0}
Let $\ve > 0$. An $\ve$-separated set is defined in Definition \ref{defn:separated-pts}. In this subsection, we
provide a construction of some examples of maximal $\ve$-separated points on the conic surface. 
 
For our construction, we shall need $\ve$-separated points on the unit sphere $\sph$. We adopt the 
following notation. For $\ve > 0$, we denote by $\Xi_{\SS}(\ve)$ a maximal $\ve$-separated set
on the unit sphere $\sph$ and we let $\SS_\xi(\ve)$ be the subsets in $\sph$ so that the collection 
$\{\SS_\xi(\ve): \xi \in \Xi_\SS(\ve)\}$ is a partition of $\sph$, and we assume
\begin{equation}\label{eq:ptsV01} 
      \sc_{\SS}(\xi, c_1 \ve) \subset \SS_\xi(\ve) \subset \sc_{\SS}(\xi, c_2 \ve), \qquad \xi \in \Xi_{\SS}(\ve),
\end{equation} 
where $\sc_{\SS}(\xi,\ve)$ denotes the spherical cap centered at $\xi$ with radius $\ve$, $c_1$ and $c_2$ 
depending only on $d$. Such a $\Xi_\SS(\ve)$ exists for all $\ve > 0$, see for example \cite[Section 6.4]{DaiX}, 
and its cardinality satisfies  
\begin{equation}\label{eq:ptsV02} 
c_d' \ve^{-d+1} \le \# \Xi_{\SS}(\ve) \le c_d \ve^{-d+1}.
\end{equation} 

We now consider the subsets of points on the cone $\VV_0^{d+1}$. For $\ve > 0$, we denote by 
$\Xi_{\VV_0} = \Xi_{\VV_0}(\ve)$ an $\ve$-separated set and denote by 
$\{\VV_0(\xi,t): (t\xi,t) \in \Xi_{\VV_0}\}$ a partition of $\VV_0^{d+1}$.
%$$
%      \sw_{\b,\g} \big(\VV_0(\xi,t)\big) \sim \sw_{\b,\g} \big(\sc((t\xi,t), c_2 \ve)\big).  
%$$
We start with an explicit construction of such $\Xi_{\VV_0}$ and $\VV_0(\xi,t)$. 

Let $\ve > 0$ and let $N = \lfloor \frac{\pi}{2}\ve^{-1} \rfloor$. For $1\le j \le N$ we define 
$$
 \t_j:= \frac{(2j-1)\pi}{2 N},  \qquad \t_j^- :=  \t_j- \frac{\pi}{2 N}  \quad \hbox{and} \quad \t_j^+ :=  \t_j +\frac{\pi}{2 N}.
$$
Let $t_j =  \sin^2 \frac{\t_j}{2}$ and define $t_j^-$ and $t_j^+$ accordingly. In particular, $t_1^- = 0$ and 
$t_N^+ = 1$. Then $\t_{j+1}^- =\t_j^+$ and $\VV_0^{d+1}$ can be partitioned by 
$$
   \VV_0^{d+1} =  \bigcup_{j=1}^N \VV_0^{(j)}, \quad \hbox{where}\quad \VV_0^{(j)}:= 
        \left\{(x,t) \in \VV_0^{d+1}:   t_j^- < t \le t_j^+ \right \}.  
$$
Let $\ve_j := (2 \sqrt{t_j})^{-1} \pi \ve$.  Then $\Xi_\SS(\ve_j)$ is the maximal $\ve_j$-separated 
set of $\sph$ such that $\{\SS_\xi(\ve_j): \xi \in \Xi_\SS(\ve_j)\}$ is a partition $\sph = \bigcup_{\eta \in \Xi_\SS(\ve_j)} \SS_\eta(\ve_j)$, and 
$$
   \# \Xi_\SS(\ve_j) \sim \ve_j^{-d+1}.
$$
For each $j =1,\ldots, N$, we decompose $\VV_0^{(j)}$ by 
$$
 \VV_0^{(j)} =  \bigcup_{\xi \in  \Xi_\SS(\ve_j)} \VV_0(\xi,t_j), \quad \hbox{where}\quad 
 \VV_0(\xi,t_j):= \left\{(t\eta,t):  t_j^- < t \le t_j^+, \, \eta \in \SS_\xi(\ve_j) \right\}.
$$
Finally, we define the subset $\Xi_{\VV_0}$ of $\VV_0^{d+1}$ by
$$
   \Xi_{\VV_0} = \big\{(t_j \xi, t_j): \,  \xi \in \Xi_\SS(\ve_j), \, 1\le j \le N \big\}. 
$$

\begin{prop} \label{prop:subsetV0}
Let $\ve > 0$ and $N = \lfloor \frac{\pi}{2} \ve^{-1} \rfloor$. Then $\Xi_{\VV_0}$ is a maximal $\ve$-separated 
set of $\VV_0^{d+1}$ and $\{\VV_0(\xi, t_j): \xi \in \Xi_\SS(\ve_j), \, 1\le j \le N \}$ is a partition 
$$
   \VV_0^{d+1} =  \bigcup_{j=1}^N \bigcup_{\xi \in \Xi_\SS(\ve_j)} \VV_0(\xi,t_j).
$$
Moreover, there are positive constants $c_1$ and $c_2$ depending only on $d$ such that 
\begin{equation}\label{eq:incluV0cap}
      \sc \big((t_j\xi,t_j), c_1 \ve\big) \subset \VV_0(\xi,t_j) \subset \sc \big( (t_j \xi,t_j), c_2 \ve\big), 
\end{equation}
and $c_d'$ and $c_d$ such that 
$$
c_d' \ve^{-d} \le \# \Xi_{\VV_0} \le c_d \ve^{-d}. 
$$
\end{prop}

\begin{proof}
Let $(t_j \xi, t_j)$ and $(t_k \eta, t_k)$ be two distinct points in $\Xi_{\VV_0}$. If $t_j \ne t_k$,
then 
$$
\sd_{\VV_0}\big((t_j \xi, t_j), (t_k \eta, t_k)  \big) \ge \sd_{[0,1]}(t_j,t_k) = \frac12 |\t_j - \t_k| \ge \frac{\pi}{2N}  \ge \ve.
$$
If $j = k$, then $\xi$ and $\eta$ are both elements of $\SS(\ve_j)$, so that $\sd_{\SS}(\xi,\eta) \ge \ve_j$. Hence,
using $\f{2}{\pi}\phi \le \sin \phi \le \phi$, we deduce from \eqref{eq:d=d+d} that 
$$
\sd_{\VV_0}\big((t_j \xi, t_j), (t_j \eta, t_j)\big) \ge \frac{2}{\pi} \sqrt{t_j} \sd_{\SS}(\xi,\eta) \ge \frac{2}{\pi} 
    \sqrt{t_j}\ve_j = \ve. 
$$
Hence, $\Xi_{\VV_0}$ is $\ve$-separated. Moreover, since $\#\Xi_\SS(\ve_j) \sim \ve_j^{-d+1}$
and $\ve_j \sim \ve/ \t_j$, it follows that 
$$
 \#\Xi_{\VV_0} = \sum_{j=1}^N \#\Xi_\SS(\ve_j) 
     \sim \sum_{j=1}^N \ve_j^{-d+1} \sim \ve^{-d+1} \sum_{j=1}^N \t_j^{d-1} \sim \ve^{-d+1} N \sim \ve^{-d}. 
$$

For the proof of \eqref{eq:incluV0cap}, we first consider the ball $\sc_{[0,1]}(t_j, r) = \{s: \sd_{[0,1]}(s,t_j) \le r\}$
on $[0,1]$. For $0< \delta < \pi/2$, it is easy to see that 
$$
    \sc_{[0,1]}(t_j, \delta /N) \subset \{s: (s\eta, s)\in \VV_0(\xi,t_j)\} \subset \sc_{[0,1]}(t_j, \pi/N).
$$
We further choose $\delta$ so that $4 \delta (1+\delta) < \f12$. For $s \in \sc_{[0,1]}(t_j, \delta /N)$, write
$ s= \sin^2 \f{\phi} 2$, then  
$$
  |t_j-s| = \frac12 |\cos \t -\cos\phi| = \left|\sin \frac{\t_j-\phi}{2} \sin \frac{\t_j+\phi}{2}\right| \le \frac{\delta}{N} 
        \left( 2 \sqrt{t_j} + \frac{\delta}{N}\right), 
$$
where we have used $|\sin \frac{\t_j+\phi}{2}| \le 2 \sin \f{\t_j}{2} + |\sin \frac{\t_j-\phi}{2}| \le 2 \sqrt{t_j} + \sd_{[0,1]}(s,t_j)$. 
For $j \ge 1$, $N^{-1} = \frac{2}{(2j-1)\pi} \t_j \le 2 \sin \frac{\t_j}2 = 2 \sqrt{t_j}$, it follows that 
$$
 |t_j-s| \le  2 \delta ( 2 + \delta)  t_j \le \tfrac12 t_j,
$$
which implies in particular that $\frac{1}{2}t_j \le s \le \f 3 2 t_j$. Furthermore, the same proof shows 
if $s \in\sc_{[0,1]}(t_j, \pi/N)$, then $s \le c_* t_j$. By definition, there are constants $b_1 > 0$ and 
$b_2> 0$ such that $\sc_\SS(\xi,b_1\ve_j)\subset \SS_\xi(\ve_j) \subset \sc_\SS(\xi,b_2 \ve_j)$. 
We claim that \eqref{eq:incluV0cap} holds for some $c_1 < \delta$ and some $c_2 > b_2$. Indeed,
if $(y,\eta) \subset \sc \big((t_j\xi,t_j), c_1 \ve\big)$, then $\sd_{[0,1]}(s,t_j) \le c_1 \ve \le \delta/N$ 
so that $s \ge t_j/2$, and $(s t)^{\f14} \sd_{\SS}(\xi,\eta) \le c c_1 \ve$ by \eqref{eq:d2=d2+d2} 
so that $\sd_\SS(\xi,\eta) \le 2^{\f14} c c_1 \ve/\sqrt{t_j} \le b_1 \ve_j$ by choosing $c_1$ small. 
This establishes the left-hand side inclusion of \eqref{eq:incluV0cap}. The right-hand side inclusion
can be similarly established. The proof is completed. 
\end{proof}
 
The above construction establishes the existence of maximal $\ve$-separated set on the conic  surface. 
Since $(\Omega, \sw_{-1,\g},\sd_{\VV_0})$ is a localizable homogeneous space, we can then deduce 
the Marcinkiewicz-Zygmund inequality on such sets for all doubling weights on the conic surface. The
weight $\sw$ on $\VV_0^{d+1}$ is in general a function of both $x$ and $t$. 

\begin{thm} \label{thm:MZinequalityV0}
Let $\sw$ be a doubling weight on $\VV_0^{d+1}$. Let  $\Xi_{\VV_0}$ be a maximal 
$\f \delta n$-separated subset of $\VV_0^{d+1}$ and $0 < \delta \le 1$.
\begin{enumerate}[$(i)$]
\item For all $0<p< \infty$ and $f\in\Pi_m(\VV_0^{d+1})$ with $n \le m \le c n$,
\begin{equation*}
  \sum_{z \in \Xi_{\VV_0}} \Big( \max_{(x,t)\in \sc((z,r), \f \delta n)} |f(x,t)|^p \Big)
     \sw\!\left(\sc((z, r), \tfrac \delta n) \right) \leq c_{\sw,p} \|f\|_{p,\sw}^p
\end{equation*}
where $c_{\sw,p}$ depends on the doubling constant $L(\sw)$ and on $p$ when $p$ is close to $0$.
\item For $0 < r < 1$, there is a $\delta_r > 0$ such that for $\delta \le \delta_r$, $r \le p < \infty$ and 
$f \in \Pi_n(\VV_0^{d+1})$,  
\begin{align*}
  \|f\|_{p,\sw}^p \le c_{\sw,r} \sum_{z \in\Xi}
       \Big(\min_{(x,t)\in \sc\bigl((z,r), \tfrac{\delta}n\bigr)} |f(x,t)|^p\Big)
          \sw\bigl(\sc((z,r), \tfrac \delta n)\bigr)
\end{align*}
where $c_{\sw,r}$ depends only on $L(\sw)$ and on $r$ when $r$ is close to $0$.
\end{enumerate}
\end{thm}

\subsection{Cubature rules and localized tight frames}\label{sec:CF-frameV0}
We now turn our attention to Assertion 4 in Subsection \ref{subsect:CFunction} and construct 
fast decaying polynomials on the conic surface. 

\begin{lem}\label{lem:A4}
Let $d\ge 2$. For each $(x,t) \in \VV_0^{d+1}$, there is a polynomial $T_{x,t}$ of degree $n$ that
satisfies
\begin{enumerate} [   (1)]
\item $T_{x,t}(x,t) =1$, $T_{x,t}(y,s) \ge c > 0$ if $(y,s) \in \sc( (x,t), \f{\delta}{n})$, and for every $\k > 0$,
$$
   0 \le T_{x,t}(y,s) \le c_\k \left(1+ \sd_{\VV_0}\big((x,t),(y,s)\big) \right)^{-\k}, \quad (y,s) \in \VV_0^{d+1}.
$$
\item there is a polynomial $q(t)$ of degree $n$ such that $q(t) T_{x,t}$ is a polynomial of degree $2 n$ 
in $(x,t)$ variables and $1 \le q_n(t) \le c$. 
\end{enumerate}
\end{lem}

\begin{proof}
Let $r$ be a positive integer such that $\k \le 2r$. For positive integer $n$, let $m = \lfloor \frac{n}{r} \rfloor +1$ 
and define 
$$
  S_n(\cos \t) = \left( \frac{\sin (m+\f12)\f\t 2} {(m+\f12) \sin \frac{\t}2} \right)^{2r}, \qquad 0 \le \t \le \pi.
$$
By considering $m \t \ge 1$ and $m \t \le 1$ separately if necessary, it follows readily that 
\begin{equation} \label{eq:fastSn}
  S_n(1) = 1, \qquad 0 \le S_n(\cos \t) \le c \big(1+ n \t\big)^{-2r}, \quad 0 \le \t \le \pi.
\end{equation}
Moreover, $S_n(z)$ is an even algebraic polynomial of degree at most $2 n$. For a fixed $(x,t) \in \VV_0^{d+1}$, 
we define 
$$
T_{(x,t)}(y,s) := \frac{S_n\Big(\sqrt{ \frac{\la x,y\ra + t s}{2}}+ \sqrt{1-t}\sqrt{1-s}\Big) + 
S_n\Big(\sqrt{ \frac{\la x,y\ra + t s}{2}}- \sqrt{1-t}\sqrt{1-s}\Big)}{1 + S_n(2 t-1)}. 
$$
As an even polynomial, $S_n(z)$ is a sum of even monomials. Since, by binomial formula, 
$(\sqrt{a} + \sqrt{b})^{2k} + (\sqrt{a} - \sqrt{b})^{2k}$ is a polynomial of degree $k$ in $a$ and $b$, it 
follows that $T_{(x,t)}$ is indeed a polynomial in $(y,s)$ of degree at most $n$. Moreover, it satisfies 
$T_{(x,t)} (x,t) = 1$ since $\la x,x\ra = t^2$. Furthermore, using $\f{2}{\pi} \t \le \sin \t \le \t$ for $0 \le \t \le \pi/2$, 
we see that $S_n$ satisfies 
$$
    S_n (\cos \t) \ge \left(\frac{2}{\pi}\right)^{2r}, \qquad \hbox{if \, $0 \le \t \le \frac{2\pi}{2m+1}$}.  
$$
Hence, since $0 \le S_n(2t-1) \le 1$, it follows that 
$$
T_{(x,t)}(y,s)  \ge \frac{S_n\Big(\sqrt{ \frac{\la x,y\ra + t s}{2}}+ \sqrt{1-t}\sqrt{1-s}\Big)}{1 + S_n(2 t-1)} 
    \ge \frac12 S_n(\cos \sd_{\VV_0}((x,t),(y,s))) \ge  \f12 \left(\frac{2}{\pi}\right)^{2r}
$$
for $(y,s) \in \sc((x,t), \frac{2\pi}{2m+1})$. Finally, we can use the fast decay \eqref{eq:fastSn} of $S_n$ 
with $\t \sim \sin\f{\t}{2} =  \sqrt{\f12 (1-\cos\t)}$ to derive an upper bound of $T_{(x,t)}$. Indeed, since evidently 
$
1- \sqrt{\frac{\la x,y\ra + t s}{2}} + \sqrt{1-t}\sqrt{1-s} \ge  1- \sqrt{\frac{\la x,y\ra + t s}{2}} - \sqrt{1-t}\sqrt{1-s},
$
we obtain then
\begin{align*}
 0 \le T_{(x,t)}(y,s) \, & \le c \frac{1}{\left(1+ n \sqrt{1 - \sqrt{ \frac{\la x,y\ra + t s}{2}} - \sqrt{1-t}\sqrt{1-s}}\right)^{2r}}\\
    & = c \frac{1}{\left(1+ n \sqrt{1 -\cos \sd_{\VV_0}((x,t),(y,s))}\right)^{2r}}.
\end{align*}
Hence, by the definition of $\sd_{\VV_0}$ and $1 - \cos \t \sim \t^2$, we obtain
\begin{equation}\label{eq:fast-decayV0}
 0 \le T_{(x,t)}(y,s) \le c \frac{1}{\left(1+ n \sd_{\VV_0} ( (x,t),(y, s))\right)^{2r}}.
\end{equation}
This completes the proof of (1). Finally, $q_n(t) = 1 + S_n(2t-1)$ is a polynomial of degree $n$ and it satisfies
$1 \le q_n(t) \le c$ for all $t \in [0,1]$. Moreover, $q_n(t)T_{x,t}$ is a polynomial of degree $2n$ in $(x,t)$ 
variable. This completes the proof. 
\end{proof}

From Propositions \ref{prop:ChristF1} and \ref{prop:ChristF2}, we have
established the following result. 

\begin{cor} \label{cor:ChristFV0}
Let $\sw$ be a doubling weight function on $\VV_0^{d+1}$. Then 
\begin{equation*}  %\label{eq:ChirstoffelV0}
   \l_n \big(\sw; (x,t) \big)  \le c \, \sw\left(\sc\left((x,t), \tfrac1n \right) \right).
\end{equation*}
Moreover, for $\g \ge -\f12$, 
$$
\l_n \big(\sw_{-1,\g}; (x,t) \big)  \ge c \, \sw_{-1,\g} \left(\sc\left((x,t), \tfrac1n \right) \right) 
   = c n^{-d} \sw_{-1,\g}(n; t).
$$
\end{cor}

This implies, in particular, that the cubature rule in Theorem \ref{thm:cubature} holds for all doubling weights 
on the conic surface. 

\begin{thm}\label{thm:cubatureV0}
Let $\sw$ be a doubling weight on $\VV_0^{d+1}$. Let $\Xi$ be a maximum $\frac{\delta}{n}$-separated 
subset of $\VV_0^{d+1}$. There is a $\delta_0 > 0$ such that for $0 < \delta < \delta_0$ there exist positive 
numbers $\l_{z,r}$, $(z,r) \in \Xi$, so that 
\begin{equation}\label{eq:CFV0}
    \int_{\VV_0^{d+1}} f(x,t) \sw(x,t)  \d \s(x,t) = \sum_{(z,r) \in \Xi }\l_{z,r} f(z,r), \qquad 
            \forall f \in \Pi_n(\VV_0^{d+1}). 
\end{equation}
Moreover, $\l_{z,r} \sim \sw\!\left(\sc((z,r), \tfrac{\delta}{n})\right)$ for all $(z,r) \in \Xi$. 
\end{thm}

We can now construct localized frame on the conic surface. Let us recapitulate necessary definitions in
the setting of $\VV_0^{d+1}$.  
For $j =0,1,\ldots,$ let $\Xi_j$ be a maximal $ \frac \delta {2^{j}}$-separated 
subset in $\VV_0^{d+1}$, so that 
\begin{equation*}
  \int_{\VV_0^{d+1}} f(x,t) \sw(x,t)  \d \s(x,t) = \sum_{(z,r) \in \Xi_j} \l_{(x,r),j} f(z,r), 
     \qquad f \in \Pi_{2^j} (\VV_0^{d+1}).  
\end{equation*}
Let $\wh a$ be an admissible cut-off function satisfying \eqref{eq:a-frame}. Let $\sL_{2^j}(\sw)*f$ denote the 
near best approximation operator defined by 
$$
   \sL_{2^j}(\sw) * f (x): = \int_{\VV_0^{d+1}}f(y,s) \sL_{2^j}(\sw; (x,t),(y,s)) \sw(y,s)  \d \s(y,s).
$$
For $j=1,\ldots,$ we introduce the notation $F_j(\sw) = \sL_{2^j}(\sw)$ for both the kernel and the operator. 
More precisely, $F_0(\sw; \cdot,\cdot) =1$ and for $j \ge 1$,
$$
  F_j (\sw; \cdot,\cdot) =  \sL_{2^{j-1}}(\sw; \cdot,\cdot) \quad\hbox{and}\quad  F_j(\sw) * f = F_j (\sw) * f
$$
and define the frame elements $\psi_{(z,r),j}$ for $(z,r) \in \Xi_j$ by 
$$ 
      \psi_{(z,r),j}(x,t):= \sqrt{\l_{(z,r),j}} F_j(\sw; (x,t), (z,r)), \qquad (x,t) \in \VV_0^{d+1}, 
$$ 
which are well defined for all doubling weight by Theorem \ref{thm:cubatureV0}. By Theorem \ref{thm:frame},
the system $\Phi =\{ \psi_{(z,r),j}: (z,r) \in \Xi_j, \, j =1,2,3,\ldots\}$ is a tight frame of $L^2(\VV_0^{d+1}, \sw)$.

\begin{thm}\label{thm:frameV0}
Let $\sw$ be a doubling weight on $\VV_0^{d+1}$. If $f\in L^2(\VV_0^{d+1}, \sw)$, then
$$ 
   f =\sum_{j=0}^\infty \sum_{(z,r) \in\Xi_j}
            \langle f, \psi_{(z,r), j} \rangle_\sw \psi_{(z,r),j}  \qquad\mbox{in $L^2(\VV_0^{d+1}, \sw)$}
$$  
and
$$ 
\|f\|_{2, \sw}  = \Big(\sum_{j=0}^\infty \sum_{(z,r) \in \Xi_j} |\langle f, \psi_{(z,r),j} \rangle_\sw|^2\Big)^{1/2}.
$$ 
\end{thm}

For the weight function $\sw_{-1,\g}$, the frame elements are highly localized. 

\begin{thm}\label{thm:frameV02}
For $\g \ge -\f12$, the frame for $\sw_{-1,\g}$ is highly localized in the sense that, for every $\k >0$,
there exists a constant $c_\k >0$ such that 
\begin{equation} \label{eq:needleV0}
   |\psi_{(z,r),j}(x,t)| \le c_\s \frac{2^{j d/2}}{\sqrt{ \sw_{\g,d}(2^{j}; t)} \left(1+ 2^j \sd_{\VV_0}((x,t),(z,r))\right)^\k}, 
     \quad (x,t)\in \VV_0^{d+1}.
\end{equation}
%Furthermore, for $0 < p \le \infty$,
%\begin{equation} \label{eq:needleV02}
%   \| \psi_{(z,r),j}\|_{p,\sw_{-1,\g}} \sim \left(n^{-d}\sw_{\g,d}(r;n)\right)^{\f1 p-\f12} \sim 
%       \left(\sw_{-1,\g}((z,r),n^{-1})\right)^{\f1 p-\f12}.
%\end{equation}
\end{thm}
 
\begin{proof}
The inequality \eqref{eq:needleV0} follows from the highly localized estimate of the kernel in
Theorem \ref{thm:kernelV0} and $\l_{(z,r),j} \sim 2^{- jd} \sw_{\g}(2^j;r)$ which holds for  $\sw_{-1,\g}$ 
according to Corollary \ref{cor:ChristFV0}. 
\end{proof}

It is worth pointing out that the localization of the frame elements is established for $\sw_{-1,\g}(t) = t^{-1}(1-t)^{\g}$ but 
not for the Lebesgue measure on $\VV_0^{d+1}$. 
 
\subsection{Characterization of best approximation} 

For $f\in L_p(\VV_0^{d+1}, \sw)$, $1 \le p < \infty$, and $f\in C(\VV_0^{d+1})$ if $p=\infty$, the error of best 
approximation by polynomials of degree at most $n$ is defined by 
$$
   \sE_n(f)_{p,\sw} = \inf_{g \in \Pi_n(\VV_0^{d+1})} \|f - g \|_{p,\sw}, \quad 1 \le p \le \infty. 
$$
Following the study in Section 3, we can give a characterization for this quantity by the modulus of 
smoothness defined via the operator $\sS_{\t, \sw}$ and the $K$-functional defined via the differential
operator $\Delta_{0,\g}$ for $\sw_{-1,\g}$. 

More specifically, for $f\in L^p(\VV_0^{d+1}, \sw_{-1,\g})$ and $r > 0$, the modulus of smoothness 
is defined by
$$
  \o_r(f; \rho)_{p,\sw_{-1,\g}} = \sup_{0 \le \t \le \rho} 
     \left\| \left(I - \sS_{\t,\sw_{-1,\g}}\right)^{r/2} f\right\|_{p,\sw_{-1,\g}}, \quad 1 \le p \le \infty, 
$$
where the operator $\sS_{\t,\sw_{-1,\g}}$ is defined by, for $n = 0,1,2,\ldots$ and $\l = \g+d-1$, 
$$
 \proj_n(\sw_{-1,\g}; \sS_{\t,\sw_{-1,\g}}f) = R_n^{(\l-\f12, -\f12)} (\cos \t) \proj_n(\sw_{-1,\g}; f).
$$
Moreover, in terms of the fractional differential operator $(-\Delta_{0,\g})^{\f r 2}$, the $K$-functional
is defined for a weight $\sw$ by 
$$
   \sK_r(f, t)_{p,\sw} : = \inf_{g \in W_p^r(\VV_0^{d+1}, \sw)}
      \left \{ \|f-g\|_{p,\sw} + t^r\left\|(-\Delta_{0,\g})^{\f r 2}f \right\|_{p,\sw} \right \}.
$$
Both these quantities are well defined as shown in Section 3. Since Assertions 1--4 hold for
$\sw_{-1,\g}$, we only have to verify that the Assertion 5 in Subsection \ref{set:Bernstein} holds 
for the differential operator $\fD_\varpi = \Delta_{0,\g}$. By Theorem \ref{thm:Delta0V0}, 
the kernel $L_n^{(r)}(\varpi)$ in Assertion 5 becomes
$$
    \sL_n^{(r)}\big(\sw_{-1,\g}; (x,t),(y,s)\big) =\sum_{k=0}^\infty \wh a\left(\frac{k}{n} \right) (k(k+\g+d-1))^{\f r 2} 
       \sP_k\big(\sw_{-1,\g}; (x,t),(y,s)\big).
$$

\begin{lem}
Let $\g \ge -\f12$ and $\k > 0$. Then, for $r > 0$ and $(x,t), (y,s) \in \VV_0^{d+1}$, 
$$
 |  \sL_n^{(r)}\big(\sw_{-1,\g}; (x,t),(y,s)\big)| \le c_\k   \frac{n^{d+r}}{\sqrt{ \sw_{\g,d} (n; t) }\sqrt{ \sw_{\g,d} (n; s) }}
\big(1 + n \sd_{\VV_0}( (x,t), (y,s)) \big)^{-\k}.
$$
\end{lem}
\begin{proof}
By \eqref{eq:sfPbCone}, the kernel can be written as 
\begin{align*}
 \sL_n^{(r)}\big(\sw_{-1,\g};(x,t),(y,s)) =  b_{\g,d}  \int_{[-1,1]^2} & L_{n,r}\left(2 \zeta(x,t,y,s;v)^2-1\right) \\
    & \times (1-v_1^2)^{\f{d-2}{2} -1} (1-v_2^2)^{\g-\f12} \d v,
\end{align*}
in which $L_{n,r}$ is defined by, with $\l = \g+d -1$, 
$$
 L_{n,r}(t) = \sum_{k=0}^\infty \wh a\left(\frac{k}{n} \right) (k(k+\g+d-1))^{\f r 2} 
     \frac{P_n^{(\l-\f12, -\f12)}(1)P_n^{(\l-\f12, -\f12)}(t)}{h^{(\l-\f12,-\f12)}}.
$$
Applying \eqref{eq:DLn(t,1)} with $\eta(t) = \wh a(t) \left( t( t + n^{-1} (\g+d-1))\right)^{\f r 2}$ and $m=0$, 
it follows that 
$$
 \left| L_{n,r}(t) \right| \le c n^{r} \frac{n^{2\l+1}}{(1+n\sqrt{1-t})^\ell}.  
$$
Consequently, it follows readily that 
\begin{align*}
|\sL_n^{(r)}\big(\sw_{-1,\g};(x,t),(y,s))| \le  c n^{2\l+r+1}
    \int_{[-1,1]^2}& \frac{1}{ \left(1+n \sqrt{1-\zeta(x,t,y,s;v)^2}\right)^{\k+2\g+d+1}} \\
    & \times (1-v_1^2)^{\f{d-2}{2} -1} (1-v_2^2)^{\g-\f12} \d v.
\end{align*}
Apart from $n^r$, this estimate is the same as that of $L_n(\sw_{-1,\g})$ appeared in the proof of 
Theorem \ref{thm:kernelV0}, so that the desired upper bound follows from the estimate there.  
\end{proof}
 
With all assertions hold for the $\sw_{-1,\g}$, the characterization of the best approximation by polynomials 
in Subsection \ref{sec:bestapp} holds on the conic surface, which we state below. 

\begin{thm}
Let $f \in L^p(\VV_0^{d+1}, \sw)$ if $1 \le p < \infty$ and $f\in C(\VV_0^{d+1})$ if $p = \infty$. 
Le $r > 0$ and $n =1,2,\ldots$. For $\sw = \sw_{-1,\g}$ with $\g \ge -\f12$, there holds 
\begin{enumerate} [   (i)]
\item  direct estimate
$$
  \sE_n(f)_{p,\sw_{-1,\g}} \le c \sK_r (f;n^{-1})_{p,\sw_{-1,\g}};
$$
\item inverse estimate 
$$
   \sK_r(f;n^{-1})_{p,\sw_{-1,\g}} \le c n^{-r} \sum_{k=0}^n (k+1)^{r-1}\sE_k(f)_{p, \sw_{-1,\g}}.
$$
\end{enumerate}
\end{thm}
 
Furthermore, the characterization can be given via the modulus of smoothness, 
since it is equivalent to the $K$-functional. 

\begin{thm} \label{thm:K=omegaV0}
Let $\g \ge -\f12$ and $f \in L_p^r(\VV_0^{d+1}, \sw_{-1,\g})$,  $1 \le p \le \infty$. Then
for $0 < \t \le \pi/2$ and $r >0$ 
$$
   c_1 \sK_r(f; \t)_{p,\sw_{-1,\g}} \le \o_r(f;\t)_{p,\sw_{-1,\g}} \sK_r(f;\t)_{p,\sw_{-1,\g}}.
$$
\end{thm}

%%%%%%%%%%%%
%%     section 5     %%    
%%%%%%%%%%%%
\section{Homogeneous space on solid cones}\label{sec:coneV}
\setcounter{equation}{0}

In this section we work in the setting of homogeneous space on the solid cone 
$$
  \VV^{d+1} = \{(x,t): \|x\|\le t, \, x \in \BB^d, \, 0 \le t \le 1\},
$$
which is bounded by $\VV_0^{d+1}$ and the hyperplane $t =1$ in $\RR^{d+1}$. As in the previous section,
we shall verify that the framework we developed for approximation and tight frame on homogeneous space 
is applicable on this domain with respect to the Jacobi weight function $W_{\g,\mu}(x,t) = (1-t)^{\g} (t^2-\|x\|^2)^{\mu-\f12}$. 

The structure of the section is parallel to that of the conic surface. We again need to define and understand an 
intrinsic distance and doubling weights on the cone and do so in the first two subsections. The orthogonal structure 
with respect to the Jacobi weight is reviewed in the third subsection. The highly localized kernels are established 
in the 
fourth subsection, where the Assertions 1 and 3 are verified for $W_{\g,\mu}$ for all $\mu \ge 0$ but the 
Assertion 2 is established only for $W_{\g,0}$. We also provide construction of $ve$-separated set on 
the cone and state the Marcinkiewicz-Zygmund inequality in the fifth subsection. Assertion 4 is verified in the 
sixth subsection, which allows us to state the positive cubature rules and the tight frame on the cone. While the 
tight localized frame holds for the Jacobi wight with $\mu =0$, the characterization of the best approximation
by polynomials works out for all $\mu \ge 0$ in the seventh subsection.  

\subsection{Distance on solid cone}

The distant function $\sd_{\BB}(\cdot, \cdot)$ on the unit ball $\BB^d$ can be deduced from regarding
$\BB^d$ as the projection of the upper hemisphere of $\SS^{d+1}$, or the distant function $\sd_\SS(\cdot,\cdot)$ 
applied on the points $X= (x,\sqrt{1-\|x\|^2})$ and $Y= (y,\sqrt{1-\|y\|^2})$ with $x \in \BB^d$. The same holds
for the solid cone $\VV^{d+1}$. 

\begin{defn}
For $(x,t)$ and $(y,s)$ on $\VV^{d+1}$, let $X=\big(x,\sqrt{t^2-\|x\|^2}\big)$ and $Y=\big(y,\sqrt{s^2-\|y\|^2}\big)$. 
Define 
\begin{equation} \label{eq:distV}
  \sd_{\VV} ((x,t), (y,s)) : =  \arccos \left(\sqrt{\frac{\la X,Y\ra +ts}2}  + \sqrt{1-t}\sqrt{1-s}\right).
\end{equation}
\end{defn}

\begin{prop}
The function $\sd_{\VV}(\cdot,\cdot)$ defines a distance function on the solid cone $\VV^{d+1}$. 
\end{prop}

\begin{proof}
Since $\|X \| = t$ and $\|Y\|= s$, both $(X,t)$ and $(Y,s)$ are elements of $\VV_0^{d+2}$. Furthermore, 
\begin{equation}\label{eq:dV=dV0}
    \sd_{\VV^{d+1}} ((x,t), (y,s)) =  \sd_{\VV_0^{d+2}} ((X,t), (Y,s)), 
\end{equation}
from which it is easy to see that $\sd_{\VV}(\cdot,\cdot)$ is a distance function on $\VV^{d+1}$. 
\end{proof}

The distance function $\sd_{\VV}(\cdot,\cdot)$ is closely related to the distance function $\sd_{[0,1]}(\cdot,\cdot)$ 
of the interval $[0,1]$ and the distance $\sd_{\BB}(\cdot,\cdot)$ of the unit ball $\BB^d$. 

\begin{prop} \label{prop:cos-d}
For $d \ge 2$ and $(x,t), (y,s) \in \VV^{d+1}$, write $x = t x'$ and $y = s y'$ with $x', y' \in \BB^d$. Then
\begin{equation*}
   1- \cos \sd_{\VV} ((x,t), (y,s)) =1-\cos \sd_{[0,1]}(t,s) + \sqrt{t}\sqrt{s} \left[1-\cos \left(\tfrac{1}{2} \sd_{\BB}(x',y') \right)\right].
\end{equation*}
In particular,  
\begin{equation*} %\label{eq:d2=d2+d2}
  c_1 \sd_{\VV} ((x,t), (y,s)) \le  \sd_{[0,1]}(t,s) +   (t s)^{\f14} \sd_{\BB}(x',y') \le  c_2  \sd_{\VV} ((x,t), (y,s)).
\end{equation*}
\end{prop}

\begin{proof}
With $x= t x'$, we see that $X = t (x', \sqrt{1-\|x'|^2})$ and $Y = s (y', \sqrt{1-\|y'|^2})$. Hence, it follows as
in the case of $\VV_0^{d+1}$ that 
\begin{equation} \label{eq:dV=dB+dt}
   \sd_{\VV} ((x,t), (y,s)) = \arccos \left[\sqrt{t}\sqrt{s} \cos \left(\tfrac{1}{2} \sd_{\BB}(x', y') \right)+ \sqrt{1-t}\sqrt{1-s} \right].
\end{equation}
The rest of the proof follows from that of Proposition \ref{eq:cos-dist} almost verbatim.
\end{proof}

Like the distance function on the surface, we see that the distance on the line segment from the apex to 
$(x',1)$ on the top boundary of the cone $\VV^{d+1}$ is exactly the distance function on $[0,1]$; that is,
for $x'\in \BB^d$, 
$$
   \sd_{\VV}((x',t), (x',s)) = \sd_{[0,1]}(t,s), \qquad  0 \le t, s  \le 1.
$$

The following lemma is an analog of Lemma \ref{lem:|s-t|} and will be needed in the next 
subsection. 

\begin{lem} \label{lem:|s-t|V}
For $(x,t), (y,s) \in \VV^{d+1}$, 
$$
  \big| \sqrt{t} - \sqrt{s} \big|\le \sd_{\VV} ((x,t), (y,s)) \quad \hbox{and} \quad 
      \big| \sqrt{1-t} - \sqrt{1-s} \big| \le \sd_{\VV} ((x,t), (y,s)),
$$
and 
$$
 \big| \sqrt{t^2-\|x\|^2} - \sqrt{s^2-\|y\|^2} \big|\le  c \max\{\sqrt{s},\sqrt{t} \} \, \sd_{\VV} ((x,t), (y,s)). 
$$
\end{lem} 

\begin{proof}
The first two inequalities are immediate consequence of \eqref{eq:dV=dV0} and Lemma \ref{lem:|s-t|}. 
Let $x= t x'$ and $y =s y'$ with $x',y' \in \BB^d$. We know \cite[(A.1.4)]{DaiX} that 
\begin{equation} \label{eq:inequaBall}
  \left| \sqrt{1-\|x'\|^2} -  \sqrt{1-\|y'\|^2}\right| \le \sqrt{2} \sd_\BB(x',y').  
\end{equation}
Without loss of generality, assuming $t \ge s$. Then $s \le \sqrt{s}\sqrt{t}$, so that
\begin{align*}
  \left| \sqrt{t^2-\|x\|^2} -  \sqrt{s^2-\|y\|^2}\right| & = \left| t \sqrt{1-\|x'\|^2} - s\sqrt{1-\|y'\|^2}\right| \\
   & \le  |t-s| \sqrt{1-\|x'\|^2} + s  \left| \sqrt{1-\|x'\|^2} - \sqrt{1-\|y'\|^2}\right| \\
   & \le \sqrt{2} \max\{\sqrt{s},\sqrt{t} \} \left(  \big|\sqrt{t} - \sqrt{s}\big|  +  (ts)^{\f14} \,\sd_\BB(x',y')\right),
\end{align*}
which is bounded by $c \max\{\sqrt{s},\sqrt{t} \} \, \sd_{\VV} ((x,t), (y,s))$ by the first inequality in the
statement and the inequality in Proposition \ref{prop:cos-d}. 
\end{proof}

\subsection{A family of doubling weights}
For $d \ge 1$, $\mu \ge 0$, $\b+2\mu > -d$ and $\g > -1$, we consider the Jacobi weight function 
defined on the solid cone by
$$
    W_{\b,\g,\mu}(x,t) := (t^2-\|x\|^2)^{\mu-\f12} t^\b (1-t)^\g,  \quad (x,t) \in \VV^{d+1}.
$$
Let $\bb_{\b,\g,\mu}$ be the normalization constant so that $\bb_{\b,\g,\mu} W_{\b,\g,\mu}$ has unit 
integral on $\VV^{d+1}$. Setting $x= t x'$ with $x' \in \BB^d$, then 
$$
   \bb_{\b,\g,\mu} =\bigg(\int_0^1 t^{d +\b+ 2\mu -1} (1-t)^\g \d t  \int_{\BB^d} (1-\|x'\|^2)^{\mu-\f12}   \d x' \bigg)^{-1}
           = c_{\b+2\mu+d-1, \g} b_\mu^\BB,
$$
where $c_{\a,\b}$ is defined in \eqref{eq:c_ab} and $b_\mu^\BB$ is the normalization constant of the
classical weight function on the unit ball. The case $\b =0$ is of particular interest and will be denoted by
$W_{\g,\mu}$; that is, 
$$
    W_{\g,\mu}(x,t):= W_{0,\g,\mu} (x,t) = (t^2-\|x\|^2)^{\mu-\f12}(1-t)^\g,  \quad (x,t) \in \VV^{d+1}.
$$

We show that this is a doubling weight with respect to the distance $\sd_\VV$. For $r > 0$ and $(x,t)$ in 
$\VV^{d+1}$, we denote the ball centered at $(x,t)$ with radius $r$ by 
$$
      \cb((x,t), r): = \big\{ (y,s) \in \VV^{d+1}: \sd_{\VV} \big((x,t),(y,s)\big)\le r \big\}.
$$   

\begin{prop}\label{prop:capV}
Let $r > 0$ and $(x,t) \in \VV^{d+1}$. Then for $\b > -d$, $\g > -1$ and $\mu \ge 0$, 
\begin{align}\label{eq:capV}
   W_{\b,\g,\mu}\big(\cb((x,t), r)\big) & := \int_{ \cb((x,t), r)} W_{\b,\g,\mu}(y,s) \d y \d s \\
     & \sim r^{d+1} (t+ r^2)^{\b+\f{d-1}{2}} (1-t+ r^2)^{\g+\f12} (t^2-\|x\|^2 + r^2)^\mu.  \notag
\end{align}
In particular, $W_{\b,\g,\mu}$ is a doubling weight and the doubling index $\a(W_{\b,\g,\mu})$
is give by $\a(W_{\b,\g,\mu}) = 2\mu+d+1 + 2 \max\{0, \b+\f{d-1}2\} + 2 \max\{0,\g+\f12\}$.
\end{prop}

\begin{proof} 
Let $\tau_r(t,s)$ and $\t_r(t,s) = \arccos \tau_r(t,s)$ be as in the proof of Proposition \ref{prop:capV0}. 
From $\sd_{\VV}((x,t),(y,s)) \le r$, we obtain $\sd_{[0,1]}(t, s) \le r$ and, by \eqref{eq:distV},
$$
  \sd_{\BB} (x', y') \le \arccos \left(2 [\tau_r(t,s)]^2 -1\right) = \tfrac12 \arccos \tau_r(t,s) = \tfrac12 \t_r(t,s), 
$$
where $\sd_\BB(\cdot,\cdot)$ is the distance on the unit ball $\BB^d$. Hence, it follows that
\begin{align*}
 & W_{\b,\g,\mu}\big(\cb((x,t), r)\big)  = \int_{\sd_{[0,1]}(t, s)\le r} s^{d} 
        \int_{\sd_{\BB}(x',y') \le \tfrac12 \t_r(t,s)} W_{\b,\g,\mu}(sy',s)  \d y' \d s\\
       & \qquad\sim  \int_{\sd_{[0,1]}(t, s)\le r} s^{d+\b + 2\mu -1}(1-s)^\g  
        \int_{\sd_{\BB}(x',y') \le \tfrac12 \t_r(t,s)} (1-\|y'\|^2)^{\mu-\f12} \d y' \d s  
\end{align*}
For $\mu \ge 0$ and $0 < \rho \le 1$, it is known \cite[Lemma 5.3]{PX2} or \cite[p. 107]{DaiX} that  
\begin{equation}\label{eq:doublineB}
 \int_{\sd_{\BB}(x',y') \le \rho} (1-\|y'\|^2)^{\mu-\f12} \d y' \sim (1-\|x'\|^2 + \rho^2)^\mu \rho^d,
\end{equation}
which can also be established by following the approach in Proposition \ref{prop:capV0}. Consequently, 
using $\t_r(t,s) \sim \sqrt{1-\tau_r(t,s)}$, we conclude that 
\begin{align*}
  W_{\b,\g,\mu}\big(\cb((x,t), r)\big)  
     \sim  \,& \int_{\sd_{[0,1]}(t, s)\le r}  s^{d+\b + 2\mu -1}(1-s)^\g  \\ 
         & \quad \times \big(1-\|x'\|^2 + 1-\tau_r(t,s) \big)^\mu \big(1-\tau_r(t,s)\big)^{\f{d}{2}} \d s.
\end{align*}
If $ t \ge 3 r^2$, then $s \sim t +r^2$ and, by \eqref{eq:tau_rV0}, it follows that 
$$
     s^\mu \big(1-\|x'\|^2 + 1-\tau_r(t,s) \big)^\mu \sim  (t^2-\|x\|^2 + r^2)^\mu. 
$$
With this term removed, the integral of the remaining integrand can be estimated by following the 
estimates of Case 1 and Case 3 of the proof of Proposition \ref{prop:capV0}. If $t \le 3 r^2$, then 
$t^2 - \|x\|^2 + r^2 \sim r^2$. We use $1-\|x'\|^2 + 1-\tau_r(t,s) \le 2$ for the upper bound and 
$1-\tau_r(t,s) \ge 2/\pi^2$ on the subset $\sd_{[0,1]}(t,s)\le r/2$, proved in the Case 2 of the proof 
of Proposition \ref{prop:capV0}, to remove the term $ \big(1-\|x'\|^2 + 1-\tau_r(t,s) \big)^\mu$ from the 
integral. The rest of the proof then follows from that of the Case 2 of the proof of Proposition \ref{prop:capV0}.
This completes the proof. 
\end{proof}

\begin{cor}
For $d\ge 2$, $\b > -d$, $\g > -1$ and $\mu \ge 0$, $(\VV^{d+1}, W_{\b,\g,\mu}, \sd_{\VV})$ is a 
homogeneous space. 
\end{cor}

\subsection{Orthogonal polynomials on the solid cone}
These polynomials are studied in \cite{X20a} and they are orthogonal with respect to the inner product 
$$
\la f, g \ra_{\b,\g,\mu} =\bb_{\b,\g,\mu}\int_{\VV^{d+1}} f(x,t) g(x,t)W_{\b,\g,\mu}(x,t) \d x \d t.
$$
Let $\CV_n(\VV^{d+1},W_{\b,\g,\mu})$ be the space of orthogonal polynomials of degree $n$, which has 
dimension $\binom{n+d+1}{n}$. An orthogonal basis of $\CV_n(\VV^{d+1}, W_{\b,\g,\mu})$ can be given 
in terms of the Jacobi polynomials and the orthogonal polynomials on the unit ball. For $m =0,1,2,\ldots$, 
let $\{P_\kb^m(W_\mu): |\kb| = m, \kb \in \NN_0^d\}$ be an orthonormal basis of $\CV_n(\BB^d, W_\mu)$ 
on the unit ball. Let $\alpha: = \mu+\frac{\b+d-1}{2}$ and 
\begin{equation} \label{eq:coneJ}
  \Jb_{m,\kb}^n(x,t):= P_{n-m}^{(2\a+2m, \g)}(1- 2t) t^m P_\kb^m\left(W_\mu; \frac{x}{t}\right), \quad 
\end{equation}
Then $\{\Jb_{m,\kb}^n(x,t): |\kb| = m, \, 0 \le m\le n\}$ is an orthogonal basis of $\CV_n(\VV^{d+1},W_{\b,\g,\mu})$
and the norm of $\Jb_{m,\kb}^n$ is given by 
\begin{equation} \label{eq:coneJnorm}
    \Hb_{m,n}^{(\a,\g)}:=  \la \Jb_{m, \kb}^n, \Jb_{m, \kb}^{n} \ra_{W_{\mu,\b,\g}} 
        =  \frac{ c_{2\a,\g}} {c_{2\a+2m,\g} } h_{n-m}^{(2\a+2m, \g)},
     %\frac{c_{2\mu +d-1,\g}} {c_{2\mu+2m +d -1,\g} } h_{n-m}^{(2\mu+2m+d-1, \g)},
\end{equation}
where $c_{\a,\b}$ is as in \eqref{eq:c_ab} and $h_m^{(\a,\g)}$ is the norm square of the Jacobi polynomial.
We call $\Jb_{m, \kb}^n$ the Jacobi polynomials on the cone. 
 
The reproducing kernel of the space 
$\CV_n(\VV^{d+1}, W_{\b,\g,\mu})$, denoted by $\Pb_n(W_{\b,\g,\mu};\cdot,\cdot)$, can be written 
in terms of the above basis, 
$$
\Pb_n\big(W_{\b,\g,\mu}; (x,t),(y,s) \big) = \sum_{m=0}^n \sum_{|\kb|=n} 
    \frac{  \Jb_{m, \kb}^n(x,t)  \Jb_{m, \kb}^n(y,s)}{\Hb_{m,n}^{\b,\g}}.
$$
It is the kernel of the projection $\proj_n(W_{\b,\g,\mu}): L^2(\VV^{d+1},W_{\b,\g,\mu}) \to 
\CV_n(\VV^{d+1}, W_{\b,\g,\mu})$, 
$$
\proj_n(W_{\b,\g,\mu};f) = \int_{\VV^{d+1}} f(y,s) \Pb_n\big(W_{\b,\g,\mu}; \,\cdot, (y,s) \big)  W_{\b,\g,\mu}(s) \d y \d s.
$$

In contrast to the conic surface, it is the case $\b= 0$ that turns out to be the most interesting one.
For $\b =0$, there is a second order differential operator that has orthogonal polynomials as eigenfunctions. 
Recall that $W_{0,\g,\mu} = W_{\g,\mu}$.

\begin{thm}\label{thm:Delta0V}
Let $\mu > -\tfrac12$, $\g > -1$ and $n \in \NN_0$. Define the second order differential operator 
$\fD_{\mu,\g}$ by 
\begin{align*}
  \fD_{\g,\mu} : = & \, t(1-t)\partial_t^2 + 2 (1-t) \la x,\nabla_x \ra \partial_t + \sum_{i=1}^d(t - x_i^2) \partial_{x_i}^2
        - 2 \sum_{i<j } x_i x_j \partial_{x_i} \partial_{x_j}  \\ 
  &   + (2\mu+d)\partial_t  - (2\mu+\g+d+1)( \la x,\nabla_x\ra + t \partial_t),  
\end{align*}
where $\nabla_x$ and $\Delta_x$ denote the gradient and the Laplace operator in $x$-variable. Then the 
polynomials in $\CV_n(\VV^{d+1},W_{\g,\mu})$ are eigenfunctions of $\fD_{\g,\mu}$,  
\begin{equation}\label{eq:cone-eigen}
   \fD_{\g,\mu} u =  -n (n+2\mu+\g+d) u, \qquad \forall u \in \CV_n(\VV^{d+1},W_{\g,\mu}).
\end{equation}
\end{thm} 
 
The reproducing kernel enjoys an addition formula that is a mixture of the addition formula for orthogonal 
polynomials with respect to the classical weight function $W_\mu$ on the unit ball and the Jacobi polynomials. 
The addition formula is complicated and has the most elegant form when $\b =0$. 

\begin{thm} \label{thm:PnCone2}
Let $d \ge 2$. For $\mu \ge 0$ and $\g \ge -\f12$, let $\a = \mu + \frac{d-1}{2}$; then 
\begin{align}\label{eq:PbCone2}
  \Pb_n \big(W_{\g,\mu}; (x,t), (y,s)\big) =\, & 
   c_{\mu,\g,d} \int_{[-1,1]^3}  Z_{2n}^{2 \a+\g+1} (\xi (x, t, y, s; u, v)) \\
     &\times   (1-u^2)^{\mu-1} (1-v_1^2)^{\a - 1}(1-v_2^2)^{\g-\f12}  \d u \d v, \notag
\end{align}
where $c_{\mu,\g,d}$ is a constant, so that $\Pb_0 =1$ and 
$\xi (x,t, y,s; u, v) \in [-1,1]$ is defined by 
\begin{align} \label{eq:xi}
\xi (x,t, y,s; u, v) = &\, v_1 \sqrt{\tfrac12 \left(ts+\la x,y \ra + \sqrt{t^2-\|x\|^2} \sqrt{s^2-\|y\|^2} \, u \right)}\\
      & + v_2 \sqrt{1-t}\sqrt{1-s}. \notag
\end{align}
When $\mu = 0$ or $\g = -\f12$, the identity \eqref{eq:PbCone2} holds under the limit \eqref{eq:limitInt}.
In particular, 
\begin{align}\label{eq:PbCone3}
   \Pb_n \big(W_{\g, 0}; & (x,t), (y,s)\big) =  c_{0,\g,d} \\
  & \times \int_{[-1,1]^2}   \left[ Z_{2n}^{\g+d} (\xi (x, t, y, s; 1, v))  
    +Z_{2n}^{\g+d} (\xi (x, t, y, s; -1, v))\right] \notag \\
     &\qquad\qquad\quad \times (1-v_1^2)^{\f{d-1}2 - 1}(1-v_2^2)^{\g-\f12}  \d v. \notag
\end{align}
\end{thm}

As a consequence, the orthogonal structure for the weight function $W_{\g,\mu}$ on the cone $\VV^{d+1}$ 
satisfies both characteristic properties specified in Definition \ref{def:LBoperator} and Definition \ref{defn:additionF}.
In the next subsection, we use the addition formula to establish Assertions 1--3 for highly localized kernels. 

\subsection{Highly localized kernels}

Let $\wh a$ be an admissible cut-off function.  For $(x,t)$, $(y,s) \in \VV^{d+1}$, the localized kernel 
$\Lb_n(W_{\g,\mu}; \cdot,\cdot)$ is defined by  
$$
   \Lb_n\left(W_{\g,\mu}; (x,t),(y,s)\right) = \sum_{j=0}^\infty \wh a\left( \frac{j}{n} \right)
       \Pb_j\left(W_{\g,\mu}; (x,t), (y,s)\right). 
$$
For $\g > -1$ and $\mu \ge 0$, we need the function $n^{d+1}W_{\g,\mu}\big(\cb((x,t),n^{-2})\big)$, 
which we denote by $W_{\g,\mu} (n; x, t)$; that is, 
$$
    W_{\g,\mu} (n; x, t) :=  \big(t+n^{-2}\big)^{\f{d-1}2} \big(1-t+n^{-2}\big)^{\g+\f12}\big(t^2-\|x\|^2+n^{-2}\big)^{\mu}. 
$$

\begin{thm} \label{thm:kernelV}
Let $d\ge 2$, $\g \ge -\f12$ and $\mu \ge 0$. Then for any $k>0$ and $(x,t), (y,s) \in \VV^{d+1}$, 
\begin{equation*}%\label{V0-bound}
\left |\Lb_n \left(W_{\g,\mu}; (x,t), (y,s)\right)\right|
   \le \frac{c_\k n^{d+1}}{\sqrt{ W_{\g,\mu} (n; x, t) }\sqrt{ W_{\g,\mu} (n; y, s) }}
    \big(1 + n  \sd_{\VV}( (x,t), (y,s)) \big)^{-\k}.
\end{equation*}
\iffalse 
\item if $\wh a$ is of type (c) and $0 < \varepsilon \le 1$, then for $(x,t), (y,s) \in \VV^{d+1}$, 
\begin{align*}%\label{simplex-bound1}
\left |\Lb_n \left(W_{0,\g,\mu}; (x,t), (y,s)\right)\right|
\le \, & \frac{c_\s n^{d+1}}{\sqrt{W_{\g,\mu}(n; x, t)}\sqrt{W_{\g,\mu}(n; x, t)}} \\
   & \times \exp\left \{-\frac{c' n \sd_{\VV}((x,t), (y,s))}{[\ln(e+n \sd_{\VV}((x,t), (y,s)))]^{1+\ve}}\right \}.
\end{align*} 
\end{enumerate}
\fi
\end{thm}

\begin{proof}
The proof follows along the line of the estimate on the surface of the cone, but it will be more involved 
as can be seen in the lower bound of $1- \xi(x,t,y,s; u,v)$ at \eqref{eq:1-xi-lbd} below. 

The kernel $\Lb_n(W_{\mu,\g})$ can be written in terms of the kernel of the Jacobi polynomials by the 
addition formula in Theorem \ref{thm:PnCone2}, which gives, with $\a = \mu+\f{d-1}{2}$ and $\l =  2\mu + \g +d$, 
that 
\begin{align*}
\Lb_n \left(W_{\g,\mu}; (x,t), (y,s) \right)=  
  c_{\mu,\g,d}  \int_{[-1,1]^3} & L_n ^{(\l-\f12,-\f12)}\big(2 \xi (x,t,y,s; u, v)^2-1 \big)\\
  &  \times    (1-u^2)^{\mu-1} (1-v_1^2)^{\a-1}(1-v_2^2)^{\g-\f12}\d u \d v.
\end{align*}
Then, as in the proof of Theorem \ref{thm:kernelV0}, we can apply the estimate \eqref{eq:Ln(t,1)} for $L_n^{\a,\b}$ 
with $\alpha=\l-1/2$, $\beta= -1/2$ to obtain
\begin{align*}
 \left| \Lb_n \left(W_{\g,\mu}; (x,t), (y,s)\right) \right| \le c n^{2 \l +1} \int_{[-1,1]^3} 
 & \frac{1}{ \left(1+n \sqrt{1- \xi (x,t,y,s; u,v)^2} \right)^{\k+5\mu+3\g+\frac{3d}2+1} }\\
       &   \times (1-u^2)^{\mu-1} (1-v_1^2)^{\a-1}(1-v_2^2)^{\g-\f12}\d u \d v.
\end{align*}
We first need an lower bound for $\xi(x,t,y,s; u,v)$. We simplify the notation and write
$$
\xi (x,t, y,s; u, v) =  v_1 \sqrt{\eta(x,t,y,s;u)}+ v_2 \sqrt{1-t}\sqrt{1-s}
$$
by introducing the notation 
$$
\eta(x,t,y,s;u)  = \tfrac12 \left(ts+\la x,y \ra + \sqrt{t^2-\|x\|^2} \sqrt{s^2-\|y\|^2} \, u\right).
$$
Let $x = t x'$ and $y = s y'$, where $x',y'\in \BB^d$. Then by $|u|\le 1$, it is easy to verify that $0 \le \eta(x,t,y,s;u)
 \le t s$
and, consequently, $|\xi(x,t,y,s;u,v)| \le \sqrt{t s} + \sqrt{1-s^2}\sqrt{1-t^2} \le 1$. Furthermore, we can write
\begin{align}\label{eq:1-xiV}
  1- \xi(x,t,y,s;u,v) = \, & 1 - \sqrt{\eta(x,t,y,s;u)}  - \sqrt{1-t}\sqrt{1-s} \\
   & + (1-v_1) \sqrt{\eta(x,t,y,s;u)} + (1-v_2) \sqrt{1-t}\sqrt{1-s} \notag \\
    \ge\, & 1 - \sqrt{\eta(x,t,y,s;1)}  - \sqrt{1-t}\sqrt{1-s} \notag \\
    = \,& 1 - \cos \sd_{\VV}((x,t),(y,s)) = 2 \sin^2 \frac{ \sd_{\VV} ((x,t), (y,s)) }{2}\notag \\
   \ge\, & \frac{2}{\pi^2} [\sd_{\VV}((x,t),(y,s))]^2. \notag
\end{align}
As in the case of $\VV_0^{d+1}$, we use this inequality to obtain the estimate
\begin{align*}
   \left| \Lb_n \left(W_{\g,\mu}; (x,t), (y,s)\right) \right| \le c n^{2 \l +1}
       \frac{1}{(1+  n \sd_{\VV}((x,t),(y,s)))^{\k+\mu+\g+\f{d}2}} I(x,t,y,s) 
\end{align*}
with the integral $I(x,t, y,s)$ defined by 
$$
 I(x,t,y,s) = c_{\g,\mu} \int_{[-1,1]^3}    
        \frac{(1-u^2)^{\mu-1} (1-v_1^2)^{\a-1}(1-v_2^2)^{\g-\f12}} {(1+n\sqrt{1- \xi (x,t,y,s;u,v)})^{4\mu+2\g+d+1}} \d u \d v,
$$
where $c_{\g,\mu}$ is the normalization constant so that $I(x,t,y,s)=1$ when $n =0$.
\iffalse
In the above inequality, we restrict the integral to $[0,1]^3$ by using the symmetry of the weight function and
trivial inequalities that replacing negative $v_i$ by positive $v_i$, then negative $u$ by positive $u$. 
\fi
To complete the proof, we need to show that
$$
\frac{ I(x,t,y,s)}{(1+  n \sd_{\VV}((x,t),(y,s)))^{\mu+\g+\f{d}{2}}} \le \frac{cn^{- (4\mu+2\g+d)}} 
   {\sqrt{W_{\g,\mu}(n;x, t)}\sqrt{W_{\g,\mu}(n; y, s)}}.
$$
This is a special case of Lemma \ref{lem:kernelV} established below. 
\end{proof}

\begin{lem}\label{lem:kernelV}
Let $d \ge 2$ and $\g\ge -\f12$ and $\mu \ge 0$. Then for $\b \ge 4\mu+2\g+d+1$,
\begin{align}\label{eq:kernelV}
  c_{\g,\mu}& \int_{[-1,1]^3}    
        \frac{(1-u^2)^{\mu-1} (1-v_1^2)^{\a-1}(1-v_2^2)^{\g-\f12}} {(1+n\sqrt{1- \xi (x,t,y,s;u,v)})^{\b}} \d u \d v \\
    &\le  \frac{cn^{- (4\mu+2\g+d)}} 
   {\sqrt{W_{\g,\mu}(n;x, t)}\sqrt{W_{\g,\mu}(n; y, s)}(1+  n \sd_{\VV}((x,t),(y,s)))^{\b - 5\mu -3\g- \f{3d}{2}-1}}. \notag
\end{align}
\end{lem}

\begin{proof}
Let $I(x,t,y,s)$ be defined as in the proof of the previous proof. Using the lower bound of $\xi$ in \eqref{eq:1-xiV}, 
we see that the left-hand side of \eqref{eq:kernelV} is bounded by
$$
\frac{1}{  (1+  n \sd_{\VV}((x,t),(y,s)))^{\b- 4\mu-2\g-d-1} } I(x,t,y,s).
$$
In order to estimate $I(x,t,y,s)$ we need a refined lower bound for $1- \xi(x,t,y,s;u,v)$. Let $x = t x'$ 
and $y = s y'$ with $x', y' \in \BB^d$. We claim that for $-1 \le v_1, v_2,u \le 1$
\begin{align} \label{eq:1-xi-lbd}
  1 - \xi (x,t,y,s;u,v) \, & \ge   \frac14 \sqrt{t s}(1-v_1) + \sqrt{1-s}\sqrt{1-t} (1-v_2)\\  
      & + \frac14 \sqrt{ts} \sqrt{1-\|x'\|^2} \sqrt{1-\|y'\|^2}(1-u). \notag
\end{align}
To prove this lower bound, we write $A(x',y') = \sqrt{1-\|x'\|^2} \sqrt{1-\|y'\|^2}$ so that the notation can
be further simplified as 
\begin{align*}
\eta(x,t,y,s;u) =\tfrac{1} 2 t s \big(1+\la x',y' \ra + A(x',y') u\big).
\end{align*}
We shall use $\xi$ and $\eta$ without their variables below. Our first step is to write 
\begin{align*}
 1- \xi \, & = 1- \sqrt{ts} -  \sqrt{1-t}\sqrt{1-s} + \sqrt{ts} - \sqrt{\eta} + (1-v_1) \sqrt{\eta} + (1-v_2) \sqrt{1-t}\sqrt{1-s}\\
   & \ge \sqrt{ts} - \sqrt{\eta} + (1-v_1) \sqrt{\eta} + (1-v_2) \sqrt{1-t}\sqrt{1-s}.
\end{align*}
Next, since $\eta(x,t,y,s;u) \le t s$ by the Cauchy's inequality, we obtain 
\begin{align*}
  \sqrt{ts} - \sqrt{\eta} \, &= \frac{ ts - \eta}{ \sqrt{ts} + \sqrt{\eta}} \ge\frac{ ts}{4\sqrt{ts}} \left(1 -\la x',y' \ra - A(x',y')\, u\right)\\
 & =  \frac{ \sqrt{ts}}{4} \big(1 -\la x',y' \ra - A(x',y') + (1- u) A(x',y')\big). 
\end{align*} 
Since $1 -\la x',y' \ra - A(x',y') \ge 0$ and $1 \ge 1- v_1$, this leads to 
\begin{align*}
   1- \xi \, & \ge  B \sqrt{ts} (1-v_1)   
       +   \sqrt{1-s}\sqrt{1-t} (1-v_2)+ \frac{\sqrt{ts}}4 A(x',y') (1-u),
\end{align*}
where $B$ is given by 
\begin{align*}
   B  = \frac14 \big(1 -\la x',y' \ra - A(x',y')\big) +  \sqrt{\tfrac{1} 2 \big(1+\la x',y' \ra + A(x',y') u\big)},
\end{align*}
which can be rewritten to give
\begin{align*}
   B \, & = 1- \frac14 (1-u) A(x',y') - \frac12 \left(1-  \sqrt{\tfrac{1} 2 \big(1+\la x',y' \ra + A(x',y') u\big)}\right)^2\\
      & \ge 1- \f14 - \f12 = \f14. 
\end{align*}
This proves \eqref{eq:1-xi-lbd}. We are now ready to prove \eqref{eq:kernelV}. Denote its left-hand side 
by $J(x,t,u,v)$. Using the inequality \eqref{eq:1-xi-lbd}, we obtain
$$
 J(x,t,y,s) \le c  \int_{[0,1]^3}    
        \frac{(1-u^2)^{\mu-1} (1-v_1^2)^{\mu+\f{d-1}{2}-1}(1-v_2^2)^{\g-\f12}} {\left(1+n\sqrt{B_1 (1-v_1)+B_2(1-v_2)+
            B_3(1-u)}\right)^{4\mu+2\g+d+1}} \d u \d v,       
$$
where $B_1 = \f14\sqrt{t}\sqrt{s}$, $B_2 = \sqrt{1-s}\sqrt{1-t}$ and $B_3 = \frac14 \sqrt{ts} \sqrt{1-\|x'\|^2} \sqrt{1-\|y'\|^2}$.
We can estimate the integral as in the proof of Theorem \ref{thm:kernelV0} by applying \eqref{eq:B+At}, which
we now need to apply three times, and we obtain 
\begin{align*}
  J(x,t,y,s) \, &\le c  \frac{n^{-2\mu}}{B_3^{\mu}} \int_{[0,1]^2}    
        \frac{ (1-v_1^2)^{\mu+\f{d-1}{2}-1}(1-v_2^2)^{\g-\f12}} 
              {\left(1+n\sqrt{B_1 (1-v_1)+B_2(1-v_2)}\right)^{2\mu+2\g+d}}\d v \\  
        &\le c  \frac{n^{-2\mu-2\g-1}}{B_3^{\mu}B_2^{\g+\f12}} \int_0^1
        \frac{ (1-v_1^2)^{\mu+\f{d-1}{2}-1}} {\left(1+n\sqrt{B_1 (1-v_1)}\right)^{2\mu +d-1}}\d v_1 \\
         &\le c  \frac{n^{-4\mu-2\g-d}}{B_3^{\mu}B_2^{\g+\f12}B_1^{\mu+\f{d-1}2}}.
\end{align*}
Combining part of $B_1$ and $B_3$, it follows from $\sqrt{t}\sqrt{s} B_3 =  \f14 \sqrt{t^2-\|x\|^2} \sqrt{s^2-\|y\|^2}$
 that 
\begin{align*}
  & J(x,t,y,s)   \le c \frac{n^{-4\mu-2\g-d}}{\left(\sqrt{t^2-\|x\|^2} \sqrt{s^2-\|y\|^2}\right)^{\mu} 
      \left(\sqrt{1-t}\sqrt{1-s}\right)^{\g+\f12}
       \left(\sqrt{t}\sqrt{s}\right)^{\f{d-1}2}}\\
     & \le c \frac{n^{-4\mu-2\g-d}}{ \left(\sqrt{t^2-\|x\|^2} \sqrt{s^2-\|y\|^2}+n^{-2}\right)^{\mu}
          \left(\sqrt{1-t}\sqrt{1-s}+n^{-2}\right)^{\g+\f12} \left (\sqrt{t}\sqrt{s}+n^{-2}\right)^{\f{d-1}2}}, 
\end{align*}
where the second inequality follows since $I(x,t,y,s) \le 1$ holds trivially.  We can now apply 
\eqref{eq:ab+=a+b+} three times to obtain 
\begin{align*}
  J(x,t,y,s) \le &\, c \frac{n^{-4\mu-2\g-d}}{\sqrt{W_{\mu,\g}(n; x,t)}\sqrt{W_{\mu,\g}(n; y,s)}}  
     \big(1+  n \big|\sqrt{t^2-\|x\|^2}-\sqrt{s^2-\|y\|^2}\big|\big)^{\mu} \\
   & \times \big(1+  n \big|\sqrt{t}-\sqrt{s}\big|\big)^{\f{d-1}2} \big(1+  n \big|\sqrt{1-t}-\sqrt{1-s}\big|\big)^{\g+\f12}.
\end{align*}
Finally, by Lemma \ref{lem:|s-t|V}, we conclude  
$$
  J(x,t,y,s) \le c \frac{n^{-4\mu-2\g-d}(1+  n \sd_{\VV}((x,t),(y,s)))^{\mu + \g+\f{d}{2}}}
            {\sqrt{W_{\g,\mu}(n; x,t)}\sqrt{W_{\g,\mu}(n; y,s)}}
$$
which is what we need to complete the proof. 
\end{proof}
 
The theorem we just proved verifies Assertion 1 for $W_{\g,\mu}$. Our next theorem verifies
Assertion 2 but only for $W_{\g,0}$. Recall that $W_{\g,0}$ depends only on $t$,
$$
    W_{\g,0} (n; x, t) =  \big(t+n^{-2}\big)^{\f{d-1}2} \big(1-t+n^{-2}\big)^{\g+\f12}, 
$$
 
\begin{thm} \label{thm:L-LkernelV}
Let $d\ge 2$, $\g \ge -\f12$. For $(x_i,t_i), (y,s) \in \VV^{d+1}$, 
and $(x_1,t_1) \in \sc ((x_2,t_2), c^* n^{-1})$ with $c^*$ small, and any $\k > 0$, 
\begin{align}\label{eq:L-LkernelV}
&\left |\Lb_n \left(W_{\g,0}; (x_1,t_1), (y,s)\right)-\Lb_n \left(W_{\g,0}; (x_2,t_2), (y,s)\right)\right| \\
  & \qquad\qquad
  \le \frac{c_\k n^{d+2} \sd_{\VV}( (x_1,t_1), (x_2,t_2))}
  {\sqrt{ W_{\g,0} (n; x_2, t_2) }\sqrt{ W_{\g,0} (n; y, s) }  \big(1 + n  \sd_{\VV}( (x_2,t_2), (y,s)) \big)^{\k}}. \notag
\end{align}
\end{thm}
 
\begin{proof}
Let $K$ denote the left-hand side of \eqref{eq:L-LkernelV}. As in the proof of Theorem \ref{thm:L-LkernelV0}, 
we use the integral expression of $\Lb_n (W_{\g,0})$ and \eqref{eq:PbCone3} to obtain $K \le K_1 + K_{-1}$, 
where   
\begin{align*}
 K_u & \le 2  \int_{[-1,1]^2} \big\| \partial L_n^{\g+d-\f12,-\f12} \big(2(\cdot)^2-1\big)\big\|_{L^\infty(I_{u,v})} 
     \big |\xi_1(u,v)^2 - \xi_2(u,v)^2| \\  
  & \qquad \qquad\qquad
         \times  (1-v_1^2)^{\frac{d-1}{2}-1} (1-v_2^2)^{\g-\f12} \d v,  
\end{align*}
where $u = 1$ or $-1$, $\partial f = f'$, $\xi_i(u,v) = \xi (x_i,t_i,y,s;u,v)$ and $I_{u,v}$ is the interval with end 
points $\xi_1(u,v)$ and $\xi_2(u,v)$. We first consider the case $u =1$. Clearly $|\xi_1(1,v)^2 - \xi_2(1,v)^2| \le 2 |\xi_1(1,v)- \xi_2(1,v)|$. 
We claim that it has the upper bound 
\begin{align} \label{eq:zeta1-zeta2V}
 |\xi_1(u,v)- \xi_2(u,v)| \le  c\, \sd_{\VV}\big ((x_1,t_1),(x_2,t_2)\big)
       \big[ \Sigma_1 + \Sigma_2(v_1) + \Sigma_3 (v_2)\big], 
\end{align}
where 
\begin{align*}
  \Sigma_1 \, & = \sd_{\VV}\big((x_i,t_i),(y,s)\big)+\sd_{\VV}\big((x_1,t_1),(x_2,t_2)\big),  \\
% \Sigma_1(-1) \, & = \sd_*\big((x_1,t_1),(y,s)\big)+\sd_* \big((x_2,t_2),(y,s)\big),  \\
  \Sigma_2(v_1) \, & = (1-v_1)\sqrt{s}, \\
   \Sigma_3 (v_2)\, & = (1 - v_2)\sqrt{1-s}. 
\end{align*}
To see this, we follow the notation used in the proof of Lemma \ref{lem:kernelV}, writing
$\xi_i = v_1 \sqrt{\eta_i (1)} + v_2 \sqrt{1-t}\sqrt{1-s}$, where $\eta_i(1) =\eta(x_i,t_i,s_i; 1)$. 
Using \eqref{eq:dV=dB+dt} and writing $x_i = t_i x_i'$ and $y = s y'$, we obtain 
\begin{align}\label{eq:xi1-xi2V}
 \xi_1(1,v)- \xi_2(1,v) \, & = \xi_1(1,\mathbf{1}) - \xi_2(1,\mathbf{1}) \\
% \cos \sd_{\VV}\big((x_1,t_1),(y,s)\big) - \cos \sd_{\VV}\big((x_2,t_2),(y,s)\big) \\
 &  +  (1-v_1) \big(\sqrt{\eta_2(1)} -  \sqrt{\eta_1(1)}\big)  \notag\\
      &   +  (1- v_2) \big(\sqrt{1- t_2} - \sqrt{1-t_1}\big) \sqrt{1-s}, \notag
\end{align}
where $\mathbf{1} = (1,1)$. For $u = 1$, as in Proposition \ref{prop:cos-d}, 
$$
    \eta_i(1) = t_i s  \cos \frac{\sd_{\BB}(x_i', y')}{2} \quad\hbox{and}\quad 
     \xi_i (1,\mathbf{1}) =  \cos \sd_{\VV}\big((x_i,t_i),(y,s)\big).
$$
Substituting these into \eqref{eq:xi1-xi2V}, the resulted identity is similar to the corresponding expression
in the proof of Theorem \ref{thm:L-LkernelV0} and can be analogously estimated using Proposition \ref{prop:cos-d} 
and Lemma \ref{lem:|s-t|V}, so that \eqref{eq:zeta1-zeta2V} follows.

As in the proof of Theorem \ref{thm:L-LkernelV0}, it follows from \eqref{eq:DLn(t,1)} with $m =1$  that 
\begin{align*}
&\big\| \partial L_n^{\l-\f12,-\f12} \big(2(\cdot)^2-1\big)\big\|_{L^\infty(I_{1,v})}  \\
 & \qquad\qquad  \le  c \left[  \frac{n^{2 \l + 3}}{\big(1+n\sqrt{1-\xi_1(1,v)^2} \big)^{\k}} 
     +  \frac{n^{2 \l + 3}}{\big(1+n\sqrt{1-\xi_2(1,v)^2}\big)^{\k}} \right].
\end{align*}
Hence, $K_1$ is bounded by, with $\l = \g+d$, %by the sum over $u =1$ and $u =-1$ of 
\begin{align*}
K_1\le   c\, \sd_{\VV}\big((x_1,t_1),(x_2,t_2)\big) & \int_{[-1,1]^2}\left[  \frac{n^{2 \l + 3}}
    {\big(1+n\sqrt{1-\xi_1(1,v)^2} \big)^{\k}}  +  \frac{n^{2 \l + 3}}{\big(1+n\sqrt{1-\xi_2(1,v)^2}\big)^{\k}} \right ]  \\
      & \times \big(\Sigma_1+\Sigma_2(v_1) + \Sigma_3(v_2) \big) (1-v_1^2)^{\frac{d-1}{2}-1} (1-v_2^2)^{\g-\f12} \d v.
\end{align*}
\iffalse
For the case $\Sigma_1(-1)$, we first observe that, following the proof of \eqref{eq:1-xiV},  
\begin{equation*}
  1- \xi (x_i,t_i,y,s;-1,v)^2 \ge  \frac{2}{\pi} \left[ \sd_*((x,t), (y,s))\right]^2. 
\end{equation*}
Using $\sd_\VV( (x_1,t_1),(x_2,t_2)) \le c^* /n$, we obtain %and \eqref{eq:d*inequality}, we obtain 
$$
  \Sigma_1 \le c \left(\sd_*((x_i,t_i),(y,s)) + n^{-1}\right) \le c n^{-1} \left(1 + n \sqrt{1-\xi_i(1, v)}\right), 
$$
which allows us to replace $\Sigma_1$ by $n^{-1}$ with the price of lowing $\s$ to $\s -1$.
The same argument also applies to $\Sigma(1)$. Furthermore, since $\la X,Y\ra \ge \la X,Y_*\ra$, we 
have $\cos \sd_{\VV} ((x,t), (y,s)) \ge \cos \sd_* ((x,t), (y,s))$, so that $\sd_\VV ((x,t), (y,s)) \le \sd_* ((x,t), (y,s))$.
In particular, this further implies 
$$
 1- \xi (x,t,y,s, -1, v)^2 \ge \frac{2}{\pi} \left[ \sd((x,t), (y,s))\right]^2. 
$$
\fi
Using the assumption that $(x_1,t_1) \in \sc ((x_2,t_2), c^* n^{-1})$, the integral that contains $\Sigma_1$ 
term can be handled as in the proof of Theorem \ref{thm:L-LkernelV0}, which leads to 
%For the case $\Sigma_1$, we use $\Sigma_1 \le   c n^{-1} \left(1 + n \sqrt{1-\xi_i(1, v)}\right)$
\begin{align*}
   \int_{[-1,1]^2}  &  \frac{n^{2 \l + 3}} {\big(1+n\sqrt{1-\xi_i(1,v)^2} \big)^{\k}}
      \Sigma_1  (1-v_1^2)^{\frac{d-1}{2}-1} (1-v_2^2)^{\g-\f12} \d v \\
      &  \le \frac{n^{2 \l + 2}}{\big(1+n\sd_\VV ((x,t), (y,s)) \big)^{\k - 2\g-d-2}}  
    \int_{[-1,1]^2} \frac{ (1-v_1^2)^{\frac{d-1}{2}-1} (1-v_2^2)^{\g-\f12}} {\big(1+n\sqrt{1-\xi_i(1,v)^2} \big)^{2\g+d+2}}
        \d v. 
\end{align*}
Now,  the inequality \eqref{eq:1-xi-lbd} shows that 
$$
1-\xi_i(1,v) \ge \frac14 \sqrt{t_i s}(1- v_1)+\sqrt{1-t_i}\sqrt{1-s}(1-v_2).
$$
Hence, applying Corollary \ref{cor:kernelV0} with $d-1$ replaced by $d$ and Lemma \ref{lem:|s-t|V}, we
conclude that the integral in the right-hand side is bound by 
\begin{align*}
        \frac{c\, n^{d+2}} {\sqrt{ W_{\g,0} (n; x_i, t_i) }\sqrt{ W_{\g,0} (n; y, s)}
               \big(1 + n  \sd_{\VV}( (x_i,t_i), (y,s)) \big)^{\k(\g,d)}},
\end{align*}
where $\k(\g,d) = \k - 3 \g - \frac{3d+4}{2}$, since $W_{\g,0} (n; x, t) = \sw_{\g, d+1}(n;t)$. Putting together,
we see that the integral containing $\Sigma_1$ has the desired upper bound. The two integrals 
containing $\Sigma_2(v_1)$ and $\Sigma_3(v_2)$ can be estimated similarly and more straightforwardly
by using Corollary \ref{cor:kernelV0}. Moreover, the estimate is parallel to the two corresponding terms in
the proof of Theorem \ref{thm:L-LkernelV0} and can be carried out similarly. This completes the proof 
of the case $u =1$. 

For $u =-1$, we observe that if we define $Y_* = (y, -\sqrt{s^2-\|y\|^2})$ and 
$$
  \sd_*((x,t), (y,s)) = \arccos \left ( \sqrt{\frac{\la X, Y_* \ra + t s}{2} } + \sqrt{1-t}\sqrt{1-s} \right),
$$
then for $u =-1$ and $v = \mathbf{1} = (1,1)$, 
$$
  \xi (x,t,y,s;-1,\mathbf{1})  = \cos \sd_*((x,t), (y,s)).  
$$
Since $(Y_*, s) \in \VV_0^{d+2}$ and $\sd_*((x,t), (y,s)) =  \sd_{\VV_0^{d+2}}( (X,t), (Y_*,s))$, it follows 
that $\sd_*$ is also a distance function and satisfies, in particular, the triangle inequality. As a result, we
can repeat the proof for $u=1$ to estimate $K_{-1}$, so that $K_{-1}$ is bounded by the right-hand side 
of \eqref{eq:L-LkernelV} with $(1+n \sd_*((x,t),(y,s)))^\k$ in place of $(1+n \sd_{\VV}((x,t),(y,s)))^\k$ in
the denominator. Since $\la X,Y\ra \ge \la X,Y_*\ra$, $\cos \sd_{\VV} ((x,t), (y,s)) \ge \cos \sd_*((x,t),(y,s))$
and, consequently, $\sd_\VV ((x,t), (y,s)) \le \sd_* ((x,t), (y,s))$. Hence, the upper bound 
of $K_{-1}$ is bounded by the right-hand side of \eqref{eq:L-LkernelV}. 
This completes the proof. 
\end{proof}

\begin{rem}
It is not clear if this theorem also holds for $W_{\g,\mu}$ when $\mu > 0$. The main task lies in establishing 
an analog of \eqref{eq:1-xi-lbd}, perhaps with a multiple of $1-u$ as one more term. This turns out to 
be elusive, not because of lack of trying, and there appears to be a real obstacle for $(x_i,t_i)$ around the 
apex of the cone. 
\end{rem}

Our next lemma is an analog of Lemma \ref{lem:intLn}, which establishes Assertion 3 when $p=1$ 
for the weight function $W_{\g,\mu}$ on the solid cone. 

\begin{lem}\label{lem:intLnV}
Let $d\ge 2$, $\g> -1$ and $\mu > -\f12$. For $0 < p < \infty$, assume 
$\k > \frac{2d+2}{p} + (\g+\mu+\f{d}{2}) |\f1p-\f12|$. Then for $(x,t) \in \VV^{d+1}$,  
\begin{align}\label{eq:intLnV}
\int_{\VV^{d+1}} \frac{ W_{\g,\mu}(s)}{ W_{\g,\mu} (n; s)^{\f{p}2}
    \big(1 + n \sd_{\VV}( (x,t), (y,s)) \big)^{\k p}}   \d y \d s 
    \le c n^{-d-1}\, W_{\g,\mu} (n; t)^{1-\f{p}{2}}.
\end{align}
\end{lem}

\begin{proof}
Let $J_{p}$ denote the left-hand side of \eqref{eq:intLnV}. Using the doubling property of $W_{\b,\g}$ and
by Lemma \ref{lem:CorA3}, it is sufficient to estimate $J_{2,\k}$. Setting $y = s y'$, $y' \in \BB^d$, we obtain
\begin{align*}
J_2  =   \int_0^1 s^d \int_{\BB^d} \frac{ W_{\g,\mu}(s y',s)}{W_{\g,\mu} (n; s y', s) 
      (1 + n \sd_{\VV}( (x,t), (s y',s)) )^{2\k}}  \d y' \d s . 
\end{align*} 
Using $\sw_{\g,d} (n; t)$ defined in \eqref{eq:w(n;t)}, it follows readily that 
$$
   \frac{W_{\g,\mu}(y,s)}{W_{\g,\mu}(n; y,s)} = c \frac{s \sw_{-1,\g}(s)}{(s^2-\|y\|^2+n^{-2})^{\f12} \sw_{\g,d+1}(n;s)} 
       \le c \frac{ \sw_{-1,\g} (n; s)}{\sw_{\g,d+1}(n;s) \sqrt{1-\|y'\|^2}}, 
$$ 
which leads to 
\begin{align*}
J_2  \le c  \int_0^1 s^d \int_{\BB^d} \frac{\sw_{-1,\g} (n; s)}{\sw_{\g,d+1}(n;s) (1 + n \sd_{\VV}( (x,t), (s y',s)) )^{2\k}}
    \frac{\d y'} {\sqrt{1-\|y'\|^2} }\d s. 
\end{align*} 
Setting $x= t x'$, $X = (x', \sqrt{1-\|x'\|^2})$ and $Y = (y', \sqrt{1-\|y'\|^2})$, so that $\sd_{\BB^d}(x',y') 
= \sd_{\SS^d}(X,Y)$, we use the identity 
\begin{equation}\label{eq:BB-SS+}
   \int_{\BB^d} g\left(y', \sqrt{1-\|y'\|^2} \right) \frac{d y'}{\sqrt{1-\|y'\|^2}} =  \int_{\SS_+^d} g(y) \d \s(y),
\end{equation}
where $\SS_+^d$ denotes the upper hemisphere of $\SS^d$, which allows us to follow the proof of 
Proposition \ref{prop:intLn} to obtain
\begin{align*}
  J \, & \le c   \int_0^1 \int_{-1}^1 \frac{s^{d} \sw_{-1,\g}(s)(1-u^2)^{\f{d-2}{2}} } 
      {\sw_{\g,d+1}(n; s) \left(1 + n \arccos \left(\sqrt{t s} \sqrt{\frac{1+u}{2}}+ \sqrt{1-t}\sqrt{1-s}\right) \right)^{2 \k}} \d u\d s.
\end{align*}
Comparing with the proof of Proposition \ref{prop:intLn}, we see that the above integral with $d$ replaced by $d-1$ 
has already appeared in that proof, so that it is bounded by $c n^{-d-1}$ as the proof of Proposition \ref{prop:intLn}
shows. The proof is completed. 
\end{proof}

\begin{prop}\label{prop:intLnV}
Let $d\ge 2$, $\mu \ge 0$ and $\g \ge -\f12$. For $0<p<\infty$ and $(x,t) \in \VV^{d+1}$,  
\begin{equation*}
   \int_{\VV^{d+1}} \left| \Lb_n\big(W_{\g,\mu};(x,t),(y,s)\big) \right| ^p W_{\g,\mu}(y,s) \d y \d s 
     \le c \left(\frac{n^{d+1}}{W_{\g,\mu}(n;x,t)}\right)^{p-1}.
\end{equation*}
\end{prop}

This follows immediately from applying Lemma \ref{lem:intLnV} on (i) of Theorem \ref{thm:kernelV}.

\begin{cor}
For $d\ge 2$ and $\g \ge -\f12$, the space $(\VV^{d+1}, W_{\g,0}, \sd_{\VV})$ is a localizable 
homogeneous space. 
\end{cor}

\subsection{Maximal $\ve$-separated sets and MZ inequality} \label{sec:ptsV}

We provide a construction of maximal $\ve$-separated sets on the solid cone as defined in 
Definition \ref{defn:separated-pts}.  
 
Our construction follows the approach in Subsection \ref{sec:ptsV0}. Instead of starting with
$\ve$-separated sets on the unit sphere, we now need such sets on the unit ball $\BB^d$. 
We adopt the following notation. For $\ve > 0$, we denote by $\Xi_{\BB}(\ve)$ a maximal 
$\ve$-separated set on the unit ball $\BB^d$ and we let $\BB_u(\ve)$ be the subsets in $\BB^d$ 
so that the collection $\{\BB_u(\ve): u \in \Xi_\BB(\ve)\}$ is a partition of $\BB^d$, and we assume
\begin{equation}\label{eq:ptsB1}
  \cb_{\BB}(u, c_1 \ve) \subset \BB_u(\ve) \subset \cb_{\BB}(u, c_2 \ve), \qquad u \in \Xi_{\BB}(\ve),
\end{equation} 
where $\cb_{\BB}(u,\ve)$ denotes the ball centered at $u$ with radius $\ve$ in $\BB^d$, $c_1$ and 
$c_2$ depend only on $d$. It is known (see, for example, \cite{PX2}) that such a $\Xi_\BB(\ve)$ 
exists for all $\ve > 0$ and its cardinality satisfies  
\begin{equation}\label{eq:ptsB2}
c_d' \ve^{-d} \le \# \Xi_{\BB}(\ve) \le c_d \ve^{-d}.
\end{equation} 
 
On the solid cone $\VV^{d+1}$ we denote  by $\Xi_{\VV} = \Xi_{\VV}(\ve)$ a maximum 
$\ve$-separated set and denote by $\{\VV(u,t): (tu,t) \in \Xi_{\VV}\}$ a partition of $\VV^{d+1}$. 
We give a construction of these sets. 

Let $\ve > 0$ and let $N = \lfloor \frac{\pi}{2}\ve^{-1} \rfloor$. We define $t_j = \sin^2 \frac{\t_j}2$, 
$t_j^+$ and $t_j^-$, $1 \le j \le N$, as in Subsection \ref{sec:ptsV0}, so that 
$$
   \VV^{d+1} =  \bigcup_{j=1}^N \VV^{(j)}, \quad \hbox{where}\quad \VV^{(j)}:= 
        \left\{(x,t) \in \VV^{d+1}:   t_j^- < t \le t_j^+ \right \}.  
$$
Let $\ve_j := (2 \sqrt{t_j})^{-1} \pi \ve$. Then $\Xi_\BB(\ve_j)$ is the maximal 
$\ve_j$-separated set of $\BB^d$ such that, for each $j \ge 1$, $\{\BB_u(\ve_j):  u \in \Xi_\BB(\ve_j)\}$
is a partition of $\BB^d$ and 
$
   \# \Xi_\BB(\ve_j) \sim \ve_j^{-d}.
$
For each $j =1,\ldots, N$, we decompose $\VV^{(j)}$ by 
$$
 \VV^{(j)} =  \bigcup_{u \in  \Xi_\BB(\ve_j)} \VV(u,t_j), \quad \hbox{where}\quad 
 \VV(u,t_j):= \left\{(t v,t):  t_j^- < t \le t_j^+, \, v \in \BB_u(\ve_j) \right\}.
$$
Finally, we define the set $\Xi_{\VV}$ of $\VV^{d+1}$ by
$$
   \Xi_{\VV} = \big\{(t_j u, t_j): \,  u \in \Xi_\BB(\ve_j), \, 1\le j \le N \big\}. 
$$

\begin{prop} \label{prop:subsetV}
Let $\ve > 0$ and $N = \lfloor \frac{\pi}{2} \ve^{-1} \rfloor$. Then $\Xi_{\VV}$ is a maximal $\ve$-separated 
set of $\VV^{d+1}$ and $\{\VV(t_j u, t_j): u \in \Xi_\BB(\ve_j), \, 1\le j \le N \}$ is a partition 
$$
   \VV^{d+1} =  \bigcup_{j=1}^N \bigcup_{u \in \Xi_\BB(\ve_j)} \VV(u,t_j).
$$
Moreover, there are positive constants $c_1$ and $c_2$ depending only on $d$ such that 
\begin{equation}\label{eq:incluVcap}
      \cb \big((t_j u,t_j), c_1 \ve\big) \subset \VV(u,t_j) \subset \cb \big( (t_j u,t_j), c_2 \ve\big), \qquad (t_j u, t_j) \in
       \Xi_\VV, 
\end{equation}
and $c_d'$ and $c_d$  depending only on $d$ such that 
$$
c_d' \ve^{-d-1} \le \# \Xi_{\VV} \le c_d \ve^{-d-1}. 
$$
\end{prop}

\begin{proof}
The proof is parallel to that of Proposition \ref{prop:subsetV0}. In fact, it can be carried out almost verbatim
under obvious modifications such as replacing $\sd_\SS$ by $\sd_\BB$ and $\VV_0$ by $\VV$.
We shall omit it. 
\end{proof}

Even though we only established Assertion 2 for $W_{\g,0}$ for $\g \ge 0$, it is sufficient for 
applying Theorem \ref{thm:MZinequalityV} to conclude that the Marcinkiewicz-Zygmund inequality
on a maximal $\ve$-separated set holds for any doubling weight on $\VV^{d+1}$. 

\begin{thm} \label{thm:MZinequalityV}
Let $W$ be a doubling weight on $\VV^{d+1}$. Let  $\Xi_{\VV}$ be a maximal $\f \delta n$-separated 
subset of $\VV^{d+1}$ and $0 < \delta \le 1$.
\begin{enumerate}[$(i)$]
\item For all $0<p< \infty$ and $f\in\Pi_m^{d+1}$ with $n \le m \le c n$,
\begin{equation*}
  \sum_{z \in \Xi_{\VV}} \Big( \max_{(x,t)\in \cb((z,r), \f \delta n)} |f(x,t)|^p \Big)
     W\!\left(\cb((z, r), \tfrac \delta n) \right) \leq c_{W,p} \|f\|_{p,W}^p
\end{equation*}
where $c_{W,p}$ depends on $p$ when $p$ is close to $0$ and on the doubling constant $\a(W)$.
\item For $0 < r < 1$, there is a $\delta_r > 0$ such that for $\delta \le \delta_r$, $r \le p < \infty$ and 
$f \in \Pi_n^{d+1}$,  
\begin{align*}
  \|f\|_{p,W}^p \le c_{W,r} \sum_{z \in\Xi}
       \Big(\min_{(x,t)\in \cb\bigl((z,r), \tfrac{\delta}n\bigr)} |f(x,t)|^p\Big)
          W\bigl(\cb((z,r), \tfrac \delta n)\bigr)
\end{align*}
where $c_{W,r}$ depends only on $L(W)$ and on $r$ when $r$ is close to $0$.
\end{enumerate}
\end{thm}

\subsection{Cubature rules and localized tight frames}\label{sec:CFframeV}
We first verify Assertion 4 in Subsection \ref{subsect:CFunction} by constructing fast decaying polynomials 
on the solid cone. 

\begin{lem}\label{lem:A4V}
Let $d\ge 2$. For each $(x,t) \in \VV^{d+1}$, there is a polynomial $\CT_{x,t}$ of degree $n$ that
satisfies
\begin{enumerate} [   (1)]
\item $\CT_{x,t}(x,t) =1$, $\CT_{x,t}(y,s) \ge c > 0$ if $(y,s) \in \cb( (x,t), \f{\delta}{n})$, and for every $\k > 0$,
$$
   0 \le \CT_{x,t}(y,s) \le c_\k \left(1+ \sd_{\VV}\big((x,t),(y,s)\big) \right)^{-\k}, \quad (y,s) \in \VV^{d+1}.
$$
\item there is a polynomial $q$ of degree $n$ such that $q(x,t) \CT_{x,t}$ is a polynomial of degree $3 n$ 
in $(x,t)$ and $1 \le q(x,t) \le c$. 
\end{enumerate}
\end{lem}

\begin{proof}
For $(x,t), (y,s) \in \VV^{d+1}$, we introduce the notation $X = (x, \sqrt{t^2-\|x\|^2})$ and $Y = (y, \sqrt{s^2-\|y\|^2})$.  
Moreover, we also denote $Y_* =  (y, - \sqrt{t^2-\|y\|^2})$. Then $(X,t), (Y,s)$ and $(Y_*, s)$ are all elements of
$\VV_0^{d+2}$. Let $T_{(X,t)}$ denote the polynomial of degree $n$ on $\VV_0^{d+2}$ defined in 
Lemma \ref{lem:A4}. We
now define
$$
   \CT_{x,t}(y,s): = \frac{T_{(X,t)} (Y,s)+ T_{(X,t)} (Y_* ,s)} {1 + T_{(X,t)}(X_*,t)}. 
$$
Since $T_{(X,t)}(X,t) =1$, it follows that $\CT_{(x,t)}(x,t) = 1$. Since $\sqrt{\frac{\la X_*, X\ra+t^2}{2}} = \|x\|$, we have
$$
T_{(X,t)}(X_*,t) = \frac{S_n(\|x\|+ (1-t))+ S_n(\|x\| -(1-t))}{1 + S_n(2 t-1)}. 
$$
By $0 \le S_n(t) \le c$, we see that $0 \le T_{(X,t)}(X_*,s) \le c$. In particular, it follows that
$$
    \CT_{x,t}(y,s) \ge c \,T_{(X,t)} (Y,s) \ge c > 0, \qquad (y,s) \in \cb\big((x,t), \tfrac \delta n\big), 
$$
since $\sd_{\VV}((x,t),(y,s)) = \sd_{\VV_0^{d+2}}((X,t),(Y,s))$. Furthermore, since 
$\cos \sd_{\VV_0} ((X,t), (Y,s)) \ge \cos \sd_{\VV_0} ((X,t), (Y_*,s))$, we obtain 
$$
\sd_{\VV_0} ((X,t), (Y_*,s)) \ge  \sd_{\VV_0} ((X,t), (Y,s)) = \sd_{\VV}((x,t),(y,s)). 
$$
Hence, using the estimate of $T_{(X,t)}$ in Lemma \ref{lem:A4}, we conclude that
\begin{align*}
 0 \le  \CT_{(x,t)}(y,s) \,& \le c\left[\big(1+n \sd_{\VV_0}((X,t),(Y,s) ) \big)^{-\k} 
      + \big(1+n \sd_{\VV_0}((X,t),(Y_*,s)) \big)^{-\k}\right] \\
      \,& \le c \big(1+n \sd_{\VV}((x,t),(y,s)) \big)^{-\k}. 
\end{align*} 
Finally, let $q(x,t) = (1+S_n(2t-1)) T_{(X,t)}(X_*,t)$. Then $q(x,t)$ is a polynomial of degree at most $2n$, so
that $q(x,t) \CT_{x,t}$ is a polynomial of degree at most $3 n$ and $1 \le q(x,t) \le c$. This completes the
proof.
\end{proof}

By Propositions \ref{prop:ChristF1} and \ref{prop:ChristF2} and using Theorem \ref{thm:kernelV}, we have 
established the following result on the Christoffel function $\l_n(W)$. 

\begin{cor} \label{cor:ChristFV}
Let $W$ be a doubling weight function on $\VV^{d+1}$. Then 
\begin{equation*}  %\label{eq:ChirstoffelV0}
   \lambda_n \big(W; (x,t) \big)  \le c \, W\!\left(\cb\left((x,t), \tfrac1n \right) \right).
\end{equation*}
Moreover, for $\g \ge -\f12$ and $\mu \ge 0$, 
$$
\lambda_n \big(W_{\g,\mu}; (x,t) \big)  \ge c \,W_{\g,\mu} \left(\cb\left((x,t), \tfrac1n \right) \right) 
   = c n^{-d-1} W_{\g,\mu}(n; t).
$$
\end{cor} 
 
Since the lower bound of the Christoffel function holds for all doubling weight, we see that the cubature
rules in Theorem \ref{thm:cubature} holds for all doubling weights on the solid cone. 

\begin{thm}\label{thm:cubatureV}
Let $W$ be a doubling weight on $\VV^{d+1}$. Let $\Xi$ be a maximum $\frac{\delta}{n}$-separated 
subset of $\VV^{d+1}$. There is a $\delta_0 > 0$ such that for $0 < \delta < \delta_0$ there exist positive numbers 
$\l_{z,r}$, $(z,r) \in \Xi$, so that 
\begin{equation}\label{eq:CFV}
    \int_{\VV^{d+1}} f(x) W(x,t)  \d x \d t = \sum_{(z,r) \in \Xi }\l_{z,r} f(z,r), \qquad 
            \forall f \in \Pi_n^{d+1}. 
\end{equation}
Moreover, $\l_{z,r} \sim W\!\left(\cb((z,r), \tfrac{\delta}{n})\right)$ for all $(z,r) \in \Xi$. 
\end{thm}

We can now state our local frame on the solid cone. For $j =0,1,\ldots,$ let $\Xi_j$ be a maximal 
$ \frac \delta {2^{j}}$-separated subset in $\VV^{d+1}$, so that 
\begin{equation*}
  \int_{\VV^{d+1}} f(x,t) W(x,t)  \d x \d t = \sum_{(z,r) \in \Xi_j} \l_{(z,r),j} f(z,r), 
     \qquad f \in \Pi_{2^j}^{d+1}.  
\end{equation*}
Denote by $\Lb_n(W)*f$ the near best approximation operator defined by 
$$
\Lb_n(W)*f (x,t) =  \int_{\VV^{d+1}}f(y,s) \Lb_n(W; (x,t),(y,s)) W(y,s)  \d y \d s.
$$
For $j= 0,1,\ldots,$ define the operator $F_j(W)$ by
$$
    F_j(W) * f = \Lb_{2^{j-1}}(W) * f 
$$
and define the frame elements $\psi_{(z,r),j}$ for $(z,r) \in \Xi_j$ by 
$$ 
       \psi_{(z,r),j}(x,t):= \sqrt{\l_{(z,r),j}} F_j((x,t), (z,r)), \qquad (x,t) \in \VV^{d+1}. 
$$ 
Then $\Phi =\{ \psi_{(z,r),j}: (z,r) \in \Xi_j, \, j =1,2,3,\ldots\}$ is a tight frame. 
 
\begin{thm}\label{thm:frameV}
Let $W$ be a doubling weight on $\VV^{d+1}$. If $f\in L^2(\VV^{d+1}, W)$, then
$$ 
   f =\sum_{j=0}^\infty \sum_{(z,r) \in\Xi_j}
            \langle f, \psi_{(z,r), j} \rangle_W \psi_{(z,r),j}  \qquad\mbox{in $L^2(\VV^{d+1}, W)$}
$$  
and
$$ 
\|f\|_{2,W}  = \Big(\sum_{j=0}^\infty \sum_{(z,r) \in \Xi_j} \left|\langle f, \psi_{(z,r),j} \rangle_W\right|^2\Big)^{1/2}.
$$ 
Furthermore, for $W_{\g,\mu}$ with $\g \ge -\f12$ and $\mu \ge 0$, the frame is highly localized in the sense
that, for every $\k >0$, there exists a constant $c_\k >0$ such that 
\begin{equation} \label{eq:needleV}
   |\psi_{(z,r),j}(x,t)| \le c_\s \frac{2^{j(d+1)/2}}{\sqrt{ W_{\g,\mu}(2^{j}; x,t)} (1+ 2^j \sd_{\VV}((x,t),(z,r)))^\k}, 
     \quad (x,t)\in \VV^{d+1}.
\end{equation}
\end{thm}

The frame elements are well defined for all doubling weight by Theorem \ref{thm:cubatureV}. 
The decomposition is the consequence of Theorem \ref{thm:frame}. Moreover, the localization 
\eqref{eq:needleV} follows from Theorem \ref{thm:kernelV} and $\l_{(z,r),j} \sim 
2^{- j(d+1)} W_{\g,\mu}(2^j;t)$ that holds for $W_{\g,\mu}$ as see from 
Corollary \ref{cor:ChristFV} and \eqref{eq:capV}.

\subsection{Characterization of best approximation} 

For $f\in L^p(\VV^{d+1}, W)$, we denote by $\Eb_n(f)_{p, W}$ the error of best approximation to $f$ 
from $\Pi_n^{d+1}$, the space of polynomials of degree at most $n$, in the norm $\|\cdot\|_{p, W}$, 
$$
    \Eb_n(f)_{p, W}:= \inf_{g \in \Pi_n^{d+1}} \|f - g\|_{p,  W}, \qquad 1 \le p \le \infty.
$$
We give a characterization of this quantity in terms of the modulus of smoothness defined via the operator 
$\sS_{\t,W}$ and the $K$-functional defined via the differential operator $\fD_{\g,\mu}$ for $W_{\g,\mu}$. 

For $f\in L^p(\VV^{d+1}, W_{\g,\mu})$ and $r > 0$, the modulus of smoothness is defined by
$$
  \o_r(f; \rho)_{p,W_{\g,\mu}} = \sup_{0 \le \t \le \rho} 
       \left\| \left(I - \Sb_{\t,W_{\g,\mu}}\right)^{r/2} f\right\|_{p,W_{\g,\mu}}, \quad 1 \le p \le \infty, 
$$
where the operator $\Sb_{\t,W_{\g,\mu}}$ is defined by, for $n = 0,1,2,\ldots$ and $\l = 2\mu+\g+d$, 
$$
 \proj_n(W_{\g,\mu}; \Sb_{\t,W_{\g,\mu}}f) = R_n^{(\l-\f12, -\f12)} (\cos \t) \proj_n(W_{\g,\mu}; f).
$$
Moreover, in terms of the fractional differential operator $(-\fD_{\g,\mu})^{\f r 2}$, the $K$-functional
is defined for a weight $W$ on $\VV^{d+1}$ by 
$$
   \Kb_r(f,\rho)_{p,W} : = \inf_{g \in \CW_p^r(\VV^{d+1}, W)}
      \left \{ \|f-g\|_{p,W} + \rho^r\left\|(-\fD_{\g,\mu})^{\f r 2}f \right\|_{p,W} \right \},
$$
where $\CW_p^r(\VV^{d+1},W)$ denotes the Sobolev space consisting of functions in $L^p(\VV^{d+1}, W)$
with finite $\left\|(-\fD_{\g,\b})^{\f r 2}f \right\|_{p,W}$. 

The weight function $W_{\g,\mu}$ admits Assertion 1 and 3 by Theorem \ref{thm:kernelV} and 
Lemma \ref{lem:intLnV}. We now verify that the Assertion 5 in Subsection \ref{set:Bernstein} holds. 
By Theorem \ref{thm:Delta0V0}, the kernel $L_n^{(r)}(\varpi)$ in Assertion 5 becomes
$$
    \Lb_n^{(r)}\big(W_{\g,\mu}; (x,t),(y,s)\big)=\sum_{k=0}^\infty \wh a\left(\frac{k}{n} \right) (k(k+2\mu+\g+d))^{\f r 2} 
       \Pb_k\big(W_{\g,\mu}; (x,t),(y,s)\big).
$$

\begin{lem}
Let $\g \ge -\f12$ and $\mu \ge 0$. Let $\k > 0$. Then, for $r > 0$ and $(x,t), (y,s) \in \VV^{d+1}$, 
$$
 |  \Lb_n^{(r)}\big(W_{\g,\mu}; (x,t),(y,s)\big)| \le c_\k  \frac{n^{r+d}}{\sqrt{ W_{\g,\mu} (n;x,t) }\sqrt{ W_{\g,\mu} (n; y,s)}
\left(1 + n \sd_{\VV}( (x,t), (y,s)) \right)^{\k}}.
$$
\end{lem}
\begin{proof}
By \eqref{eq:PbCone2}, the kernel can be written as 
\begin{align*}
 \Lb_n^{(r)}\big(W_{\g,\mu};(x,t),(y,s)) = \, & c_{\mu,\g,d} \int_{[-1,1]^3}  L_{n,r}\left(2 \xi(x,t,y,s;u,v)^2-1\right) \\
    & \times (1-u^2)^{\mu-1}(1-v_1^2)^{\a -1} (1-v_2^2)^{\g-\f12} \d u \d v,
\end{align*}
in which $\a = \mu + \f{d-1}{2}$ and $L_{n,r}$ is defined by, with $\l = 2 \a + \g+1$, 
$$
 L_{n,r}(t) = \sum_{k=0}^\infty \wh a\left(\frac{k}{n} \right) (k(k+\g+d-1))^{\f r 2} 
     \frac{P_n^{(\l-\f12, -\f12)}(1)P_n^{(\l-\f12, -\f12)}(t)}{h^{(\l-\f12,-\f12)}}.
$$
Applying \eqref{eq:DLn(t,1)} with $\eta(t) = \wh a(t) \left( t( t + n^{-1} (2\mu+\g+d))\right)^{\f r 2}$ and $m=0$, 
it follows that 
$$
 \left| L_{n,r}(t) \right| \le c n^{r} \frac{n^{2\l+1}}{(1+n\sqrt{1-t})^\ell}.  
$$
Using this estimate, we can then deduce the proof to one that has already appeared for $\Lb_n(W_{\g,\mu})$
in the proof of Theorem \ref{thm:kernelV0}. 
\end{proof}
 
With Assertions 1, 3 and 5 verified for $W_{\g,\mu}$, the characterization of the best approximation by 
polynomials in Subsection \ref{sec:bestapp} holds on the solid cone, which we state below. 

\begin{thm}
Let $f \in L^p(\VV^{d+1}, W)$ if $1 \le p < \infty$ and $f\in C(\VV^{d+1})$ if $p = \infty$. 
Le $r > 0$ and $n =1,2,\ldots$. For $W = W_{\g,\mu}$ with $\g \ge -\f12$ and $\mu \ge 0$, there holds 
\begin{enumerate} [   (i)]
\item direct estimate
$$
  \Eb_n(f)_{p,W_{\g,\mu}} \le c \, \Kb_r (f;n^{-1})_{p,W_{\g,\mu}}.
$$
\item inverse estimate, for $\mu = 0$,  
$$
   \Kb_r(f;n^{-1})_{p,W_{\g,0}} \le c n^{-r} \sum_{k=0}^n (k+1)^{r-1}\Eb_k(f)_{p, W_{\g,0}}.
$$
\end{enumerate}
\end{thm}
 
\begin{proof}
The direct estimate follows from Theorem \ref{thm:nearbest}, which requires only Assertions 1, 3, 5 and holds
for $W_{\g,\mu}$ by Theorem \ref{thm:kernelV} and Lemma \ref{lem:intLnV}. The inverse estimate follows 
from Theorem \ref{thm:Enf-Kfunctional}, which requires one weigh function that admits all Assertions 1--3 and 5
and that holds for $W_{\g,0}$ on the cone.
\end{proof}

Both the direct and the inverse estimates hold for the weight function $W_{\g,\mu}$ in the above theorem. 
However, it should be noted that the inverse estimate uses the $K$-functional 
$$
   \Kb_r(f,\rho)_{p,W} = \inf_{g \in \CW_p^r(\VV^{d+1}, W)}
      \left \{ \|f-g\|_{p,W} + \rho^r\left\|(-\fD_{\g,0})^{\f r 2}f \right\|_{p,W} \right \}
$$
defined via the operator $\fD_{\g,0}$ for the weight $W_{\g,0}$. 

For $W = W_{\g,\mu}$, both direct and inverse estimates can be given via the modulus of 
smoothness, since it is equivalent to the $K$-functional. 

\begin{thm} \label{thm:K=omegaV}
Let  $\g \ge -\f12$,  $\mu\ge 0$ and $f \in L_p^r(\VV^{d+1}, W_{\g,\mu})$,  $1 \le p \le \infty$. Then
for $0 < \t \le \pi/2$ and $r >0$ 
$$
   c_1 \Kb_r(f; \t)_{p,W_{\g,\mu}} \le \o_r(f;\t)_{p,W_{\g,\mu}} \Kb_r(f;\t)_{p,W_{\g,\mu}}.
$$
\end{thm}


\begin{thebibliography}{99}
\bibitem{Ask}
        R. Askey, 
        {\it Orthogonal Polynomials and Special Functions},
        Regional Conference Series in Applied Mathematics {\bf 21}, SIAM, Philadelphia, 1975.    

\bibitem{BKMP1}
       P. Baldi, G. Kerkyacharian, D. Marinucci, and D. Picard. 
       Asymptotics for spherical needlets, 
       \textit{Ann. Statist.} \textbf{37} (2009), 1150--1171.

\bibitem{BKMP2}
       P. Baldi, G. Kerkyacharian, D. Marinucci, and D. Picard. 
       Adaptive density estimation for directional data using needlets, 
       \textit{Ann. Statist.} \textbf{37} (2009), 3362--3395.

\bibitem{BBP}
       H. Berens, P. L. Butzer, and S. Pawelke,
       Limitierungsverfahren von Reihen mehrdimensionaler Kugelfunktionen
       und deren Saturationsverhalten,
       \textit{Publ. Res. Inst. Math. Sci. Ser. A.} \textbf{4} (1968),  201--268.
       
\bibitem{BD}
       G. Brown, F. Dai,
       Approximation of smooth functions on compact two-point homogeneous spaces,
       \textit{J. Funct. Anal.}  \textbf{220} (2005), 401--423.

\bibitem{CD}
        A. Cohen and M. Dolbeault,
        Optimal sampling and Christoffel functions on general domains. arXiv:2010.11040

\bibitem{Dai1}
        F. Dai, 
         Multivariate polynomial inequalities with respect to doubling weights and $A_\infty$ weights, 
         \textit{J. Funct. Anal.} \textbf{235} (2006), 137--170. 
         
%\bibitem{DP1}
%       F. Dai and A. Prymak,    
%        Polynomial approximation on $C^2$-domains, 2019.
%        arXiv:1910.11719

\bibitem{DP2}
        F. Dai and A. Prymak,    
        $L_p$-Bernstein inequalities on $C^2$ domains, 2020.
        arXiv:2010.06728

\bibitem{DP3}
        F. Dai and A. Prymak,  
        On directional Whitney inequality, 2020.
        arXiv:2010.08374  

\bibitem{DaiWang}
        F. Dai and H. Wang, 
        Optimal cubature formulas in weighted Besov spaces with $A_\infty$ weights on multivariate domains,
        \textit{Const. Approx.} \textbf{37} (2013), 167--194.

\bibitem{DaiX2}
        F. Dai and Y. Xu,
        Moduli of smoothness and approximation on the unit sphere and the unit ball. 
        \textit{Adv. Math.} \textbf{224} (2010), 1233--1310.

\bibitem{DaiX}
        F. Dai and Y. Xu,
        \textit{Approximation theory and harmonic analysis on spheres and balls}.
        Springer Monographs in Mathematics, Springer, 2013. 

%F. Dai, A. Primak, V. N. Temlyakov, and S. Yu. Tikhonov, Integral norm discretization and related problems, Uspekhi Mat. Nauk 74 (2019), no. 4(448), 3Ð58 (Russian, with Russian summary).
         
\bibitem{Dunkl} 
        C. F. Dunkl,  
        Differential-difference operators associated to reflection groups. 
        \textit{Trans. Amer. Math. Soc.} \textbf{311} (1989), 167--183.
        
\bibitem{DX} 
        C. F. Dunkl and Y. Xu,
        \textit{Orthogonal Polynomials of Several Variables}.
        Encyclopedia of Mathematics and its Applications \textbf{155},
         Cambridge University Press, Cambridge, 2014.

\bibitem{Dzub}
        J. Dziuba\'{n}ski,
        Triebel-Lizorkin spaces associated with Laguerre and Hermite expansions,
        \textit{Proc. Amer. Math. Soc.} \textbf{125} (1997), 3547--3554.

\bibitem{DzH}
        J. Dziuba\'{n}ski, E. Hern\'{a}ndedez,
        Band-limited wavelets with subexponential decay,
        \textit{Canad. Math. Bull.} \textbf{41} (1998), 398--403.

\bibitem{DL}
        R. A. DeVore and G. G. Lorentz, 
        {\it Constructive approximation}, Grundlehren der Mathematischen Wissenschaften, vol. 303, 
        Springer-Verlag, Berlin, 1993.

\bibitem{Epp2}
        J. Epperson,
        Hermite and Laguerre wave packet expansions,
        \textit{Studia Math.} \textbf{126} (1997), 199--217.

\bibitem{FK}
        J.  Faraut and A. Kor\'anyi
        \textit{Analysis on symmetric cones}.  
        Oxford Mathematical Monographs.  Oxford University Press, New York, 1994.

\bibitem{KPPX}
        G. Kerkyacharian, P. Petrushev, D. Picard and Y. Xu,
        Decomposition of Triebel-Lizorkin and Besov spaces in the context of Laguerre expansions,
         \textit{J. Funct. Anal.} \textbf{256} (2009), 1137--1188.

\bibitem{KS}
        H. L. Krall and I. M. Sheffer, 
        Orthogonal polynomials in two variables,
        \textit{Annali di Matema. Pura \& Appl.} \textbf{76} (1967), 325--376.

\bibitem{Kroo}
         A. Kr\'{o}o, 
        Christoffel functions on convex and starlike domains in $\RR^d$, 
        \textit{J. Math. Anal. Appl.} \textbf{421} (2015), 718--729.

\bibitem{KrooLub}
         A. Kr\'{o}o and D. S. Lubinsky,
         Christoffel functions and universality in the bulk for multivariate orthogonal polynomials,
         Canadian J. Math. \textbf{65}, 600--620.

\bibitem{KPX1}
        G. Kyriazis, P. Petrushev and Y. Xu,
        Jacobi decomposition of weighted Triebel-Lizorkin and Besov spaces,
         \textit{Studia Math.} \textbf{186} (2008), 161--202.

\bibitem{KPX2}
        G. Kyriazis, P. Petrushev and Y. Xu,
        Decomposition of weighted Triebel-Lizorkin and Besov spaces on the ball,
         \textit{Proc. London Math. Soc.}  \textbf{97} (2008), 477--513.

\bibitem{KP1}
        K. Ivanov and P. Petrushev, 
        Fast memory efficient evaluation of spherical polynomials at scattered points, 
         \textit{Adv. Comput. Math.} \textbf{41} (2015), 191--230. 

\bibitem{KP2}
        K. Ivanov and P. Petrushev,
        Highly effective stable evaluation of bandlimited functions on the sphere, 
        \textit{Numer. Algorithms}, \textbf{71} (2016), 585--611.
        
\bibitem{IPX}
        K. Ivanov, P. Petrushev and Y. Xu, 
        Sub-exponentially localized kernels and frames induced by orthogonal expansions.
         \textit{Math. Z.} \textbf{264} (2010), 361--397.

\bibitem{IPX2}
        K. Ivanov, P. Petrushev and Y. Xu,         
        Decomposition of spaces of distributions induced by tensor product bases. 
         \textit{J. Funct. Anal.} \textbf{263} (2012), 1147--1197.
         
\bibitem{KT}
        K. Ivanov and V. Totik, 
        Fast decreasing polynomials, 
        \textit{Constr. Approx.} \textbf{6} (1990), 1--20.

\bibitem{LSWW}
        Q. T. Le Gia, I. H. Sloan, Y. G. Wang and R. S. Womersley, 
        Needlet approximation for isotropic random fields on the sphere.
        \textit{J. Approx. Theory} \textbf{216} (2017), 86--116.
        
\bibitem{MT}
        G. Mastroianni and V. Totik
        Weighted polynomial inequalities with doubling and $A_\infty$ weights.
        \textit{Const. Approx.}  \textbf{16} (2000) 37--71.
          
\bibitem{MNW}
        H. N. Mhaskar, F. J. Narcowich and J. D. Ward,
        Spherical Marcinkiewicz-Zygmund inequalities and positive quadrature,
        {\it  Math. Comp.} {\bf 70} (2001), 1113 - 1130
        (Corrigendum: Math. Comp. {\bf 71} (2001), 453 - 454).

\bibitem{NPW1}
        F. J. Narcowich, P. Petrushev and J. D. Ward,
        Localized tight frames on spheres,
         \textit{SIAM J. Math. Anal.} \textbf{38} (2006), 574--594.

\bibitem{NPW2}
        F. J. Narcowich, P. Petrushev and J. D. Ward,
        Decomposition of Besov and Triebel-Lizorkin spaces on the sphere,
         \textit{J. Funct. Anal.}  \textbf{238} (2006), 530--564.

\bibitem{P}
         S. Pawelke,
         \"Uber Approximationsordnung bei Kugelfunktionen und algebraischen
         Polynomen,
         \textit{T\^ohoku Math. J.} \textbf{24} (1972), 473--486.

\bibitem{PX1}
        P. Petrushev and Y. Xu,
        Localized polynomial frames on the interval with Jacobi weights,
         \textit{J. Fourier Anal. and Appl.} \textbf{11} (2005), 557--575.

\bibitem{PX2}
        P. Petrushev and Y. Xu,
        Localized polynomial frames on the ball,
         \textit{Constr. Approx.}  \textbf{27} (2008), 121--148.

\bibitem{PX3}
        P. Petrushev and Y. Xu,
        Decomposition of spaces of distributions induced by Hermite expansion,
         \textit{J. Fourier Anal. and Appl.}  \textbf{14} (2008), 372-414.

\bibitem{Pr1}
        A. Prymak, 
        Upper estimates of Christoffel function on convex domains,
         \textit{J. Math. Anal, Appl.}  \textbf{455} (2017), 1984--2000.

\bibitem{Pr2}
        A. Prymak, 
        Christoffel functions on planar domains with piecewise smooth boundary, 
        \textit{Acta Math. Hungar.} \textbf{158} (2019),  216--234.

\bibitem{Rus}
        Kh. P. Rustamov,
        On the approximation of functions on a sphere, (Russian),
         {\it Izv. Ross. Akad. Nauk Ser. Mat.} {\bf  57} (1993), 127--148;
         translation in Russian Acad. Sci. Izv. Math. {\bf 43} (1994), no. 2, 311-- 329.\
         
\bibitem{ST}
       E. Saff and V. Totik,
       \textit{Logarithmic Potentials with External Fields},
       Springer-Verlag, Berlin-Heidelberg, 1997.
       
\bibitem{Stein}
        E. Stein, 
        {\em Harmonic Analysis: Real-Variable Methods, Orthogonality, and Oscillatory Integrals}
        Princeton University Press, Princeton, 1993.
  
\bibitem{SW}
        E. Stein and G. Weiss, 
        {\em Introduction to Fourier Analysis on Euclidean Spaces}. Princeton Univ. Press, Princeton, 1971.

\bibitem{Sz}
       G. Szeg\H{o},
       \textit{Orthogonal polynomials}. 4th edition,
       Amer. Math. Soc., Providence, RI. 1975

\bibitem{T1}
       V. Totik, 
       Polynomial approximation on polytopes, 
       {\it Mem. Amer. Math. Soc.} \textbf{232} (2014), no. 1091, vi+112.
 
\bibitem{T2}
       V. Totik,  
       Polynomial approximation in several variables, 
       \textit{J. Approx. Theory} \textbf{252}: 105364 (2020).

\bibitem{WLSW}
       Y. G. Wang, Q. T. Le Gia, I. H. Sloan and R. S. Womersley, 
       Fully discrete needlet approximation on the sphere.
       \textit{Appl. Comput. Harmon. Anal.} \textbf{43} (2017), 292--316.

\bibitem{X95}
       Y. Xu,
       Christoffel functions and Fourier series for multivariate orthogonal polynomials. 
       \textit{J. Approx. Theory} \textbf{82} (1995), 205--239. 

\bibitem{X99}
       Y. Xu,
       Summability of Fourier orthogonal series for Jacobi weight on a ball in $\RR^d$, 
       \textit{Trans. Amer. Math. Soc.}, \textbf{351} (1999), 2439--2458.
       
\bibitem{X05} 
       Y. Xu,
       Weighted approximation of functions on the unit sphere.
       \textit{Const. Approx.} \textbf{21} (2005), 1--28. 
      
\bibitem{X20a}
       Y. Xu, 
       Fourier series in orthogonal polynomials on a cone of revolution,
       \textit{J. Fourier Anal. Appl.}, \textbf{26} (2020), Paper No. 36, 42 pp.
  
\bibitem{X20b}
       Y. Xu, 
       Orthogonal structure and orthogonal series in and on a double cone or a hyperboloid.
        \textit{Trans. Amer. Math. Soc.},  374 (2021),  3603--3657. 

\end{thebibliography}
\end{document}